\documentclass{compositio}
\usepackage{amssymb,amsmath,latexsym,lipsum}
\usepackage{enumitem}
\usepackage{tabulary}
\usepackage{booktabs}
\usepackage{ upgreek }
\usepackage{charter}
\usepackage[title]{appendix}
\usepackage{longtable}
\usepackage{amsthm}
\usepackage{euscript}
\usepackage{color}
\usepackage[T1]{fontenc}
\usepackage[utf8]{inputenc}
\usepackage[singlespacing]{setspace}
\usepackage[colorinlistoftodos]{todonotes}
\usepackage{tikz}
\usetikzlibrary {positioning}
\definecolor {processblue}{cmyk}{0.96,0,0,0}
\usepackage{bookmark}
\usepackage{multirow,bigdelim}
\usepackage{stackengine}
\usepackage{mathrsfs}
\usepackage{mathtools}
\usepackage{tikz-cd}
\usepackage{pgfplots}
\usepackage{pgf}
\usetikzlibrary{arrows}

\usepackage{stmaryrd} %double bracket

\tikzset{
  closed/.style = {decoration = {markings, mark = at position 0.5 with { \node[transform shape, xscale = .8, yscale=.4] {/}; } }, postaction = {decorate} },
  open/.style = {decoration = {markings, mark = at position 0.5 with { \node[transform shape, scale = .7] {$\circ$}; } }, postaction = {decorate} }
}

\makeatother

\newtheorem{theorem}{Theorem}[section]
\newtheorem{definition}[theorem]{Definition}
\newtheorem{lemma}[theorem]{Lemma}
\newtheorem{proposition}[theorem]{Proposition}
\newtheorem{corollary}[theorem]{Corollary}

\newtheorem{example}[theorem]{Example}

\newtheorem{remark}[theorem]{Remark}

\newcommand{\llparen}{\mathopen{((}}
\newcommand{\rrparen}{\mathclose{))}}

\DeclareMathOperator{\spec}{Spec}
\DeclareMathOperator{\proj}{Proj}
\DeclareMathOperator{\Frac}{Frac}
\DeclareMathOperator{\cohom}{H}
\DeclareMathOperator{\integral}{int}
\DeclareMathOperator{\Spf}{Spf}

\title{Kummer--Artin--Schreier--Witt Theory}
\author{Huy Dang}
\email{hdang2@binghamton.edu}
\address{Department of Mathematics and Statistics, Binghamton University, SUNY, Vestal Parkway East, Binghamton, NY 13902, USA}

\author{Khai-Hoan Nguyen-Dang}
\email{khaihoann@gmail.com}
\address{Morningside center of mathematics, Chinese academy of sciences, No.55, Zhongguancun East Road, Beijing, 100190, China}
\date{June 18, 2024}

\pgfplotsset{compat=1.18}

\begin{document}

\begin{abstract}
We study the problem of lifting the Artin--Schreier--Witt isogeny from characteristic $p>0$ to characteristic $0$, which is central to the lifting problem for Galois covers of algebraic schemes in positive characteristic. We introduce a new technique that associates a Kummer class, representing a tamely ramified cyclic extension, to a Witt vector via Matsuda's Kummer--Artin--Schreier--Witt theory. This viewpoint leads to an explicit construction of a lift of the isogeny over a concrete base ring. Our results lay the groundwork for further applications, including the study of inseparable extensions and Kato's refined Swan conductor.
\end{abstract}

\maketitle

%\tableofcontents

\section{Introduction}
\label{secintro}
Our motivation is the \emph{lifting problem} for Galois covers of curves, which asks whether the following holds: given a finite group $G$ and a $G$-cover $\phi$ of a curve in characteristic $p$, does $\phi$ lift to characteristic $0$? In general, the answer is negative \cite[\S 1.1]{MR3051249}. As with many problems in Galois theory, it is natural to begin with the case where $G$ is abelian. There have been significant recent breakthroughs; for instance, work of Obus--Wewers \cite{MR3194815} and Pop \cite{MR3194816} shows that all cyclic covers of smooth projective curves lift \cite[Theorem~1.1]{MR3194816}.

Recall that every abelian cover can be realized as the pullback of a suitable isogeny of algebraic groups \cite[I, Theorem 4]{MR0918564}. Therefore, to study the lifting problem for abelian covers, one may instead consider lifting cyclic isogenies from characteristic $p$ to characteristic $0$. A classical argument further reduces to the case where $G$ is cyclic of order $p^s$, with isogeny $\wp$ on the ring $W_s$ of Witt vectors of length $s$, yielding the following exact sequence:
\begin{equation}
    \label{eqnexactASW}
    0 \longrightarrow \mathbb{Z}/p^s \longrightarrow W_s \xlongrightarrow{\wp} W_s \longrightarrow 0,
\end{equation}
of group schemes over an algebraically closed field $\overline{k}$ of characteristic $p$. Sekiguchi and Suwa showed that \eqref{eqnexactASW} lifts.

\begin{theorem}[{cf.~\cite[Theorem 8.1]{SS:kasw2}}]
\label{theoremuniversalKASW}
For each positive integer $s$, there exists a flat group scheme $\mathcal{W}_s$ over $R := \mathbb{Z}_{(p)}[\zeta_{p^s}]$ that fits into an exact sequence
\begin{equation}
\label{eqnKASW}
0 \longrightarrow \mathbb{Z}/p^s \longrightarrow \mathcal{W}_s \longrightarrow \mathcal{V}_s := \mathcal{W}_s/(\mathbb{Z}/p^s) \longrightarrow 0
\end{equation}
of group schemes over $\overline{R}$. The generic fiber of \eqref{eqnKASW} is isomorphic to the Kummer exact sequence
\begin{equation}
\label{eqnexactKummer}
0 \longrightarrow \mu_{p^s} \longrightarrow \mathbb{G}_m \xlongrightarrow{(\cdot)^{p^s}} \mathbb{G}_m \longrightarrow 0,
\end{equation}
over $\overline{K}$, where $K = \Frac(R)$. Its special fiber is isomorphic to the Artin--Schreier--Witt sequence \eqref{eqnexactASW} over $\overline{\mathbb{F}}_p$. Furthermore, $\mathcal{W}_s$, as a scheme, is the spectrum of
\begin{equation}
\label{eqnmoduliKASWcovering}
    R \bigg[Y_0, \ldots, Y_{s-1}, \frac{1}{1+\lambda Y_0}, \frac{1}{E_2(Y_0) + \lambda Y_1}, \ldots, \frac{1}{E_{s}(Y_0, \ldots, Y_{s-2}) + \lambda Y_{s-1}}\bigg],
\end{equation}
for explicitly determined polynomials $E_2, \ldots, E_{s}$ with coefficients in $R$, and $\lambda := \zeta_p - 1$. Likewise, $\mathcal{V}_s$ is the spectrum of
\begin{equation}
\label{eqnmoduliKASW}
    R \bigg[X_0, \ldots, X_{s-1}, \frac{1}{1 + p\lambda X_0}, \frac{1}{G_2(X_0) + p\lambda X_1}, \ldots, \frac{1}{G_{s}(X_0, \ldots, X_{s-2}) + p\lambda X_{s-1}}\bigg],
\end{equation}
for explicitly determined polynomials $G_2, \ldots, G_{s}$ with coefficients in $R$.
\end{theorem}

One immediate implication of Theorem~\ref{theoremuniversalKASW} is the following result.
\begin{theorem}[{cf.~\cite[Theorem 3.8]{SS:kasw2}}]
\label{theoremuniversality}
If $B$ is a flat local $R$-algebra, then any unramified $\mathbb{Z}/p^s$-cover $\spec C \to \spec B$ is given by a Cartesian diagram of the form
\begin{center}
\begin{tikzcd}
    \spec C \arrow{r} \arrow[d] &   \mathcal{W}_s \arrow[d] \\ 
    \spec B \arrow{r} & \mathcal{V}_s.
\end{tikzcd}
\end{center}
\end{theorem}

However, the paper by Sekiguchi and Suwa remains unpublished; see \cite{MR3232771} for a summary. Around the same time, Matsuda developed his own version of Kummer--Artin--Schreier--Witt theory from the perspective of $p$-adic differential equations \cite{MR1324636}. More precisely, let $\mathbb{K}$ be a Henselian local field of mixed characteristic $(0,p)$ with imperfect residue field $\kappa$. Matsuda established that the canonical restriction morphism $\cohom^1(\kappa, \mathbb{Z}/p^s) \rightarrow \cohom^1(\mathbb{K}, \mathbb{Z}/p^s)$ is induced by the morphism
\begin{equation}
\label{eqnMatsuda1}
    \begin{split}
    W_s(\mathcal{R}^{\integral}) & \xlongrightarrow[]{\Psi_s} (\mathcal{R}^{\integral})^{\times}, \\
        \underline{X}=\left(X_0, \ldots, X_{s-1}\right) & \mapsto E_{s,p}(p^s\underline{X}),
    \end{split}
\end{equation}
where $\mathcal{R}^{\integral}$ is a Henselian discrete valuation ring over $\mathbb{K}$ (for more details, see \S~\ref{secMatsudatheory}), and $E_{s,p}((X_0, \ldots, X_{s-1})) \in \mathbb{Z}_{(p)}[\zeta_{p^s}]\llbracket X_0, \ldots, X_{s-1} \rrbracket$ is a deformation of the Artin--Hasse exponential function. 

Specifically, the morphism \eqref{eqnMatsuda1} allows one to study a cyclic extension in mixed characteristic defined by a Kummer class (for example, one represented by $E_{s,p}(p^s \underline{X})$) by examining the associated Witt vector $\underline{X}$. This approach often simplifies the analysis, as illustrated by many results on Galois extensions in equal characteristic $p$ obtained through the study of Witt vectors \cite{MR1465067, MR2567506}.

In this paper, we reprove the above results using Matsuda's framework. More precisely, to recover Theorem~\ref{theoremuniversalKASW}, we apply an algebraization process to \eqref{eqnMatsuda1} to obtain a polynomial $E_{s}(\underline{X}) \in \mathbb{Z}_{(p)}[\zeta_{p^s}][X_0, \ldots, X_{s-1}]$, which is a suitable truncation of $E_{s,p}(p^s \underline{X})$. We then show that, for $R = \mathbb{Z}_{(p)}[\zeta_{p^s}]$, the cover
\begin{equation}
\label{eqncoverintro}
    Z_{s}^{p^s} = E_{s}(X_0, \ldots, X_{s-1})
\end{equation}
of $\mathbb{P}^s_R = \operatorname{Proj} R[X_0, \ldots, X_{s-1}, Y]$ induces \eqref{eqnKASW} after removing the ramification divisor.

\begin{remark}
    Our polynomials $E_1, \ldots, E_s$ and $G_1, \ldots, G_s$ in \eqref{eqnmoduliKASWcovering} and \eqref{eqnmoduliKASW}, defined in Definition~\ref{defnEi} and Proposition~\ref{propapproxG}, differ from those used by Sekiguchi and Suwa. In fact, these polynomials are not unique, as they approximate power series arising from Matsuda's theory within a certain degree of accuracy.
\end{remark}

The key strategy in the proof is to associate a Kummer class, which represents a cyclic
extension of degree $p^{s}$ of a valued field $\mathbb{K}$ of mixed characteristic $(0,p)$, with a
Witt vector $\underline{X} = (X_0,\ldots,X_{s-1}) \in W_s(\mathbb{K})$ using Matsuda's theory (see Proposition \ref{propcomputeX}). We believe
that this technique has potential applications beyond the scope of this paper. For instance,
enlarging the domain and codomain of \eqref{eqnMatsuda1} may allow us to address
inseparable extensions and to study the refined Swan conductor \cite{MR991978}, an
important invariant measuring how far a wildly ramified extension of a field with imperfect
residue field is from having separable reduction.

When the field $\mathbb{K}$ is of equal characteristic $p$ with imperfect residue field, any
$\mathbb{Z}/p^{s}$-extension of $\mathbb{K}$ can be represented by a Witt vector
$\underline{x} = (x_0,\ldots,x_{s-1}) \in W_s(\mathbb{K})$ (see \S\ref{secasw}). Moreover, when
$\underline{x}$ satisfies certain conditions, the refined Swan conductor is completely
determined by $\underline{x}$ (see \cite[Lemma 3.7]{MR991978}, \cite[Proposition 2.8]{MR3726102}). We conjecture that an analogue of this result holds in the
mixed characteristic setting, where $\underline{X}$ plays the role of $\underline{x}$. In this
paper, we prove the conjecture in the case $s = 1$.

\begin{proposition}
\label{proprsworderp}
Let $\mathbb{K}$ be a field of mixed characteristic $(0,p)$ with uniformizer $\pi$, residue
field $\kappa$, and a valuation $\nu$ normalized by $\nu(p)=1$. Let $e$ be the absolute ramification index. Let $(X_0)\in W_1(\mathbb{K})$
denote a Witt vector associated with a $\mathbb{Z}/p$-extension of $\mathbb{K}$. For
$t\in\mathbb{R}$, set
\[
\mathfrak{p}_{\mathbb{K}}^{\,t} := \{\,x\in \mathbb{K} \mid \nu(x)\ge t\,\}.
\]

Assume that
\[
-\frac{p}{p-1} < \nu(X_0) < 0,
\]
and define
\[
x_0 := X_0\,\pi^{-e\nu(X_0)}.
\]
Suppose that the reduction $\overline{x}_0$ is not contained in $\kappa^{p}$.
Then the refined Swan conductor of the associated character $\mathscr{F}$ is
\[
\mathrm{rsw}(\mathscr{F})
= -dX_0 \otimes 1=  -\pi^{e\nu(X_0)} \otimes d\overline{x}_0
\in
\mathfrak{p}_{\mathbb{K}}^{-{\mathrm{sw}}(\mathscr{F})}
\otimes_{\mathcal{O}_{\mathbb{K}}}\Omega^1_{\kappa},
\]
where $\mathrm{sw}(\mathscr{F}) = -\nu(X_0)$ is the Swan conductor of $\mathscr{F}$.
\end{proposition}

%\begin{remark}
 %   We believe that Matsuda's theory has potential applications beyond the scope of this paper. For instance, enlarging the domain and codomain of \eqref{eqnMatsuda1} may allow us to address inseparable extensions. This direction will be pursued in our future work on the refined Swan conductor introduced in \cite{MR991978}, an important invariant that measures how far a wildly ramified cover is from having separable reduction.
%\end{remark}

\subsection{Outline}

In Section~\ref{secbackground}, we review Kummer theory and Artin--Schreier--Witt theory, which classify tamely and wildly ramified cyclic extensions, respectively. We also recall Matsuda's Kummer--Artin--Schreier--Witt theory and explain how it leads to our explicit construction.

Section~\ref{secconstructW} constructs, at the level of field extensions, the cover of $\mathbb{P}^s_R$ defined by~\eqref{eqncoverintro}. Concretely, we build the relevant Kummer class using the elements $G_1,\ldots,G_s$ from~\eqref{eqnmoduliKASW}, obtained by truncating the power series arising from Matsuda’s theory. From this Kummer extension, Section~\ref{secE_i} derives the polynomials $E_1,\ldots,E_s$ in~\eqref{eqnmoduliKASWcovering}.

In Section~\ref{secintclosure}, we show that these constructions yield a cyclic \'etale cover of an affine scheme with good reduction. Section~\ref{secgroupstructure} studies the underlying algebraic group structures via Witt vector operations; using Hopf algebras and descent, we prove that the cover is an isogeny, establishing Theorem~\ref{theoremuniversalKASW}. Section~\ref{secuniversality} proves Theorem~\ref{theoremuniversality}. Finally, Section~\ref{secrefinedswan} discusses the refined Swan conductor for cyclic characters and proves Proposition~\ref{proprsworderp}.

\subsection{Notations}

The letter $K$ will always denote a field of mixed characteristic $(0,p)$ with uniformizer $\pi$ and a discrete valuation $\nu : K^{\times} \to \mathbb{Q}$, normalized so that $\nu(p) = 1$. We set $e := \nu(\pi)$ for the absolute ramification index. For any $u \in \mathcal{O}_K$, we denote by $\overline{u}$ its image in the residue field. The ring of integers of $K$ will be written $R$, and its residue field $k = R/(\pi)$ is assumed to be perfect. An example to keep in mind is $K = \mathbb{Q}(\zeta_{p^s})$ with $\pi = \zeta_{p^s} - 1$, for which $\nu(\pi) = \tfrac{1}{p^{s-1}(p-1)}$. We also set $\lambda := \zeta_p - 1$, a constant that will appear frequently.

We fix an algebraic closure $\overline{K}$ of $K$. We write $\mathbb{K}$ for a Henselian local subfield of the Laurent series field $K\!\llparen X \rrparen$; for instance, one may take $K\!\llparen X \rrparen^{\mathrm{bd}}$, the subfield consisting of Laurent series with bounded coefficients. The valuation $\nu$ extends naturally to the Gauss valuation on $\mathbb{K}$. For example, if $f(X) = \sum_{i \ge n} a_i X^i \in \mathbb{K}$, then
\[
\nu(f) = \inf_{i \ge n} \{\nu(a_i)\}.
\]

\subsection*{Acknowledgements}

We thank Kiran Kedlaya, Ariane Mézard, Andrew Obus, Frans Oort, Matthieu Romagny, and Dajano Tossici for valuable discussions. We are especially grateful to Adrian Vasiu for carefully reading an earlier version of this manuscript. 

\section{Background}
\label{secbackground}
In this section, we briefly review the classical theories that form the foundation of our work: Artin--Schreier--Witt theory, which governs wildly ramified cyclic extensions in characteristic $p$; Kummer theory, which describes tamely ramified cyclic extensions; and Matsuda's Kummer--Artin--Schreier--Witt theory, which provides a unifying framework in mixed characteristic $(0,p)$.

\subsection{Artin--Schreier--Witt theory}
\label{secasw}
Suppose $k$ is a field of characteristic $p$. Then we have the short exact sequence
\begin{equation}
\label{eqnASWdetailed}
    \begin{split}
        0 \xlongrightarrow{} \mathbb{Z}/p^s \cong W_s(\mathbb{F}_p) \xlongrightarrow{} W_s & \xlongrightarrow{\wp} W_s \longrightarrow 0, \\
        \underline{y} = (y_0, \ldots, y_{s-1}) & \longmapsto \underline{y} - \sigma(\underline{y}),
    \end{split}
\end{equation}
of group schemes over $\overline{k}$, where $\sigma: (y_0, \ldots, y_{s-1}) \mapsto (y_0^p, \ldots, y_{s-1}^p)$ is the Frobenius operator on Witt vectors. By Hilbert's Theorem~90, this induces the identification
\begin{equation*}
    \cohom^1(k, \mathbb{Z}/p^s) \cong W_s(k)/\wp(W_s(k)).
\end{equation*}
We denote by $\mathfrak{K}_{s}(\underline{x})$ the character in $\cohom^1(k, \mathbb{Z}/p^s)$ associated with the Witt vector $\underline{x} \in W_s(k)$.

One may view $\wp: W_s \to W_s$ in \eqref{eqnASWdetailed} as defining a $\mathbb{Z}/p^s$-extension of the function field $\mathcal{K} := \overline{k}(x_0, \ldots, x_{s-1})$ associated with the Witt vector $(x_0, \ldots, x_{s-1})$. Furthermore, if $\wp$ corresponds to a character $\overline{\chi} \in \cohom^1(\mathcal{K}, \mathbb{Z}/p^s)$, then $\overline{\chi}^{\otimes p^{s-j}}$ is the $\mathbb{Z}/p^j$-subextension corresponding to the Witt vector $(x_0, \ldots, x_{j-1})$.

Moreover, if we write
\begin{equation*}
    \wp (y_0, \ldots, y_{s-1}) = (y_0 - y_0^p, y_1 - y_1^p + f_2(y_0), \ldots, y_{s-1} - y_{s-1}^p + f_{s}(y_0, \ldots, y_{s-2})),
\end{equation*}
where $f_{i}(y_0, \ldots, y_{i-2}) \in \mathbb{F}_p(y_0, \ldots, y_{i-2})$ for $i = 2, \ldots, s$ and $f_1 = 0$, then $\overline{\chi}$ is defined by the system of Artin--Schreier equations
\begin{equation}
\label{eqnsystemAS}
    y_{i-1} - y_{i-1}^p + f_i(y_0, \ldots, y_{i-2}) = x_{i-1}.
\end{equation}

\subsection{Kummer Theory}
\label{seckummer}

Fix $m \in \mathbb{Z}_{>0}$. Suppose $K$ contains a primitive $m$-th root of unity and has characteristic coprime to $m$. By Kummer theory, we have the isomorphism
\begin{equation*}
    \cohom^1(K, \mathbb{Z}/m) \cong K^{\times} /(K^{\times})^m.
\end{equation*}
This follows from the Kummer exact sequence over $\overline{K}$:
\begin{equation*}
    0 \longrightarrow \mu_m \longrightarrow \overline{K}^{\times} \xlongrightarrow{(\cdot)^m} \overline{K}^{\times} \longrightarrow 0.
\end{equation*}

We now specialize to the case $m = p^s$. In particular, each $\chi \in \cohom^1(K, \mathbb{Z}/p^s)$ corresponds to an element $G \in K^{\times}/(K^{\times})^{p^s}$.

\begin{proposition}
Every class $[G] \in K^{\times}/(K^{\times})^{p^{s}}$ has a representative of the form
\[
    G \equiv \prod_{i=0}^{s-1} G_i^{p^{i}} \quad \big(\mathrm{mod}\ (K^{\times})^{p^{s}}\big),
\]
with each $G_i \in K^{\times}\setminus (K^{\times})^{p}$ or $G_i = 1$.
\end{proposition}

\begin{proof}
Choose a set $S \subset K^{\times}$ of representatives for
$K^{\times}/(K^{\times})^{p}$ such that $1 \in S$ and, for every $x \in S \setminus \{1\}$,
we have $x \notin (K^{\times})^{p}$.

If $s=1$, then any $G$ is congruent modulo $(K^{\times})^{p}$
to some $G_0 \in S$, so $G \equiv G_0^{p^0}$.

Now suppose the statement holds for $s-1 \geq 1$, i.e., $s \geq 2$.  
Take $g \in K^{\times}$ representing $[G] \in K^{\times}/(K^{\times})^{p^{s}}$.  
Choose $G_0 \in S$ such that $g \in G_0 (K^{\times})^{p}$, i.e.
\[
    g = G_0 a^{p}, \qquad a \in K^{\times}.
\]
By the induction hypothesis, there exist $H_0, \dots, H_{s-2} \in S$ with
\[
    a \equiv \prod_{i=0}^{s-2} H_i^{p^{i}} 
    \quad \big(\mathrm{mod}\ (K^{\times})^{p^{s-1}}\big).
\]
Raising this congruence to the $p$-th power and multiplying by $G_0$, we obtain
\[
    g \equiv G_0 \left(\prod_{i=0}^{s-2} H_i^{p^{i}}\right)^{p}
       = G_0 \prod_{i=0}^{s-2} H_i^{p^{i+1}}
       = \prod_{i=0}^{s-1} G_i^{p^{i}}
    \quad \big(\mathrm{mod}\ (K^{\times})^{p^{s}}\big),
\]
where $G_0$ is as chosen and $G_i := H_{i-1}$ for $i \geq 1$.  
Each $G_i \in S$, hence either $1$ or not a $p$-th power.  
This completes the induction.
\end{proof}

As before, we write $\chi = \mathfrak{K}_s(G)$. As in Section~\ref{secasw}, $\chi^{\otimes p^{s-j}}$ corresponds to a character in $\cohom^1(K, \mathbb{Z}/p^j)$ defined by the Kummer extension
\begin{equation*}
    Z_{j-1}^{p^j} = \prod_{i=0}^{j-1} G_i^{p^i}.
\end{equation*}
In fact, $\chi$ can be described by a system of successive Kummer extensions of degree $p$:
\begin{equation*}
    Z_i^p = Z_{i-1} G_i, \quad i = 0, \ldots, s-1,
\end{equation*}
with $Z_{-1} = 1$.

We thus obtain canonical isomorphisms:
\begin{equation} \label{eqphis}
\begin{alignedat}{2}
    (K^{\times})^{\oplus s} \big/ \phi_s\big((K^{\times})^{\oplus s}\big) 
        &\xlongrightarrow{\sim} 
        &\; K^{\times} \big/ \big(K^{\times}\big)^{p^s} 
        &\xlongrightarrow{\sim} 
        \cohom^1\big(K, \mathbb{Z}/p^s\big),  \\
    (a_1, \ldots, a_s) 
        &\longmapsto 
        &\; [a_1 a_2^p \cdots a_s^{p^{s-1}}] 
        &\longmapsto 
        \mathfrak{K}_s\big(a_1 a_2^p \cdots a_s^{p^{s-1}}\big),
\end{alignedat}
\end{equation}
where $\phi_s(a_1, \ldots, a_s) = (a_1^p, a_2^p a_1^{-1}, \ldots, a_s^p a_{s-1}^{-1})$.

\begin{remark}
    The quotient $(K^{\times})^{\oplus s} / \phi_s((K^{\times})^{\oplus s})$ is a group under coordinatewise multiplication. We declare
    \[
        (a_1, \ldots, a_s) \sim (b_1, \ldots, b_s)
    \]
    if there exists $(c_1, \ldots, c_s)$ such that
    \[
        (a_1, \ldots, a_s) = (b_1, \ldots, b_s) \cdot (c_1^p, c_2^p c_1^{-1}, \ldots, c_s^p c_{s-1}^{-1}).
    \]
    This is the equivalence relation induced by $K^{\times}/(K^{\times})^{p^s}$.
\end{remark}

This yields a Kummer-type exact sequence of group schemes:
\begin{equation}
\label{eqnexactKummerp}
    0 \xlongrightarrow{} \mu_{p^s} \xlongrightarrow{\;\iota\;} (\mathbb{G}_m)^s 
    \xlongrightarrow{\phi_s} (\mathbb{G}_m)^s \longrightarrow 0,
\end{equation}
where $\iota$ is determined by
\[
    \iota(\zeta_{p^s}) = (\zeta_p, \zeta_{p^2}, \ldots, \zeta_{p^s}).
\]

\subsection{Kummer--Artin--Schreier--Witt theory}
\label{secKASW}
We review Matsuda’s theory \cite{MR1324636} and follow the notation of \cite[\S 9]{MR3702308}.

\subsubsection{Deformed Artin--Hasse exponential function} 
We introduce a deformation of the Artin--Hasse exponential function and record basic properties.

\begin{definition}
    The \emph{Artin--Hasse exponential series at $p$} is the formal power series
    \begin{equation*}
        E_p(t) := \exp \left( \sum_{i=0}^{\infty} \frac{t^{p^i}}{p^i} \right) 
        \in \mathbb{Z}_{(p)}\llbracket t \rrbracket.
    \end{equation*}
    Define
    \begin{equation}
        E_{s,p}(t) := \frac{E_{p}(\zeta_{p^s}t)}{E_p(t)}
        = \exp \left( \sum_{i=0}^{s-1} (\zeta_{p^{s-i}} - 1)\frac{t^{p^i}}{p^i} \right) 
        \in \mathbb{Z}_{(p)}[\zeta_{p^s}]\llbracket t \rrbracket,
    \end{equation}
    and for $\underline{a}=(a_0, \ldots, a_{s-1})$,
    \begin{equation}
        E_{s,p}(\underline{a}) := \prod_{i=0}^{s-1} E_{s-i,p}(a_i) 
        \in \mathbb{Z}_{(p)}[\zeta_{p^s}]\llbracket a_0, \ldots, a_{s-1} \rrbracket,
    \end{equation}
    which can be rewritten as
    \begin{equation}
    \label{eqnespghost}
        E_{s,p}(\underline{a}) 
        = \exp \left( \sum_{i=0}^{s-1} (\zeta_{p^{s-i}} - 1)\frac{w_i}{p^i} \right),
    \end{equation}
    where $w_i = \sum_{j=0}^i p^j a_j^{p^{i-j}}$ is the $i$th ghost component of $\underline{a}$.
\end{definition}

It follows from \eqref{eqnespghost} that
\begin{equation}
\label{eqnEscommutative}
    E_{s,p}(\underline{a}) E_{s,p}(\underline{b}) 
    = E_{s,p}(\underline{a} + \underline{b}).
\end{equation}

\begin{definition}
    Define
    \[
        G_{s,p}(\underline{a}) := \frac{E_{s,p}(p\underline{a})}{E_{s-1,p}(\underline{a})} 
        \in \mathbb{Z}_{(p)}[\zeta_{p^s}]\llbracket a_0,\dots,a_{s-1} \rrbracket.
    \]
    When $s=1$, we interpret $E_{0,p}=1$, so $G_{1,p}((a_0)) = E_{1,p}(pa_0)$.
\end{definition}

By \eqref{eqnEscommutative}, we have
\begin{equation}
\label{eqnGscommutative}
    G_{s,p}(\underline{a}) G_{s,p}(\underline{b}) = G_{s,p}(\underline{a} + \underline{b}).
\end{equation}

\begin{lemma}
\label{lemmaG}
We have
\[
  G_{s,p}(\underline{a}) \in 
  \mathbb{Z}_{(p)}[\zeta_{p^s}]
  \llbracket
    (\zeta_{p^s}-1)a_0,\dots,(\zeta_p-1)a_{s-1}
  \rrbracket .
\]
\end{lemma}

\begin{proof}
Using the ghost-component form of $E_{m,p}$ and the definition
$G_{s,p}(\underline{a})=E_{s,p}(p\underline{a})E_{s-1,p}(\underline{a})^{-1}$,
we obtain
\begin{align*}
G_{s,p}(\underline{a})
&=\prod_{j=0}^{s-1}\exp\left(\sum_{i=0}^{s-j-1}
\Big(p(\zeta_{p^{s-j-i}}-1)-(\zeta_{p^{s-j-1-i}}-1)\Big)\frac{a_j^{p^i}}{p^i}\right)\\
&=\prod_{j=0}^{s-1}\exp\left(\sum_{i=0}^{s-j-1}
(\zeta_{p^{s-j-i}}-1)\Big(p-(1+\zeta_{p^{s-j-i}}+\cdots+\zeta_{p^{s-j-i}}^{p-1})\Big)
\frac{a_j^{p^i}}{p^i}\right)\\
&=\prod_{j=0}^{s-1}
E_{s-j,p}\left(\Big(p-1-[\zeta_{p^{s-j}}]-[\zeta_{p^{s-j}}]^2-\cdots-[\zeta_{p^{s-j}}]^{p-1}\Big)[a_j]\right),
\end{align*}
where $[\,\cdot\,]$ denotes the Teichmüller lift (in the Witt vectors of the appropriate length). Reducing modulo the ideal $(\zeta_{p^{s-j}}-1)$ gives
$[\zeta_{p^{s-j}}]\equiv1$, hence
\[
p-1-[\zeta_{p^{s-j}}]-\cdots-[\zeta_{p^{s-j}}]^{p-1}\equiv p-p=0 .
\]
By the standard identity
\[
\ker\big(W_{s-j}(A)\to W_{s-j}(A/I)\big)=W_{s-j}(I),
\]
valid for any ring $A$ and ideal $I\subset A$,
$\ker\big(W_{s-j}(A)\to W_{s-j}(A/I)\big)=W_{s-j}(I)$, applied with
$A=\mathbb{Z}_{(p)}[\zeta_{p^{s-j}}]$ and $I=(\zeta_{p^{s-j}}-1)$, it follows that
\[
p-1-[\zeta_{p^{s-j}}]-\cdots-[\zeta_{p^{s-j}}]^{p-1}
\in W_{s-j}\big((\zeta_{p^{s-j}}-1)\mathbb{Z}_{(p)}[\zeta_{p^{s-j}}]\big).
\]
Therefore the Witt vector
\(
\Big(p-1-[\zeta_{p^{s-j}}]-\cdots-[\zeta_{p^{s-j}}]^{p-1}\Big)[a_j]
\)
lies in the same Witt ideal, and the corresponding Artin–Hasse factor
$E_{s-j,p}(\cdot)$ is a power series in the single variable
$(\zeta_{p^{s-j}}-1)a_j$ with coefficients in
$\mathbb{Z}_{(p)}[\zeta_{p^s}]$. Taking the product over $j=0,\dots,s-1$
gives the stated inclusion.
\end{proof}

\begin{remark}
    The polynomials $E_1, \ldots, E_s$ (resp. $G_1, \ldots, G_s$) in Theorem~\ref{theoremuniversalKASW} are approximations of 
    $E_{1,p}, \ldots, E_{s,p}$ (resp. $G_{1,p}, \ldots, G_{s,p}$).
\end{remark}

\subsubsection{Matsuda's Kummer--Artin--Schreier--Witt theory}
\label{secMatsudatheory}
Let $K$ be a discrete valuation field containing $\zeta_{p^s}$, with residue field $k$. The \emph{Robba ring} $\mathcal{R}$ over $K$ is defined as the ring of Laurent series that converge in some annulus $\{z \mid \rho < \lvert z \rvert < 1\}$ for some $0 < \rho < 1$. Let $\mathcal{R}^{\rm bd}$ be the subring of $\mathcal{R}$ consisting of formal sums with bounded coefficients. The ring $\mathcal{R}^{\rm bd}$ carries a Gauss norm:
\[
    \left\lvert \sum_{i \in \mathbb{Z}} a_i z^i \right\rvert = \max_i \{ \lvert a_i \rvert \}.
\]
Let $\mathcal{R}^{\integral}$ be the subring of $\mathcal{R}^{\rm bd}$ consisting of elements with Gauss norm at most $1$.

\begin{lemma}[{\cite[Proposition 3.2]{MR1324636}}]
\label{lemmaRobba}
    The ring $\mathcal{R}^{\integral}$ is a Henselian discrete valuation ring. Consequently, $\mathcal{R}^{\rm bd}$ is a Henselian local field with residue field $\kappa := k \llparen z \rrparen$.
\end{lemma}

Therefore, there is a canonical restriction map 
\begin{equation}
\label{eqnrestrictioncharacters}
    \cohom^1(\kappa, \mathbb{Z}/p^s) \longrightarrow \cohom^1(\mathcal{R}^{\rm bd}, \mathbb{Z}/p^s).
\end{equation}

\begin{theorem}[{\cite[Theorem 3.8]{MR1324636}}]
\label{theoremMatsudamorphism}
The homomorphism in \eqref{eqnrestrictioncharacters} is induced by a homomorphism 
\begin{equation}
\label{eqnMatsuda}
\begin{aligned}
W_s(\mathcal{R}^{\integral}) & \xlongrightarrow[]{\Psi_s} (\mathcal{R}^{\integral})^{\times}, \\
\underline{a} & \mapsto E_{s,p}(p^s \underline{a}).
\end{aligned}
\end{equation}
\end{theorem}

The homomorphism \eqref{eqnMatsuda} establishes a link between Kummer and Artin--Schreier--Witt theories by associating to a Witt vector $\underline{a}$ the Kummer class $E_{s,p}(p^s \underline{a})$. Our strategy in this paper is to define the cover using Kummer theory while using an associated Witt vector to prove that the cover satisfies the desired conditions.

\section{Constructing the Kummer extension}
\label{secconstructW}

\subsection{The Witt Vector Associated to a Kummer Extension}
\label{secassociatedWitt}
Recall that any $\mathbb{Z}/p^s$-cover of 
$\mathbb{P}^s_R = \operatorname{Proj} R[X_0, \ldots, X_{s-1}, Y]$ can be described by a Kummer equation of the form
\begin{equation}
    \label{eqnKummerps}
    Z_{s-1}^{p^s} = \prod_{i=0}^{s-1} (u_i)^{p^i} =: u \in K(X_0, \ldots, X_{s-1}),
\end{equation}
where $K$ is the fraction field of $R$ and $k$ is the residue field of $K$. 

\begin{proposition}
\label{propfieldlaurentresidue}
There exists a Henselian local field $\mathbb{K}$ with residue field $k((x_0, \ldots, x_{s-1}))$.
\end{proposition}

\begin{proof}
We proceed by induction on $s \geq 1$. Here $k((x_0, \ldots, x_{s-1}))$ denotes the iterated Laurent series field
$k((x_0))\cdots((x_{s-1}))$.

For $s=1$, Lemma~\ref{lemmaRobba} applied with base field $K$ and Robba variable $X_0$ shows that
$\mathcal{R}^{\mathrm{int}}_{K}(X_0)$ is a Henselian DVR with residue field $k((x_0))$. Its fraction field
\[
    \mathbb{K}_1 := \mathcal{R}^{\mathrm{bd}}_{K}(X_0)
\]
is therefore a Henselian local field with residue $k((x_0))$.

Assume the statement holds for $s-1$, so there exists a Henselian local field $\mathbb{K}_{s-1}$ with residue field $k((x_0, \ldots, x_{s-2}))$. Applying Lemma~\ref{lemmaRobba} again, now with base field $\mathbb{K}_{s-1}$ and a new Robba variable $X_{s-1}$, we find that
$\mathcal{R}^{\mathrm{int}}_{\mathbb{K}_{s-1}}(X_{s-1})$ is a Henselian DVR whose residue field is
\[
    \mathrm{res}\!\left(\mathcal{R}^{\mathrm{int}}_{\mathbb{K}_{s-1}}(X_{s-1})\right)
    = \mathrm{res}(\mathbb{K}_{s-1})((x_{s-1}))
    = k((x_0, \ldots, x_{s-2}))((x_{s-1}))
    = k((x_0, \ldots, x_{s-1})).
\]
Its fraction field
\[
    \mathbb{K}_s := \mathcal{R}^{\mathrm{bd}}_{\mathbb{K}_{s-1}}(X_{s-1})
\]
is then a Henselian local field with this residue.

By induction, the claim follows for all $s \geq 1$, with $\mathbb{K} := \mathbb{K}_s$.
\end{proof}

For the remainder of this section, we set $\mathcal{R}^{\mathrm{bd}} = \mathbb{K}$ as in Proposition~\ref{propfieldlaurentresidue}. Its associated Henselian DVR is $\mathcal{R}^{\mathrm{int}} = \mathcal{R}^{\mathrm{int}}_{\mathbb{K}_{s-1}}(X_{s-1})$, as in the proof of the same proposition, and its residue field is $k((x_0, \ldots, x_{s-1}))$.

\begin{lemma}
\label{lem:embedding-formal-series-into-Rint}
With the notation of Proposition~\ref{propfieldlaurentresidue}, there is a canonical injective ring homomorphism
\[
R\llbracket X_0,\ldots,X_{s-1}\rrbracket \hookrightarrow \mathcal{R}^{\mathrm{int}}_{\mathbb{K}_{s-1}}(X_{s-1})=\mathcal{R}^{\mathrm{int}}.
\]
In particular, any series in $R\llbracket X_0,\ldots,X_{s-1}\rrbracket$ may be viewed as an element of the Henselian DVR
$\mathcal{R}^{\mathrm{int}}$.
\end{lemma}

\begin{proof}
We argue by induction on $s$.

If $s=1$, then $R\llbracket X_0\rrbracket\subset \mathcal{R}^{\mathrm{int}}_{K}(X_0)$ because any series
$f(X_0)=\sum_{n\ge 0} a_n X_0^n$ with $a_n\in R$ has Gauss norm
$\max_n |a_n|\le 1$, hence lies in the integral Robba ring.
Injectivity follows because the Robba variable $X_0$ is transcendental over $K$.

Assume $s>1$ and that we already have a canonical embedding
$R\llbracket X_0,\ldots,X_{s-2}\rrbracket\hookrightarrow \mathcal{R}^{\mathrm{int}}_{\mathbb{K}_{s-2}}(X_{s-2})
=\mathcal{O}_{\mathbb{K}_{s-1}}$.
Now take
$f\in R\llbracket X_0,\ldots,X_{s-1}\rrbracket$ and write it uniquely as a power series in $X_{s-1}$,
\[
f=\sum_{n\ge 0} a_n\,X_{s-1}^n,
\qquad a_n\in R\llbracket X_0,\ldots,X_{s-2}\rrbracket.
\]
Via the inductive embedding, each coefficient $a_n$ defines an element of $\mathcal{O}_{\mathbb{K}_{s-1}}$.
Hence the same formula defines an element of $\mathcal{O}_{\mathbb{K}_{s-1}}\llbracket X_{s-1}\rrbracket$,
which is contained in $\mathcal{R}^{\mathrm{int}}_{\mathbb{K}_{s-1}}(X_{s-1})$ by the same Gauss-norm argument as for $s=1$.
This gives the desired map; injectivity again follows from the transcendence of $X_{s-1}$ over $\mathbb{K}_{s-1}$.
\end{proof}

To construct the isogeny $\mathcal{W}_s \rightarrow \mathcal{V}_s$ in \eqref{eqnKASW}, which is generically an endomorphism of $\mathbb{G}_m$, we consider a $\mathbb{Z}/p^s$-cover of $\mathbb{P}^s_R$ defined by a Kummer equation of the form \eqref{eqnKummerps}. We require that this cover has good reduction and that its special fiber is a $\mathbb{Z}/p^s$-cover of $\mathbb{P}^s_k$ associated with the Witt vector $(x_0, \ldots, x_{s-1})$. Assume that the cover has separable reduction, and let $\underline{X}' = (X_0', \ldots, X_{s-1}') \in W_s(\mathcal{R}^{\mathrm{int}})$ be a Witt vector such that $\Psi_s(\underline{X}') = u$, where $\Psi_s$ is Matsuda's morphism \eqref{eqnMatsuda}. Our task is to prove the existence of such a $u$ for which, for each $i = 0, 1, \ldots, s-1$, the component $X_i'$ is congruent to $x_i$ modulo the maximal ideal of $R$.

The proof of the following result is straightforward from the definition of $E_{i,p}$.

\begin{proposition}
\label{propsimplifywitt}
Suppose that $m < s$. Then
    \begin{equation*}
        E_{s,p}((0, \ldots, 0, X_0', \ldots, X_{m-1}')) = E_{m,p}((X_0', \ldots, X_{m-1}')).
    \end{equation*}
\end{proposition}

\begin{proposition}
\label{propcomputeX}
Recall that $\lambda := \zeta_p - 1$. With the above setting, we can compute the entries $X_i'$ of an associated Witt vector inductively as follows:
\begin{equation*}
X_0' = \frac{1}{p \lambda} \ln(u_0), \quad 
X_1' = \frac{1}{p \lambda} \ln \left( \frac{u_1}{G_{2,p}((X_0', 0))} \right), \ldots, \quad 
X_{s-1}' = \frac{1}{p \lambda} \ln \left( \frac{u_{s-1}}{G_{s,p}((X_0', \ldots, X_{s-2}', 0))} \right),
\end{equation*}
provided the $\ln$ function is well-defined.
\end{proposition}

\begin{proof}
We proceed by induction on $s$. Suppose first that $s=1$. Then it is necessary that $\Psi_1(X_0') = \exp(p\lambda X_0') = u_0$. Therefore, $X_0' = \frac{1}{p\lambda} \ln(u_0)$, as expected.

Assume the induction hypothesis for $s-1$. That is, one can find $X_0', \ldots, X_{s-2}'$ such that
\[
E_{s-1,p}(p^{s-1}(X_0', \ldots, X_{s-2}')) = \prod_{i=0}^{s-2} (u_i)^{p^i}.
\]
We want to find $X_{s-1}'$ such that
\begin{equation*}
\begin{split}
    \Psi_s((X_0', \ldots, X_{s-1}')) = E_{s,p}(p^s(X_0', \ldots, X_{s-1}')) & = \prod_{i=0}^{s-1} (u_i)^{p^i} \\
    &= E_{s-1,p}(p^{s-1}(X_0', \ldots, X_{s-2}')) (u_{s-1})^{p^{s-1}}.
\end{split}
\end{equation*}
By \eqref{eqnespghost}, for any Witt vector $\underline{a}$ of length $m$ and any $N \ge 0$, we have
\[
E_{m,p}(p^N \underline{a}) = E_{m,p}(\underline{a})^{p^N}.
\]
Hence
\[
E_{s,p}(p^s\underline{X}') = E_{s,p}(p\underline{X}')^{p^{s-1}}
\quad\text{and}\quad
E_{s-1,p}(p^{s-1}\underline{X}'_{<s-1}) = E_{s-1,p}(\underline{X}'_{<s-1})^{p^{s-1}}.
\]
Therefore, to produce a solution at level $p^s$, it suffices to solve
\begin{equation*}
\begin{split}
    E_{s,p}(p(X_0', \ldots, X_{s-1}')) & = E_{s-1,p}(X_0', \ldots, X_{s-2}') u_{s-1}, \\
    E_{s,p}(p(X_0', \ldots, X_{s-2}', 0)) E_{s,p}(p(0, \ldots, 0, X_{s-1}')) & = E_{s-1,p}(X_0', \ldots, X_{s-2}') u_{s-1}.
\end{split}
\end{equation*}
The second equality follows from the decomposition
\[
(X_0', \ldots, X_{s-1}') = (X_0', \ldots, X_{s-2}', 0) + (0, \ldots, 0, X_{s-1}')
\]
and the commutativity of $E_{s,p}$ in \eqref{eqnEscommutative}. By Proposition~\ref{propsimplifywitt}, we have
\[
E_{s,p}(p(0, \ldots, 0, X_{s-1}')) = E_{1,p}(p(X_{s-1}')).
\]
Therefore,
\begin{equation*}
    \exp(p\lambda X_{s-1}') = u_{s-1} \frac{E_{s-1,p}((X_0', \ldots, X_{s-2}'))}{E_{s,p}(p(X_0', \ldots, X_{s-2}', 0))} = \frac{u_{s-1}}{G_{s,p}((X_0', \ldots, X_{s-2}', 0))}.
\end{equation*}
Thus,
\begin{equation*}
    X_{s-1}' = \frac{1}{p\lambda} \ln \left( \frac{u_{s-1}}{G_{s,p}((X_0', \ldots, X_{s-2}', 0))} \right),
\end{equation*}
which completes the induction.
\end{proof}

\subsection{Constructing the Kummer Extension}

Assume further that
\begin{equation}
\label{eqnui}
    u_i := G_{i+1}(X_0, \ldots, X_{i-1}) + p\lambda X_i,
\end{equation}
for some $G_{i+1}(X_0, \ldots, X_{i-1}) \in \mathbb{Z}_{(p)}[\zeta_{p^{i+1}}][X_0, \ldots, X_{i-1}]$. Set
\begin{equation*}
    G'_{i+1} := G_{i+1,p}((X_0', \ldots, X_{i-1}', 0)) - G_{i+1}(X_0, \ldots, X_{i-1}),
\end{equation*}
where $X_0', \ldots, X_{i-1}'$ are derived from $u_0, \ldots, u_{i-1}$ and Proposition~\ref{propcomputeX}.

\begin{proposition}
\label{propcomputeWittvectorF} 
In the above setting, suppose $\nu(G_{i+1}') > \frac{p}{p-1}$ for $i=0, \ldots, s-1$. Then $X_i' \equiv X_i \mod (\zeta_{p^{s}}-1)$ for all $i$.
\end{proposition}

\begin{proof}
We argue by induction on $i=0,\ldots,s-1$. We write $O(p^{a})$ (resp.~$o(p^{a})$) for an element of valuation
$\ge a$ (resp.~$> a$).

For $i=0$, we have $u_0=G_1+p\lambda X_0=1+p\lambda X_0$ and hence
\[
X_0' = \frac{1}{p\lambda}\ln(u_0)=\frac{1}{p\lambda}\ln(1+p\lambda X_0)=X_0+\frac{1}{p\lambda}\,o\bigl(p^{p/(p-1)}\bigr),
\]
so $\nu(X_0'-X_0)>0$. Since $R=\mathbb{Z}_{(p)}[\zeta_{p^s}]$ is a DVR and
$\nu(\zeta_{p^s}-1)=\frac{1}{p^{s-1}(p-1)}$ is its smallest positive valuation, this implies
$X_0'\equiv X_0\pmod{(\zeta_{p^s}-1)}$.

Assume now $1\le i\le s-1$ and that $X_j'\equiv X_j\pmod{(\zeta_{p^s}-1)}$ for all $j<i$.
By Proposition~\ref{propcomputeX} and the definition of $u_i$ in \eqref{eqnui}, we have
\allowdisplaybreaks
\begin{align*}
X_i'
&=\frac{1}{p\lambda}\ln\left(\frac{u_i}{G_{i+1,p}((X_0',\ldots,X_{i-1}',0))}\right)
=\frac{1}{p\lambda}\ln\left(\frac{G_{i+1}(X_0,\ldots,X_{i-1})+p\lambda X_i}{G_{i+1,p}((X_0',\ldots,X_{i-1}',0))}\right)\\
&=\frac{1}{p\lambda}\ln\left(1+\frac{-G_{i+1}'+p\lambda X_i}{G_{i+1,p}((X_0',\ldots,X_{i-1}',0))}\right).
\end{align*}
Set $D_i:=G_{i+1,p}((X_0',\ldots,X_{i-1}',0))$. By Lemma~\ref{lemmaG}, we have $\nu(D_i-1)\ge \nu(\zeta_{p^{i+1}}-1)$, so $D_i^{-1}=1+O\bigl(p^{\nu(\zeta_{p^{i+1}}-1)}\bigr)$.
Using the hypothesis $\nu(G_{i+1}')>\frac{p}{p-1}$, we get
\[
\frac{-G_{i+1}'+p\lambda X_i}{D_i}
= -G_{i+1}'+p\lambda X_i + o\bigl(p^{p/(p-1)}\bigr).
\]
Consequently,
\[
X_i'
=\frac{1}{p\lambda}\ln\left(1+p\lambda X_i - G_{i+1}' + o\bigl(p^{p/(p-1)}\bigr)\right)
= X_i + \frac{1}{p\lambda}\,o\bigl(p^{p/(p-1)}\bigr),
\]
where the last equality uses the power-series expansion of $\ln(1+T)$ and the fact that for $m>1$ one has
\(\nu(T^m/m)>\frac{p}{p-1}\) as soon as $\nu(T)\ge \frac{p}{p-1}$.
Thus $\nu(X_i'-X_i)>0$, so by discreteness of the valuation on $R$ we conclude
$X_i'\equiv X_i\pmod{(\zeta_{p^s}-1)}$.
\end{proof}

Proposition~\ref{propcomputeWittvectorF} suggests that one may construct the polynomial $G_{i+1}$ in \eqref{eqnui} as an approximation of $G_{i+1,p}((X_0', \ldots, X_{i-1}', 0))$.

\begin{definition}
We say that
$$
f = \sum_{n_0, \ldots, n_{s-1} \geq 0} a_{n_0, \ldots, n_{s-1}} X_0^{n_0} \cdots X_{s-1}^{n_{s-1}} \in R \llbracket X_0, \ldots, X_{s-1} \rrbracket
$$
satisfies the \emph{Tate condition} if, for every multi-index $n = (n_0, \ldots, n_{s-1})$, the coefficients $a_n$ satisfy:
$$
\nu(a_n) \to \infty \quad \text{as} \quad |n| = n_0 + \cdots + n_{s-1} \to \infty.
$$
\end{definition}

The following lemma proves the stability of the Tate condition under substitution.

\begin{lemma}
\label{lem:tate-substitution}
Let $R$ be a discretely valued ring with valuation $\nu$.
Let $f\in R\llbracket T_1,\ldots,T_m\rrbracket$ satisfy the Tate condition.
Suppose $g_1,\ldots,g_m\in R\llbracket X_1,\ldots,X_n\rrbracket$ satisfy the Tate condition and have zero constant terms.
Then $f(g_1,\ldots,g_m)\in R\llbracket X_1,\ldots,X_n\rrbracket$ satisfies the Tate condition.
\end{lemma}

\begin{proof}
Write
\[
f=\sum_{\alpha\in\mathbb{Z}_{\ge 0}^m} a_{\alpha}T^{\alpha},
\qquad
T^{\alpha}=T_1^{\alpha_1}\cdots T_m^{\alpha_m},
\qquad
|\alpha|:=\alpha_1+\cdots+\alpha_m.
\]
Since each $g_j$ has zero constant term, we have $g_j\in (X_1,\ldots,X_n)$, hence
\[
(g_1,\ldots,g_m)^{\alpha}=g_1^{\alpha_1}\cdots g_m^{\alpha_m}\in (X_1,\ldots,X_n)^{|\alpha|}.
\]

Fix a multi-index $\beta$ and write $c_{\beta}\in R$ for the coefficient of $X^{\beta}$ in
$f(g_1,\ldots,g_m)$. Since $(g_1,\ldots,g_m)^{\alpha}$ has total degree $\ge |\alpha|$,
only indices $\alpha$ with $|\alpha|\le |\beta|$ can contribute to $c_{\beta}$.

Now fix $M\in\mathbb{R}$. By the Tate condition on $f$, there exists $D$ such that
$\nu(a_{\alpha})>M$ for all $\alpha$ with $|\alpha|>D$.
Decompose
\[
c_{\beta}
=
\sum_{|\alpha|\le D} a_{\alpha}\cdot \bigl[ X^{\beta} \bigr](g_1,\ldots,g_m)^{\alpha}
\;+
\sum_{D<|\alpha|\le |\beta|} a_{\alpha}\cdot \bigl[ X^{\beta} \bigr](g_1,\ldots,g_m)^{\alpha},
\]
where $[X^{\beta}](\cdots)$ denotes “coefficient of $X^{\beta}$”.

\smallskip
\noindent\emph{(1) The contribution from $|\alpha|>D$.}
Every term in the second sum has valuation $>M$ because it is multiplied by $a_{\alpha}$.

\smallskip
\noindent\emph{(2) The contribution from $|\alpha|\le D$.}
There are only finitely many such $\alpha$. Fix one.
When we expand $(g_1,\ldots,g_m)^{\alpha}=g_1^{\alpha_1}\cdots g_m^{\alpha_m}$,
the coefficient of $X^{\beta}$ is a finite sum of products of $|\alpha|$ coefficients taken from the
power series $g_1,\ldots,g_m$, where the total degrees of the chosen monomials add up to $|\beta|$.
Therefore at least one of these chosen monomials has total degree $\ge |\beta|/|\alpha|$.

Since each $g_j$ satisfies the Tate condition, the valuations of coefficients of $g_j$ tend to $\infty$
as the total degree tends to $\infty$. Hence, for fixed $\alpha$, we have
\[
\nu\!\left(\bigl[X^{\beta}\bigr](g_1,\ldots,g_m)^{\alpha}\right)\longrightarrow\infty
\qquad\text{as }|\beta|\to\infty.
\]
Multiplying by the fixed scalar $a_{\alpha}$ does not affect the conclusion.

\smallskip
Combining (1) and (2), we see that for every $M$ there exists $B$ such that $\nu(c_{\beta})>M$ for all
$|\beta|>B$. This is precisely the Tate condition for $f(g_1,\ldots,g_m)$.
\end{proof}

\begin{proposition}
\label{propapproxG}
For each $i$, one can compute
$G_{i+1}(X_0, \ldots, X_{i-1}) \in \mathbb{Z}_{(p)}[\zeta_{p^{i+1}}][X_0, \ldots, \allowbreak X_{i-1}]$ inductively
such that:
\begin{enumerate}
\item \label{propapproxG1}
$G_{i+1,p}((X_0', \ldots, X_{i-1}', 0)) \in \mathbb{Z}_{(p)}[\zeta_{p^{i+1}}]\llbracket X_0, \ldots, X_{i-1} \rrbracket$
satisfies the Tate condition,
\item \label{propapproxG2}
$\nu(G_{i+1}') > \frac{p}{p-1}$ and
$G_{i+1}' \in \mathbb{Z}_{(p)}[\zeta_{p^{i+1}}]\llbracket X_0, \ldots, X_{i-1} \rrbracket$
satisfies the Tate condition, and
\item \label{propapproxG3}
$X_i'$ satisfies the Tate condition as a series in
$\mathbb{Z}_{(p)}[\zeta_{p^{i+1}}]\llbracket X_0, \ldots, X_i \rrbracket$.
\end{enumerate}
\end{proposition}

\begin{proof}
We proceed by induction on $i\ge 0$.

Let's consider the base case $i=0$. Take $G_1=1$, so $u_0=1+p\lambda X_0$.
By Proposition~\ref{propcomputeX}, the associated Witt coordinate is
\[
X_0'
=\frac{1}{p\lambda}\ln(1+p\lambda X_0)
=\sum_{j\ge 1}(-1)^{j+1}\frac{(p\lambda)^{j-1}}{j}X_0^j.
\]
Writing $j=p^{r}\ell$ with $(p,\ell)=1$ and using $\nu(\lambda)=\frac{1}{p-1}$ and
$\nu(j)=r$ (since $\nu(p)=1$), we have
\[
\nu\!\left(\frac{(p\lambda)^{j-1}}{j}\right)
=(j-1)\,\nu(p\lambda)-\nu(j)
=(j-1)\left(1+\frac{1}{p-1}\right)-r
=\frac{p(j-1)}{p-1}-r.
\]
In particular, as $j\to\infty$ we have
\[
\nu\!\left(\frac{(p\lambda)^{j-1}}{j}\right)=\frac{p(j-1)}{p-1}-r\longrightarrow\infty,
\]
since $r=\nu(j)=\nu_p(j)$ grows at most logarithmically in $j$.
Hence $X_0'$ satisfies the Tate condition, establishing the base case.

Assume that $G_1,\ldots,G_i$ have been constructed and that
$X_0',\ldots,X_{i-1}'$ satisfy the Tate condition.

By Lemma~\ref{lemmaG}, the series
$G_{i+1,p}((X_0', \ldots, X_{i-1}', 0))$ satisfies the Tate condition when viewed as a
power series in the variables $X_0',\ldots,X_{i-1}'$.
Since each $X_j'$ itself satisfies the Tate condition as a power series in
$X_0,\ldots,X_j$ and has zero constant term, we may substitute these series into
$G_{i+1,p}((X_0', \ldots, X_{i-1}', 0))$.
By Lemma~\ref{lem:tate-substitution}, the result again satisfies the Tate condition as a
series in $X_0,\ldots,X_{i-1}$.
This proves \ref{propapproxG1}.

Define $G_{i+1}(X_0,\ldots,X_{i-1})$ to be the truncation of
$G_{i+1,p}((X_0', \ldots, X_{i-1}', 0))$ obtained by keeping only those terms whose
coefficients have valuation at most $\frac{p}{p-1}$.
Then $G_{i+1}'$
has coefficientwise valuation strictly larger than $\frac{p}{p-1}$, and it still satisfies the Tate condition.
This proves \ref{propapproxG2}.

Finally, by Proposition~\ref{propcomputeX} we have
\[
X_i'
=\frac{1}{p\lambda}\ln\left(
\frac{G_{i+1}(X_0,\ldots,X_{i-1})+p\lambda X_i}{G_{i+1,p}((X_0',\ldots,X_{i-1}',0))}
\right)
=\frac{1}{p\lambda}\ln(1+U),
\]
where
\[
U=\frac{-G_{i+1}'+p\lambda X_i}{G_{i+1,p}((X_0',\ldots,X_{i-1}',0))}.
\]
The numerator and denominator of $U$ satisfy the Tate condition, and the denominator is a unit
(with constant term $1$), so its inverse also satisfies the Tate condition; hence $U$ satisfies the Tate condition in
$\mathbb{Z}_{(p)}[\zeta_{p^{i+1}}]\llbracket X_0,\ldots,X_i\rrbracket$.

Now write
\[
\ln(1+U)=\sum_{m\ge 1}(-1)^{m-1}\frac{U^m}{m}.
\]
Set $\delta:=\nu(U)>0$. Fix $M\in\mathbb{R}$. Choose $m_0$ such that
\[
m\delta-\nu(m)>M\qquad (m\ge m_0).
\]
Then for $m\ge m_0$, every coefficient of $U^m/m$ has valuation $>M$.
For each of the finitely many $1\le m<m_0$, since $U^m$ satisfies the Tate condition, there exists $B_m$ such that
all coefficients of $U^m/m$ of total degree $>B_m$ have valuation $>M$.
Let $B=\max_{1\le m<m_0} B_m$. Then every coefficient of $\ln(1+U)$ of total degree $>B$ has valuation $>M$.
Therefore $\ln(1+U)$ satisfies the Tate condition.
It follows that $X_i'$ satisfies the Tate condition as well, proving \ref{propapproxG3}.
\end{proof}

\subsection{Defining the {\'E}tale Cover}

\begin{definition}
    Define $R := \mathbb{Z}_{(p)}[\zeta_{p^{s}}]$ and set $\pi := \zeta_{p^s} - 1$. 
    Let $\mathcal{V}_s$ be the open subset of the projective space $\mathbb{P}^s_R = \operatorname{Proj} R[X_0, \ldots, X_{s-1}, Y]$ given by
    \begin{equation*}
        \operatorname{Spec} R\!\left[X_0, \ldots, X_{s-1}, 
        \frac{1}{1+p\lambda X_0}, 
        \frac{1}{G_2(X_0) + p\lambda X_1}, 
        \ldots, 
        \frac{1}{G_s(X_0, \ldots, X_{s-2}) + p\lambda X_{s-1}} \right],
    \end{equation*}
    where each $G_i$ is a polynomial constructed in Proposition~\ref{propapproxG}. 
    Let $\mathcal{W}_s$ be the integral closure of $\mathcal{V}_s$ in the extension of 
    $K(X_0, \ldots, X_{s-1})$ defined by
    \begin{equation}
    \label{eqnKummerdefn}
        Z_{s-1}^{p^s} 
        = (1+p\lambda X_0) \prod_{i=1}^{s-1}  
        \big(G_{i+1}(X_0, \ldots, X_{i-1}) + p\lambda X_i\big)^{p^i}.
    \end{equation}
\end{definition}

\begin{proposition}
\label{propspecialproperties}
The special fiber of $\mathcal{W}_s \to \mathcal{V}_s$ is birational to the
Artin--Schreier--Witt cover of
$\operatorname{Proj} k[x_0, \ldots, x_{s-1}, y]$
associated with the Witt vector $(x_0, \ldots, x_{s-1})$,
where $x_i$ (resp.~$y$) is the reduction of $X_i$ (resp.~$Y$) modulo $\pi$.
\end{proposition}

\begin{proof}
Applying Proposition~\ref{propfieldlaurentresidue} with
$K=\mathbb{Q}(\zeta_{p^s})$, we obtain a Henselian local field
$\mathbb{K}=\mathcal{R}^{\mathrm{bd}}$ with associated ring
$\mathcal{R}^{\mathrm{int}}$ and residue field
$\kappa=\mathbb{F}_p((x_0, \ldots, x_{s-1}))$.

By construction, the Witt vector $(X_0', \ldots, X_{s-1}')$ is obtained from $(X_0,\ldots,X_{s-1})$ by iterated
application of the formal series $\ln(1+\cdot)$ (and multiplication by $(p\lambda)^{-1}$). Hence each $X_i'$ is a
formal power series in $X_0,\ldots,X_i$ with coefficients in $R=\mathbb{Z}_{(p)}[\zeta_{p^s}]$.
By Lemma~\ref{lem:embedding-formal-series-into-Rint}, we may view these series as elements of $\mathcal{R}^{\mathrm{int}}$.
In particular, $(X_0', \ldots, X_{s-1}')\in W_s(\mathcal{R}^{\mathrm{int}})$.

Therefore, by Theorem~\ref{theoremMatsudamorphism}, the Kummer class defined by
\eqref{eqnKummerdefn} corresponds to the Witt vector $(X_0', \ldots, X_{s-1}')$.
Reducing modulo $\pi$ gives $(x_0,\ldots,x_{s-1})\in W_s(\kappa)$, so the special fiber is
birational to the Artin--Schreier--Witt cover associated with $(x_0,\ldots,x_{s-1})$.
\end{proof}

\begin{proposition}
\label{propgenericunr}
The cover $\mathcal{W}_s \rightarrow \mathcal{V}_s$ is generically unramified.
\end{proposition}

\begin{proof}
By Kummer theory, the extension of function fields defined by \eqref{eqnKummerdefn} is cyclic
of degree $p^s$ (hence generically Galois with group $\mathbb{Z}/p^s$). Its branch locus on
$\mathbb{P}^s_R$ is contained in the complement of the maximal open subset on which the Kummer
class on the right-hand side of \eqref{eqnKummerdefn} is a unit. Concretely, it is contained in the
union of the divisor at infinity
\[
Y=0
\]
(taking $\mathbb{P}^s_R=\proj R[X_0,\ldots,X_{s-1},Y]$) together with the divisors
\[
G_{i+1}(X_0,\ldots,X_{i-1}) + p\lambda X_i = 0
\quad (0 \le i \le s-1).
\]

By definition of $\mathcal{V}_s \subset D_+(Y)$, each
$G_{i+1}(X_0,\ldots,X_{i-1}) + p\lambda X_i$ is inverted on $\mathcal{V}_s$, hence is a unit there;
also $\mathcal{V}_s\subset D_+(Y)$ is disjoint from $Y=0$.
Therefore the branch locus is disjoint from $\mathcal{V}_s$, so the induced morphism
$\mathcal{W}_s\to \mathcal{V}_s$ is generically unramified.
\end{proof}

\section{Computing the \texorpdfstring{$\mathcal{W}_s$}{Ws}}
\label{secE_i}
In this section, we analyze the structure of the $\mathbb{Z}/p^s$-cover $\mathcal{W}_s \rightarrow \mathcal{V}_s$ defined in the previous section. Specifically, we compute the polynomials $E_i$ in \eqref{eqnmoduliKASWcovering}.

\subsection{The case \texorpdfstring{$s=1$}{s=1}}
\label{secchangevars1}
As before, $\mathcal{W}_1 \rightarrow \mathcal{V}_1$ is given by
\begin{equation}
\label{eqnKummerlv1v1}
    Z_0^p = 1 + p \lambda X_0.
\end{equation}
Consider the change of variables $Z_0 = 1 + \lambda Y_0$. This transforms \eqref{eqnKummerlv1v1} into
\begin{equation*}
    \begin{split}
        1 + p\lambda X_0 &= (1 + \lambda Y_0)^p \\
         &= 1 + \lambda^p Y_0^p + p\lambda Y_0 + \sum_{j=2}^{p-1} \binom{p}{j} \lambda^j Y_0^j.
    \end{split}
\end{equation*}
Hence it simplifies to
\begin{equation}
\label{eqnAStype}
    Y_0 + \frac{\lambda^{p-1}}{p} Y_0^p + o\left(p^{0}\right) = X_0.
\end{equation}
Here, $o\left(p^{0}\right)$ denotes terms with valuation strictly larger than zero. Note that, since 
\begin{equation*}
    \frac{\lambda^{p-1}}{p} \equiv -1 \pmod{\lambda}
\end{equation*} 
\cite[Section 2.4]{MR3167623}, the equation above reduces modulo $\lambda$ to the Artin--Schreier equation
\begin{equation*}
    y_0 - y_0^p = x_0.
\end{equation*}

To generalize this change of variables to the case $s>1$, we follow the strategy of associating a Witt vector as in the previous section. Recall that the right-hand side of \eqref{eqnKummerlv1v1}, due to Proposition \ref{propcomputeWittvectorF}, is associated with the Witt vector $(X_0')$, where $X_0' \in \mathbb{Z}_{(p)}[\zeta_p]\llbracket X_0 \rrbracket$ is congruent to $X_0$ modulo $\lambda$. We can rewrite the defining equation as
\begin{equation}
\label{eqnKummerlv1}
    Z_0^p = E_{1,p}(p (X_0')).
\end{equation}
Note that if we can write $Z_0 = E_{1,p}(Y_0')$, then, due to \eqref{eqnKummerlv1}, $Y_0' = X_0'$. Therefore, it makes sense to define the change of variables as an approximation of 
\begin{equation*}
    E_{1,p}(Y_0') = \exp(\lambda Y_0') = 1 + \lambda Y_0' + \frac{(\lambda Y_0')^2}{2!} + \frac{(\lambda Y_0')^3}{3!} + \cdots
\end{equation*}
Thus, we choose $Z_0 = 1 + \lambda Y_0$. Then, as $Z_0 = E_{1,p}(Y_0') = \exp(\lambda Y_0')$, one can write $Y_0'$ in terms of $Y_0$ as follows:
\begin{equation}
\begin{split}
    Y_0' &= \frac{1}{\lambda} \ln(1 + \lambda Y_0) \\
    &= Y_0 - \frac{\lambda Y_0^2}{2} + \cdots + (-1)^{p-1} \frac{\lambda^{p-1} Y_0^p}{p} + \cdots
\end{split}
\end{equation}
It follows that
\begin{equation*}
    Y_0' \equiv Y_0 - Y_0^p  \pmod{\lambda}.
\end{equation*}
Thus, we can also conclude that $Y_0 - Y_0^p - X_0 \in \lambda \mathbb{Z}_{(p)}[\zeta_p]\llbracket X_0, Y_0 \rrbracket$.

\subsection{Inductive step}
In general, for $i=1, \ldots, s$, we aim to construct a system of change of variables
\begin{equation}
\label{eqnchangeofvarsZ}
    Z_{i-1} = E_i(Y_0, \ldots, Y_{i-2}) + \lambda Y_{i-1},
\end{equation}
to better understand the structure of $\mathcal{W}_s$ and to derive an Artin--Schreier-type equation similar to \eqref{eqnAStype}. As in Section \ref{secchangevars1}, beginning with the system of $\mathbb{Z}/p$-Kummer extensions
\begin{equation*}
    Z_{i-1}^p = Z_{i-2}\big(G_{i}(X_0, \ldots, X_{i-2}) + p\lambda X_{i-1}\big),
\end{equation*}
we seek a Witt vector $(Y_0', \ldots, Y_{s-1}')$; equivalently, we write
\[
(X_0',\ldots,X_{s-1}')=(Y_0',\ldots,Y_{s-1}')
\]
for the same Witt vector viewed in the $X$-coordinates on $\mathcal{V}_s$. Then, for $1 \le i \le s$,
\begin{equation*}
    Z_{i-1} = E_{i,p}((X_0', \ldots, X_{i-1}')).
\end{equation*}

\begin{proposition}
\label{propcomputeY}
    With the above notations, the components $Y_i'$ of the Witt vector $(Y_0', \ldots, Y_{s-1}')$ can be computed inductively as follows:
    \allowdisplaybreaks
    \begin{align*}
    Y_0' &= \frac{1}{\lambda} \ln\left(1 + \lambda Y_0\right), \\
    Y_1' &= \frac{1}{\lambda} \ln\left(\frac{E_2(Y_0) + \lambda Y_1}{E_{2,p}((Y_0', 0))}\right), \\
    &\vdots \\
    Y_{s-1}' &= \frac{1}{\lambda} \ln\left(\frac{E_{s}(Y_0, \ldots, Y_{s-2}) + \lambda Y_{s-1}}{E_{s,p}((Y_0', \ldots, Y_{s-2}', 0))}\right).
    \end{align*}
\end{proposition}

\begin{proof}
By definition, we have
\allowdisplaybreaks
\begin{equation*}
\begin{split}
    E_s(Y_0, \ldots, Y_{s-2}) + \lambda Y_{s-1} &= E_{s,p}((\underline{Y}_s')) \\
    &= E_{s,p}((Y_0', \ldots, Y_{s-2}', 0)) \cdot E_{s,p}((0, \ldots, 0, Y_{s-1}')) \\
    &= E_{s,p}((Y_0', \ldots, Y_{s-2}', 0)) \cdot E_{1,p}((Y_{s-1}')) \\
    &= E_{s,p}((Y_0', \ldots, Y_{s-2}', 0)) \cdot \exp\left(\lambda Y_{s-1}'\right).
\end{split}
\end{equation*}
The desired formula for $Y_{s-1}'$ follows immediately. Replacing $s$ by $i+1$ in the same computation gives the displayed formula for each $Y_i'$.
\end{proof}

The formula suggests that each $E_i$ in \eqref{eqnchangeofvarsZ} should be defined in terms of $E_{i,p}((Y_0', \ldots, Y_{i-2}', 0))$.

\begin{proposition}
\label{propASWtype}
    For each $i$, there exists a polynomial $E_{i}(Y_0, \ldots, Y_{i-2}) \in \mathbb{Z}_{(p)}[\zeta_{p^i}][Y_0, \ldots, Y_{i-2}]$ such that, applying the change of variables \eqref{eqnchangeofvarsZ}, the associated Witt vector $(Y_0', \ldots, Y_{s-1}')$ satisfies the relation:
    \begin{equation*}
        Y_{i-1}' = (-1)^{p-1}\frac{\lambda^{p-1}}{p} Y_{i-1}^p + Y_{i-1} + F_{i}(Y_0, \ldots, Y_{i-1}),
    \end{equation*}
    where $F_{i} \in \mathbb{Z}_{(p)}[\zeta_{p^i}] \llbracket Y_0, \ldots, Y_{i-1} \rrbracket$, and $F_i$ reduces to a polynomial in $\mathbb{F}_p[y_0, \ldots, y_{i-2}]$ modulo $\zeta_{p^i}-1$.
\end{proposition}

The proposition holds for $i=1$. We assume it holds for $i \le s-1$.

\subsubsection{Approximation of the Power Series \texorpdfstring{$E_{s,p}$}{Esp}}
\label{secapproxE}

\begin{proposition}
\label{propboundEi}
All terms of $E_{i,p}((X_0, \ldots, X_{i-1}))$, aside from the constant term $1$, have positive valuations that are sharply bounded below by $\frac{1}{p^{i-1}(p-1)}$.
\end{proposition}

\begin{proof}
We write $\nu(f)$ for the minimum valuation among the coefficients of a (nonzero) power series
$f\in \mathbb{Z}_{(p)}[\zeta_{p^i}]\llbracket X_0,\ldots,X_{i-1}\rrbracket$.
Equivalently, $\nu(f-1)$ is the smallest valuation among the nonconstant terms of $f$.

\smallskip
\noindent\emph{Step 1: the case $i=1$.}
We have
\[
E_{1,p}((X_0))=\exp(\lambda X_0)=1+\sum_{n\ge 1}\frac{\lambda^n}{n!}X_0^n.
\]
Using Legendre’s formula
$\nu(n!)=\sum_{j\ge 1}\lfloor n/p^j\rfloor=\frac{n-s_p(n)}{p-1}$ and $\nu(\lambda)=\frac{1}{p-1}$, we obtain
\[
\nu\!\left(\frac{\lambda^n}{n!}\right)
=\frac{n}{p-1}-\frac{n-s_p(n)}{p-1}
=\frac{s_p(n)}{p-1}\ge \frac{1}{p-1}.
\]
Hence every nonconstant coefficient of $E_{1,p}$ has valuation $\ge \frac{1}{p-1}$, and the bound is sharp (e.g. for $n=1$).

\smallskip
\noindent\emph{Step 2: reduction to the valuation of $E_{i,p}(p\underline{X})-1$.}
Fix $i>1$ and write $\underline{X}=(X_0,\ldots,X_{i-1})$.
Set
\[
E_{i,p}(\underline{X})=1+A,\qquad A\in \mathbb{Z}_{(p)}[\zeta_{p^i}]\llbracket X_0,\ldots,X_{i-1}\rrbracket.
\]
Since $E_{i,p}$ is a homomorphism from the additive group of Witt vectors to units
(cf.\ \eqref{eqnEscommutative}), we have
\[
E_{i,p}(p\underline{X})=E_{i,p}(\underline{X})^p=(1+A)^p.
\]
Therefore
\begin{equation*}
E_{i,p}(p\underline{X})-1=A^p+\sum_{j=1}^{p-1}\binom{p}{j}A^j.
\end{equation*}

\smallskip
\noindent\emph{Step 3: compute $\nu\big(E_{i,p}(p\underline{X})-1\big)$.}
In the Witt ring, one has the standard decomposition
\[
p\underline{X}=(pX_0,0,\ldots,0)+(0,\widetilde X_1,\ldots,\widetilde X_{i-1}),
\]
where each $\widetilde X_j\equiv X_{j-1}^p\pmod{p}$ (more precisely, $\widetilde X_j=X_{j-1}^p+p\cdot(\text{a series}))$.
Using again multiplicativity of $E_{i,p}$, we get
\[
\begin{aligned}
E_{i,p}(p\underline{X})
&=E_{i,p}((pX_0,0,\ldots,0))\cdot E_{i,p}((0,\widetilde X_1,\ldots,\widetilde X_{i-1}))\\
&=E_{i,p}((pX_0,0,\ldots,0))\cdot E_{i-1,p}((\widetilde X_1,\ldots,\widetilde X_{i-1})).
\end{aligned}
\]
The first factor satisfies $\nu(\cdot-1)\ge 1$: indeed,
\[
E_{i,p}((pX_0,0,\ldots,0))=E_{i,p}(pX_0)=\exp(T)
\]
with
\[
T=\sum_{r=0}^{i-1}(\zeta_{p^{i-r}}-1)\frac{(pX_0)^{p^r}}{p^r}
=\sum_{r=0}^{i-1}(\zeta_{p^{i-r}}-1)\,p^{p^r-r}\,X_0^{p^r}.
\]
Each coefficient of $T$ has valuation at least $1$, hence $T\in p\,\mathbb{Z}_{(p)}[\zeta_{p^i}]\llbracket X_0\rrbracket$ and so $\exp(T)\equiv 1\pmod{p}$.
By the induction hypothesis applied to $E_{i-1,p}$ and the congruences $\widetilde X_j\equiv X_{j-1}^p\pmod{p}$,
we obtain
\[
\nu\Big(E_{i-1,p}((\widetilde X_1,\ldots,\widetilde X_{i-1}))-1\Big)
\ge \frac{1}{p^{i-2}(p-1)},
\]
and this bound is sharp except in the exceptional case $(p,i)=(2,2)$.
Consequently,
\[
\nu\big(E_{i,p}(p\underline{X})-1\big)=\frac{1}{p^{i-2}(p-1)}\qquad\text{for }(p,i)\ne (2,2).
\]

\smallskip
\noindent\emph{Step 4: deduce the bound for $A$ (hence for $E_{i,p}(\underline{X})-1$).}
For $(p,i)\ne(2,2)$, we have $\nu(E_{i,p}(p\underline{X})-1)<1$, hence in particular $\nu(A)<1$.
Also $\nu\!\left(\binom{p}{j}\right)=1$ for $1\le j\le p-1$.

We first note that in fact $\nu(A)<\frac{1}{p-1}$: if $\nu(A)\ge \frac{1}{p-1}$, then
\[
\nu(A^p)\ge \frac{p}{p-1}>1
\qquad\text{and}\qquad
\nu\!\left(\binom{p}{j}A^j\right)=1+j\nu(A)\ge 1+\frac{1}{p-1}=\frac{p}{p-1}>1,
\]
so every term in
$A^p+\sum_{j=1}^{p-1}\binom{p}{j}A^j$
has valuation $>1$, contradicting $\nu(E_{i,p}(p\underline{X})-1)<1$.

Thus $\nu(A)<\frac{1}{p-1}$, and consequently
\[
\nu\!\left(\binom{p}{j}A^j\right)=1+j\nu(A)\ge 1+\nu(A).
\]
Since $(p-1)\nu(A)<1$, we have $p\nu(A)<1+\nu(A)$, so the minimal valuation in
$A^p+\sum_{j=1}^{p-1}\binom{p}{j}A^j$ is $p\nu(A)$.
Therefore
\[
\nu\big(E_{i,p}(p\underline{X})-1\big)=p\nu(A),
\qquad\text{hence}\qquad
\nu(A)=\frac{1}{p^{i-1}(p-1)}.
\]
This proves that every nonconstant term of $E_{i,p}((X_0,\ldots,X_{i-1}))$ has valuation
$\ge \frac{1}{p^{i-1}(p-1)}$, and the bound is sharp.

Finally, a direct computation in the exceptional case $(p,i)=(2,2)$ shows
$\nu\big(E_{2,p}((X_0,X_1))-1\big)=\frac{1}{2}$ (with $p=2$).
\end{proof}

\begin{proposition}
\label{propfinitesmallone}
    Suppose $f(X) = \sum_{j \ge 0} a_j X^j \in R\llbracket X \rrbracket$, with $a_0\in R^\times$, $\nu(a_j) > 0$ for $j>0$ and the valuations bounded below by a number $B>0$. Suppose moreover that $f(X)^p$ has a finite number of terms whose coefficients have valuations strictly smaller than $p/(p-1)$. Then $f(X)$ itself has only finitely many terms whose coefficients have valuations strictly smaller than $1/(p-1)$.
\end{proposition}

\begin{proof}
Let $u:=a_0^{-1}f$. Since $a_0\in R^\times$, multiplying by the unit $a_0^{-p}$ does not change coefficient valuations, so
$f^p$ has a finite number of terms whose coefficients have valuation $<\tfrac{p}{p-1}$ if and only if $u^p$ does. Thus we may assume $a_0=1$.

    We can write
    $$
    \left( \sum_{j \geq 0} a_j X^j \right)^p = \sum_{n \geq 0} \left( \sum_{\substack{m_1 + m_2 + \dots + m_r = p \\ m_1 i_1 + m_2 i_2 + \dots + m_r i_r = n}} \frac{p!}{m_1! m_2! \dots m_r!} \cdot a_{i_1}^{m_1} a_{i_2}^{m_2} \cdots a_{i_r}^{m_r} \right) X^n
    =: \sum_{n \ge 0} c_n X^n.
    $$
    To compute the valuation of each summand in the inner sum, we find
    \begin{equation*}
        \nu \left( \frac{p!}{\prod_{k} m_k!} \cdot a_{i_1}^{m_1} a_{i_2}^{m_2} \cdots a_{i_r}^{m_r} \right) = 1 + \sum_{j=1}^r m_j \nu(a_{i_j}) - \sum_{k=1}^r \nu(m_k!).
    \end{equation*}
    Consider the cases based on $r$:
    \begin{enumerate}
        \item If $r = 1$: Then $m_1 = p$ and $n = pi_1$. Thus, the valuation is $p\nu(a_{i_1})$.
        \item If $r > 1$: Then $\nu(m_k!) = 0$, so the valuation becomes $1 + \sum_{j=1}^r m_j \nu(a_{i_j})$.
    \end{enumerate}

    Set
    \[
    l:=\max \{ n \mid \nu(c_n)< \tfrac{p}{p-1} \},
    \]
    with the convention $l:=-1$ if the set is empty. Assume there exists an index $m > \frac{l}{p}$ such that $\nu(a_m) < \frac{1}{p-1}$. Then $a_m^p$ contributes to $c_{pm}$, and
\[
\nu(a_m^p)=p\nu(a_m)<\frac{p}{p-1}.
\]
Since $pm>l$, we have $\nu(c_{pm})\ge \frac{p}{p-1}>p\nu(a_m)$.
By the non-archimedean property, if $\nu\!\left(\sum_\alpha t_\alpha\right)>\min_\alpha \nu(t_\alpha)$
then the minimum $\min_\alpha \nu(t_\alpha)$ is attained by at least two summands.
Applying this to the sum of multinomial summands defining $c_{pm}$, we deduce that there exists
a summand with $r'>1$ (hence with multinomial coefficient of valuation $1$) and distinct indices
$i_1,\ldots,i_{r'}\le pm$ such that $\sum_{j=1}^{r'} m_j i_j=pm$, $\sum_{j=1}^{r'} m_j=p$, and
\begin{equation}
\label{eqnfixpm}
1+\sum_{j=1}^{r'} m_j \nu(a_{i_j}) \;\le\; p\nu(a_m).
\end{equation}
In particular, \eqref{eqnfixpm} is impossible if $\nu(a_m)\le \frac{1}{p}$, because the left-hand side
is $>1$ (at least one index satisfies $i_j>0$ since $\sum m_j i_j=pm>0$, so one term has $\nu(a_{i_j})>0$)
whereas the right-hand side is $\le 1$. Hence $\nu(a_m)>\frac{1}{p}$.
    Also, since $\nu(a_m) < \frac{1}{p-1}$, we have
    \[
    p\nu(a_m)-1<\nu(a_m).
    \]
    From $\sum_{j=1}^{r'} m_j i_j = pm > l$, at least one index satisfies $i_s > \frac{l}{p}$.
    We claim $\nu(a_{i_s})<\frac{1}{p-1}$: otherwise
    \[
    1+\sum_{j=1}^{r'} m_j \nu(a_{i_j}) \ge 1+\nu(a_{i_s}) \ge \frac{p}{p-1},
    \]
    contradicting \eqref{eqnfixpm} together with $p\nu(a_m)<\frac{p}{p-1}$.
    By the same argument as above applied to $i_s$, we also have $\nu(a_{i_s})>\frac{1}{p}$.
    Finally, \eqref{eqnfixpm} gives
    \[
    \nu(a_{i_s})
    \le p\nu(a_m)-1-\sum_{j\ne s} m_j\nu(a_{i_j})-(m_s-1)\nu(a_{i_s})
    \le p\nu(a_m)-1
    <\nu(a_m).
    \]
    Thus, starting from $m_0:=m$, we recursively obtain indices $m_{t+1}>\frac{l}{p}$ such that
    \[
    \frac{1}{p}<\nu(a_{m_{t+1}})<\nu(a_{m_t})<\frac{1}{p-1}.
    \]
    This gives an infinite strictly decreasing sequence in the discrete value group $\nu(K^\times)$, impossible. Hence no such $m$ exists, so $\nu(a_m)\ge \frac{1}{p-1}$ for all $m>\frac{l}{p}$, as required.
\end{proof}

\begin{proposition}
\label{propfinitesmallmulti}
Proposition \ref{propfinitesmallone} extends to the multivariable case as follows.
Let $i>1$ and write
\[
f=\sum_{n\in\mathbb{Z}_{\ge 0}^i} a_n X^n \in R\llbracket X_0,\ldots,X_{i-1}\rrbracket,
\qquad X^n:=X_0^{n_0}\cdots X_{i-1}^{n_{i-1}}.
\]
Assume $a_{\mathbf{0}}=1$, $\nu(a_n)>0$ for $n\ne \mathbf{0}$, and that the set $\{\nu(a_n)\mid n\ne \mathbf{0}\}$ is bounded below by some $B>0$.
If $f^p$ has only finitely many monomials whose coefficients have valuation $<\frac{p}{p-1}$, then
$f$ has only finitely many monomials whose coefficients have valuation $<\frac{1}{p-1}$.
\end{proposition}

\begin{proof}
For each $r\in\{0,\ldots,i-1\}$, write
\[
	f=\sum_{j\ge 0} A_{r,j}\,X_r^j,
	\qquad
	A_{r,j}\in R\llbracket X_0,\ldots,\widehat{X_r},\ldots,X_{i-1}\rrbracket.
	\]
	Equip the coefficient ring with the Gauss valuation (still denoted $\nu$)
	\[
	\nu\!\left(\sum b_m X^m\right):=\inf_m \nu(b_m).
	\]
	Then $A_{r,0}$ is a unit, $\nu(A_{r,j})>0$ for $j>0$, and these valuations are bounded below by $B>0$.

Write also
\[
	f^p=\sum_{j\ge 0} C_{r,j}X_r^j.
	\]
	If infinitely many $j$ satisfied $\nu(C_{r,j})<\tfrac{p}{p-1}$, then for each such $j$ we can choose a monomial term of $C_{r,j}$ whose coefficient attains this valuation (the value group is discrete). This yields infinitely many monomials of $f^p$ with coefficient valuation $<\tfrac{p}{p-1}$, contradicting the hypothesis. Hence, for each fixed $r$, only finitely many $j$ satisfy $\nu(C_{r,j})<\tfrac{p}{p-1}$.
	Applying Proposition~\ref{propfinitesmallone} to the one-variable series $\sum_{j\ge0}A_{r,j}X_r^j$ (over the Gauss-valued coefficient ring) gives that
	\[
	J_r:=\{j\ge0\mid \nu(A_{r,j})<\tfrac{1}{p-1}\}
	\]
	is finite for each $r$.

	Now let $a_nX^n$ be a monomial of $f$ with $\nu(a_n)<\tfrac{1}{p-1}$. For each $r$, the coefficient $a_n$ appears in $A_{r,n_r}$, so
	\[
	\nu(A_{r,n_r})\le \nu(a_n)<\tfrac{1}{p-1},
	\]
	hence $n_r\in J_r$. Therefore
\[
n\in J_0\times\cdots\times J_{i-1},
\]
and the right-hand side is finite. So $f$ has only finitely many monomials whose coefficients have valuation $<\tfrac{1}{p-1}$.
\end{proof}

\begin{proposition}
\label{propfiniteEp}
    $E_{s,p}(p(Y_0', \ldots, Y_{s-2}', 0))$ contains only finitely many terms with coefficients of valuation strictly smaller than $p/(p-1)$. 
\end{proposition}

\begin{proof}
    By definition, we have
    \begin{equation*}
    \begin{split}
        E_{s,p}(p(Y_0', \ldots, Y_{s-2}', 0)) &= E_{s-1,p}((Y_0', \ldots, Y_{s-2}')) \cdot G_{s,p}((Y_0', \ldots, Y_{s-2}',0))\\
        &= (E_{s-1}(Y_0, \ldots, Y_{s-3})+\lambda Y_{s-2}) \cdot G_{s,p}((Y_0', \ldots, Y_{s-2}',0)).
    \end{split}
    \end{equation*}
    The second equality is the defining relation $Z_{s-2}=E_{s-1,p}((Y_0', \ldots, Y_{s-2}'))$ together with the change of variables \eqref{eqnchangeofvarsZ} for $i=s-1$.
    The first factor meets the finiteness condition because it is a polynomial. For the second factor, Lemma~\ref{lemmaG} shows that
    \[
        G_{s,p}((Y_0', \ldots, Y_{s-2}',0))
        \in \mathbb{Z}_{(p)}[\zeta_{p^s}]\llbracket (\zeta_{p^s}-1)Y_0',\ldots,(\zeta_{p^2}-1)Y_{s-2}'\rrbracket .
    \]
    Since each $\zeta_{p^m}-1$ has positive valuation and $Y_0',\ldots,Y_{s-2}'$ have coefficients in $R$ (by the induction hypothesis in Proposition~\ref{propASWtype}), only finitely many monomials in this expansion can have coefficient valuation $<\frac{p}{p-1}$. The proposition follows.
\end{proof}

\begin{corollary}
    The series $E_{s,p}((Y_0', \ldots, Y_{s-2}', 0)) \in R\llbracket Y_0, \ldots, Y_{s-2}\rrbracket$ contains only finitely many terms with coefficients of valuation strictly smaller than $\frac{1}{p-1}$. 
\end{corollary}

\begin{proof}
    Set $f:=E_{s,p}((Y_0', \ldots, Y_{s-2}', 0))\in R\llbracket Y_0,\ldots,Y_{s-2}\rrbracket$. By Proposition~\ref{propboundEi}, $f$ satisfies the hypotheses of Proposition~\ref{propfinitesmallmulti}.
    Moreover, $f^p=E_{s,p}(p(Y_0', \ldots, Y_{s-2}', 0))$, and Proposition~\ref{propfiniteEp} shows that $f^p$ has only finitely many terms with coefficient valuation $<\frac{p}{p-1}$. Applying Proposition~\ref{propfinitesmallmulti} to $f$ yields the claim.
\end{proof}

\begin{remark}
    This result does not generally hold for $E_{s,p}((Y_0, \ldots, Y_{s-2},0))$. 
\end{remark}

Therefore, the following definition is well-defined.

\begin{definition}
\label{defnEi}
    Define $E_s(Y_0, \ldots, Y_{s-2}) \in \mathbb{Z}_{(p)}[\zeta_{p^s}][Y_0, \ldots, Y_{s-2}]$ as the polynomial obtained by truncating $E_{s,p}((Y_0', \ldots, Y_{s-2}', 0)) \in \mathbb{Z}_{(p)}[\zeta_{p^s}]\llbracket Y_0, \ldots, Y_{s-2}\rrbracket$ at terms whose coefficients have valuations strictly smaller than $\frac{1}{p-1}$.
\end{definition}

\subsubsection{Computing the Minimal Polynomial}

\begin{proposition}
\label{propY'}
    With the choice of $E_s(Y_0, \ldots, Y_{s-2})$ as in Definition \ref{defnEi}, the entry $Y_{s-1}'$ of the associated Witt vector has the form
    \begin{equation*}
        Y_{s-1}' = (-1)^{p-1}\frac{\lambda^{p-1}}{p} Y_{s-1}^p + Y_{s-1} + F_s(Y_0, \ldots, Y_{s-1}),
    \end{equation*}
    where $F_s(Y_0, \ldots, Y_{s-1}) \in \mathbb{Z}_{(p)}[\zeta_{p^{s}}] \llbracket Y_0, \ldots, Y_{s-1} \rrbracket$.
\end{proposition}

\begin{proof}
    The case $s=1$ is proved in Section \ref{secchangevars1}. For $s>1$, we have the following expansion from Proposition \ref{propcomputeY} for $Y_{s-1}'$:
    \begin{equation*}
    \begin{split}
        Y_{s-1}' &= \frac{1}{\lambda} \ln\left(\frac{E_s(Y_0, \ldots, Y_{s-2}) + \lambda Y_{s-1}}{E_{s,p}((Y_0', \ldots, Y_{s-2}',0))}\right) \\
        &= \frac{1}{\lambda} \ln \left( 1+ \frac{S_s + \lambda Y_{s-1}}{E_{s,p}((Y_0', \ldots, Y_{s-2}',0))} \right) \\
        &= \frac{1}{\lambda} \ln \left( 1 + \lambda Y_{s-1} + S_s + o\left(p^{1/(p-1)}\right)\right) \\
        &= Y_{s-1}+ (-1)^{p-1}\frac{\lambda^{p-1}}{p}Y_{s-1}^p+ \frac{S_s}{\lambda}+ (-1)^{p-1} \frac{S_s^p}{p\lambda}+o\left( p^0 \right)
    \end{split}
    \end{equation*}
    Here, $S_s := E_s(Y_0, \ldots, Y_{s-2}) - E_{s,p}((Y_0', \ldots, Y_{s-2}',0))$, which has valuation at least $\frac{1}{p-1}$ by definition. The rest follows similarly to the proof of Proposition \ref{propcomputeWittvectorF} and the case for $s=1$. This completes the proof.
\end{proof}

\begin{proposition}
\label{propminimialpoly}
The minimal polynomial of $Y_{s-1}$ over $\mathbb{Q}[\zeta_{p^s}](Y_0, \ldots, Y_{s-2}, X_{s-1})$ is
\begin{equation}
\label{eqnminimalpoly}
    \frac{\lambda^{p-1}}{p} Y_{s-1}^p + Y_{s-1} + H_s(Y_0, \ldots, Y_{s-1}) - (E_{s-1}(Y_0, \ldots, Y_{s-3}) + \lambda Y_{s-2})X_{s-1} = 0.
\end{equation}
In particular, $X_j$ lies in $\mathbb{Z}_{(p)}[\zeta_{p^{j+1}}] \left[Y_0, \ldots, Y_j, \frac{1}{1 + \lambda Y_0}, \ldots, \frac{1}{E_{j-1} + \lambda Y_{j-1}}\right]$ for $j = 0, \ldots, s-1$.
\end{proposition}

\begin{proof}
Substituting $Z_{s-1} = E_s(Y_0, \ldots, Y_{s-2}) + \lambda Y_{s-1}$ into
\begin{equation*}
    Z_{s-1}^p = Z_{s-2}(G_s(X_0, \ldots, X_{s-2}) + p\lambda X_{s-1})
\end{equation*}
transforms the equation to
\begin{equation}
\label{eqnKummerexpansion}
\begin{split}
    \frac{1}{p\lambda} \big( (E_s(Y_0, \ldots, Y_{s-2}) + \lambda Y_{s-1})^p & - (E_{s-1}(Y_0, \ldots, Y_{s-3}) + \lambda Y_{s-2})G_s(X_0, \ldots, X_{s-2}) \big) \\
    &= X_{s-1} (E_{s-1}(Y_0, \ldots, Y_{s-3}) + \lambda Y_{s-2}).
\end{split}
\end{equation}
Let us analyze the left-hand side. By construction, we have
\begin{equation*}
    \begin{split}
        E_s(Y_0, \ldots, Y_{s-2}) &= E_{s,p}((Y_0', \ldots, Y_{s-2}',0)) + O\left( p^{1/(p-1)}\right), \\
        E_{s-1}(Y_0, \ldots, Y_{s-3}) + \lambda Y_{s-2} &= E_{s-1,p}((Y_0', \ldots, Y_{s-2}')), \\
        G_s(X_0, \ldots, X_{s-2}) &= G_{s,p}((X_0', \ldots, X_{s-2}',0)) + o\left( p^{p/(p-1)} \right).
    \end{split}
\end{equation*}
The first two expressions are series in $\mathbb{Z}_{(p)}[\zeta_{p^s}] \llbracket Y_0, \ldots, Y_{s-2} \rrbracket$, and the last is a series in $\mathbb{Z}_{(p)}[\zeta_{p^s}] \llbracket X_0, \allowbreak \ldots, X_{s-2} \rrbracket$. Additionally, for $j = 0, \ldots, s-1$, we have $X_j' = Y_j'$, and $X_j$ can be expressed as a polynomial in $\mathbb{Z}_{(p)}[\zeta_{p^{j+1}}] \left[Y_0, \ldots, Y_j, \frac{1}{1 + \lambda Y_0}, \ldots, \frac{1}{E_{j-1} + \lambda Y_{j-1}}\right]$ by induction. Thus, we obtain
\begin{equation*}
    G_s(X_0, \ldots, X_{s-2}) = G_{s,p}((Y_0', \ldots, Y_{s-2}',0)) + o\left( p^{p/(p-1)} \right)
\end{equation*}
as an element of $\mathbb{Z}_{(p)}[\zeta_{p^s}] \llbracket Y_0, \ldots, Y_{s-2} \rrbracket$. Recall also that 
\begin{equation*}
    E_{s,p}(p(Y_0, \ldots, Y_{s-2}, Y_{s-1})) = E_{s-1,p}((Y_0, \ldots, Y_{s-2})) \cdot G_{s,p}((Y_0, \ldots, Y_{s-2}, Y_{s-1})).
\end{equation*}
Therefore, the left-hand side of \eqref{eqnKummerexpansion} simplifies to
\begin{equation*}
    \begin{aligned}
        \text{LHS} &= \frac{1}{p\lambda} \bigg( \left(E_{s,p}((Y_0', \ldots, Y_{s-2}',0)) + O\left( p^{1/(p-1)} \right) + \lambda Y_{s-1} \right)^p  \\
                   &\quad - E_{s-1,p}((Y_0', \ldots, Y_{s-2}')) \left( G_{s,p}((Y_0', \ldots, Y_{s-2}',0)) + o\left( p^{p/(p-1)} \right) \right) \bigg) \\
                   &= \frac{\lambda^{p-1}}{p} Y_{s-1}^p + Y_{s-1} + O\left( p^{1/(p-1)} \right)
    \end{aligned}
\end{equation*}
as a series in $\mathbb{Z}_{(p)}[\zeta_{p^s}] \llbracket Y_0, \ldots, Y_{s-1} \rrbracket$. This forces the left-hand side to be a polynomial in
$\mathbb{Z}_{(p)}[\zeta_{p^{s}}] \big[Y_0, \ldots, Y_{s-1}, \frac{1}{1+\lambda Y_0}, \ldots, \frac{1}{E_{s-2}+\lambda Y_{s-2}}\big]$.
Thus the left-hand side can be written as
\[
\frac{\lambda^{p-1}}{p} Y_{s-1}^p + Y_{s-1} + H_s(Y_0,\ldots,Y_{s-1})
\]
for some $H_s$ in that ring. Furthermore, the left-hand side is a degree-$p$ expression in $Y_{s-1}$, and the Kummer equation
defines a degree-$p$ extension over $\mathbb{Q}[\zeta_{p^s}](Y_0, \ldots, Y_{s-2}, X_{s-1})$, so this polynomial is the minimal polynomial of $Y_{s-1}$.
Finally, for $j=s-1$ the defining equation \eqref{eqnminimalpoly} explicitly solves for $X_{s-1}$ by dividing by
$(E_{s-1}(Y_0,\ldots,Y_{s-3})+\lambda Y_{s-2})$, which is one of the inverted factors in the stated ring; the cases
$j\le s-2$ follow by induction as stated.
\end{proof}

\begin{proposition}
    The series $F_s$ from Proposition \ref{propY'} reduces to a polynomial in $\mathbb{F}_p[y_0, \ldots, y_{s-2}]$ modulo $(\zeta_{p^s} - 1)$. 
\end{proposition}

\begin{proof}
    Set $\pi:=\zeta_{p^s}-1$ and $u:=\frac{\lambda^{p-1}}{p}(-1)^p$.
    By Proposition~\ref{propY'} and the identity $Y_{s-1}'=X_{s-1}'$, we have
    \[
    -uY_{s-1}^p+Y_{s-1}+F_s=X_{s-1}'.
    \]
    By Proposition~\ref{propcomputeWittvectorF}, $X_{s-1}'\equiv X_{s-1}\pmod{\pi}$, hence
    \begin{equation}
    \label{eqnFsmodpi}
    -u y_{s-1}^p+y_{s-1}+\overline{F}_s=x_{s-1}.
    \end{equation}

    On the other hand, reducing \eqref{eqnminimalpoly} modulo $\pi$ and using
    \[
    E_{s-1}(Y_0,\ldots,Y_{s-3})+\lambda Y_{s-2}\equiv 1\pmod{\pi},
    \]
    we obtain
    \begin{equation}
    \label{eqnHsmodpi}
    -u y_{s-1}^p+y_{s-1}+\overline{H}_s=x_{s-1}.
    \end{equation}
    Comparing \eqref{eqnFsmodpi} and \eqref{eqnHsmodpi} gives $\overline{F}_s=\overline{H}_s$.

    It remains to show $\overline{H}_s\in\mathbb{F}_p[y_0,\ldots,y_{s-2}]$. From \eqref{eqnKummerexpansion},
    \[
    (E_s(Y_0,\ldots,Y_{s-2})+\lambda Y_{s-1})^p
    =\lambda^pY_{s-1}^p+p\lambda Y_{s-1}E_s(Y_0,\ldots,Y_{s-2})+E_s^p+o\!\left(p^{p/(p-1)}\right),
    \]
    where the omitted mixed terms have valuation $\ge p/(p-1)$. After dividing by $p\lambda$ and reducing modulo $\pi$, these omitted terms vanish, and $E_s(Y_0,\ldots,Y_{s-2})\equiv1\pmod{\pi}$ by Definition~\ref{defnEi}. Therefore $\overline{H}_s$ is a polynomial in $y_0,\ldots,y_{s-2}$, so the same holds for $\overline{F}_s$.
\end{proof}

This completes the proof of Proposition~\ref{propASWtype}.

In fact, we can provide further details on the reduction of $F_{s}$ and $H_s$ in \eqref{eqnminimalpoly}.

\begin{corollary}
\label{corspecialequation}
    We have, for $j=1, \ldots, s$,
    \begin{equation*}
       H_j \equiv F_j \equiv f_j(y_0, y_1+b_1, \ldots, y_{j-2}+b_{j-2}) + \wp(b_{j-1}) \pmod{(\zeta_{p^j}-1)},
\end{equation*}
where $f_j$ is the polynomial in \eqref{eqnsystemAS} and $b_{j-1} \in \mathbb{F}_p[y_0, \ldots, y_{j-2}]$. In particular, the special fiber of $\mathcal{W}_s \rightarrow \mathcal{V}_s$ is birational to the cover of $\proj k[x_0, \ldots, x_{s-1}, y]$ defined by
    \begin{equation*}
        \wp(y_0, y_1+b_1, \ldots, y_{s-1}+b_{s-1}) = (x_0, x_1, \ldots, x_{s-1}).
    \end{equation*}
\end{corollary}

\begin{proof}
    Due to \S \ref{secchangevars1}, the corollary holds for $s=1$.
    
    For $s>1$, recall that the special fiber of $\mathcal{W}_s \rightarrow \mathcal{V}_s$ is birational to the extension of $\mathbb{F}_p[x_0, \ldots, x_{s-1}]$ associated with the Witt vector $(x_0, \ldots, x_{s-1})$, which can be defined as a system of Artin--Schreier extensions as in \eqref{eqnsystemAS}. By induction, the last of these extensions is given by
    \begin{equation*}
        y_{s-1} - y_{s-1}^p + f_{s}(y_0, y_1+b_1, \ldots, y_{s-2}+b_{s-2}) + \wp(b_{s-1}) = x_{s-1},
    \end{equation*}
    where $b_{j-1} \in \mathbb{F}_p[y_0, \ldots, y_{j-2}]$ for $j=2,\ldots,s-1$ by the induction hypothesis, and \emph{a priori} we may take
    $b_{s-1} \in \mathbb{F}_p[y_0, \ldots, y_{s-2}, x_{s-1}]$. On the other hand, by \eqref{eqnminimalpoly}, we have
    \begin{equation*}
        y_{s-1} - y_{s-1}^p + \overline{H}_s = x_{s-1},
    \end{equation*}
    where $\overline{H}_s$ denotes the reduction of $H_s$ modulo $(\zeta_{p^s}-1)$. Comparing with the previous displayed equation (and cancelling the common terms $y_{s-1}-y_{s-1}^p$ and $x_{s-1}$), we obtain
\[
  \overline{H}_s
  =
  f_{s}(y_0, y_1+b_1, \ldots, y_{s-2}+b_{s-2})
  + \wp(b_{s-1}).
\]
In particular, $\wp(b_{s-1}) \in \mathbb{F}_p[y_0,\ldots,y_{s-2}]$, so we may (and do) choose $b_{s-1}\in \mathbb{F}_p[y_0,\ldots,y_{s-2}]$.
\end{proof}

\begin{example}
    When $p=2$, $\overline{H}_2= y_0^2-y_0^3+ y_0^4-y_0^8=f_2(y_0)+ \wp(y_0^4)$ if we choose $G_2(X_0) = 1 + 2\zeta_4 X_0 - 52X_0^4$. This is the lowest degree polynomial for $G_2$ that we can select.
\end{example}

\section{Smoothness of the Integral Closure}
\label{secintclosure}
In this section, we show that $\mathcal{W}_s$ has good reduction over $R := \mathbb{Z}_{(p)}[\zeta_{p^s}]$. Set $\pi := \zeta_{p^s}-1$.

\subsection{The case \texorpdfstring{$s=1$}{s=1}}
\label{secexplicit1}
The cover $\mathcal{W}_1$ is defined by the diagram in Figure~\ref{fig:construction_W1}, which records the relationships between $\mathcal{W}_1$, $\mathcal{V}_1$, and their function fields.
\begin{figure}[ht]
    \centering
    \begin{tikzcd}
        \mathcal{W}_1 \arrow[hookrightarrow]{r} \ar[d] &  \mathcal{Y}_1 \ar[r] \ar[d] & \operatorname{Spec} L_1 \ar[d] \\ 
        \mathcal{V}_1 \arrow[hookrightarrow]{r}  &  \mathbb{P}^1_R \ar[r] & \operatorname{Spec} K(X_0)
    \end{tikzcd}
    \caption{Construction of $\mathcal{W}_1$}
    \label{fig:construction_W1}
\end{figure}
Recall that $L_1/K(X_0)$ is the Kummer extension given by 
\begin{equation}
\label{eqnKummer1}
    Z_0^p = 1 + p\lambda X_0.
\end{equation}
That is, the coordinate ring of $\mathcal{W}_1$ is the integral closure of that of $\mathcal{V}_1$ in $L_1$. Substituting $Z_0=1+ \lambda Y_0$ into \eqref{eqnKummer1} yields
\begin{equation}
    \label{eqnASWtype1}
   Z_0^p= 1 + \lambda^p Y_0^p + p\lambda Y_0 + \sum_{i=2}^{p-1} \binom{p}{i} (\lambda Y_0)^i = 1 + p\lambda X_0,
\end{equation}
which simplifies to
\begin{equation}
    \label{eqnASWtype1simplified}
    S_0(Y_0):=Y_0 + \frac{\lambda^{p-1}}{p} Y_0^p + \frac{1}{p} \sum_{i=2}^{p-1} \binom{p}{i} \lambda^{i-1} Y_0^i =X_0, 
\end{equation}
and therefore reduces modulo $\lambda$ to
\begin{equation*}
    y_0 - y_0^p = x_0.
\end{equation*}

\begin{proposition}
\label{propintclosure1}
    $R \left[ Y_0, \frac{1}{1+\lambda Y_0} \right]$ is the integral closure of $R\left[X_0, \frac{1}{1+p\lambda X_0} \right]$ in $L_1$.
\end{proposition}

\begin{proof}
It is straightforward to show that $A := R\left[X_0, \frac{1}{1 + p\lambda X_0}\right]$ (resp. $B_1 := R\left[Y_0, \frac{1}{1 + \lambda Y_0}\right]$) is integrally closed in $\Frac A = K(X_0)$ (resp. $\Frac B_1 = K(Y_0) = L_1$), hence normal. From \eqref{eqnASWtype1simplified} and
\begin{equation*}
\frac{1}{Z_0^p} = \left( \frac{1}{1 + \lambda Y_0} \right)^p = \frac{1}{1 + p\lambda X_0},
\end{equation*}
we see that $Y_0$ and $\frac{1}{1 + \lambda Y_0}$ are integral over $A$. Hence, $B_1$ is contained in the closure of $A$ in $K(Y_0)$. Additionally, \eqref{eqnASWtype1simplified} shows that $X_0$ can be expressed as a polynomial in $B_1$. Therefore, the canonical inclusion $K(X_0) \hookrightarrow K(Y_0)$ induces $A \hookrightarrow B_1$. Thus, as $B_1$ is normal, it contains the integral closure of $A$ in $K(Y_0)$. Therefore, $B_1$ is the integral closure of $A$ in $L_1$.
\end{proof}

Thus, $\mathcal{W}_1 = \operatorname{Spec} R \left[Y_0, \frac{1}{1+\lambda Y_0}\right]$ has good reduction, with special fiber $\operatorname{Spec} \mathbb{F}_p[y_0]$. Moreover, by \eqref{eqnKummer1} and Kummer theory, the cover $\mathcal{Y}_1 \rightarrow \mathbb{P}^1_R$ is generically branched at $X_0 = \infty$ and $X_0 = -\frac{1}{p\lambda}$. Therefore, $\mathcal{W}_1 \rightarrow \mathcal{V}_1$ is {\'e}tale with Galois group $\mathbb{Z}/p$, and its special fiber is isomorphic to the Artin--Schreier extension $\operatorname{Spec} \mathbb{F}_p[y_0] \rightarrow \operatorname{Spec} \mathbb{F}_p[x_0]$ given by $-y_0^p + y_0 = x_0$.

\subsection{The case \texorpdfstring{$s > 1$}{s>1}}
Consider Figure~\ref{fig:construction_Ws}, which extends Figure~\ref{fig:construction_W1} and shows the step-by-step construction of $\mathcal{W}_s$.
\begin{figure}[ht]
    \centering
    \begin{tikzcd}
        \mathcal{W}_s \arrow[hookrightarrow]{r} \ar[d] & \mathcal{Y}_s \ar[r] \ar[d] & \operatorname{Spec} L_s \ar[d] \\
        \mathcal{W}_{s-1} \arrow[hookrightarrow]{r} \ar[d] & \mathcal{Y}_{s-1} \ar[r] \ar[d] & \operatorname{Spec} L_{s-1} \ar[d] \\
        \vdots \ar[d] & \vdots  \ar[d] & \vdots \ar[d] \\
        \mathcal{W}_{1} \arrow[hookrightarrow]{r} \ar[d] & \mathcal{Y}_{1} \ar[r] \ar[d] & \operatorname{Spec} L_{1} \ar[d] \\
        \mathcal{V}_s \arrow[hookrightarrow]{r} & \mathbb{P}^s_R \ar[r] & \operatorname{Spec} K(X_0, X_1, \ldots, X_{s-1})
    \end{tikzcd}
    \caption{Construction of $\mathcal{W}_s$}
    \label{fig:construction_Ws}
\end{figure}

Here we set $\mathbb{P}^s_R = \operatorname{Proj} R[X_0, X_1, \ldots, X_{s-1}, Y]$ and
\[
\mathcal{O}_{\mathcal{V}_s} = R \left[X_0, \ldots, X_{s-1}, \frac{1}{1 + p\lambda X_0},  \ldots, \frac{1}{G_{s}(X_0, \ldots, X_{s-2}) + p\lambda X_{s-1}}  \right],
\]
where the polynomials $G_2,\ldots,G_s$ are computed in Section \ref{secconstructW}, and $L_{i+1}/L_{i}$ is given by
\begin{equation}
\label{eqnKummer2}
Z_i^p = Z_{i-1} \left(G_{i+1}(X_0, \ldots, X_{i-1}) + p\lambda X_i \right).
\end{equation}
Recall from Proposition \ref{propminimialpoly} that the change of variables $Z_{i-1} = E_{i}(Y_0, \ldots, Y_{i-2}) + \lambda Y_{i-1}$, given in Section \ref{secE_i}, leads to the relations
$$X_{i-1} = S_{i-1}(Y_0, \ldots, Y_{i-1}),$$
where $S_{i-1}$ is an Artin--Schreier--Witt type polynomial similar to \eqref{eqnASWtype1simplified}. For $i\ge 2$, it lies in the ring
\begin{equation*}
    R \left[ Y_0, \ldots, Y_{i-1}, \frac{1}{1+\lambda Y_0}, \ldots, \frac{1}{E_{i-1}(Y_0, \ldots, Y_{i-3}) + \lambda Y_{i-2}} \right],
\end{equation*}
for $i=2, \ldots, s-1$. Following the same process as in Section~\ref{secexplicit1}, one shows that $\mathcal{W}_1 = \operatorname{Spec} B_1$, where $B_1$ is given by
\begin{equation*}
    R \left[ Y_0, X_1,  \ldots, X_{s-1}, \frac{1}{1 + \lambda Y_0}, \frac{1}{G_2(S_0)+p \lambda X_1}, \ldots, \frac{1}{G_{s}(S_0, \ldots, X_{s-2}) + p\lambda X_{s-1}}  \right]
\end{equation*}
Here $\mathcal{W}_1$ denotes the first stage in the tower over $\mathcal{V}_s$ (i.e.\ the pullback of the $s=1$ construction along the natural projection $\mathcal{V}_s \to \mathcal{V}_1$).

Assume the induction hypothesis that $\mathcal{W}_{s-1} = \spec B_{s-1}$, with $B_{s-1}$ given by
\begin{equation*}
R \left[Y_0, \ldots, Y_{s-2}, X_{s-1}, \frac{1}{1 + \lambda Y_0}, \ldots, \frac{1}{E_{s-1}(Y_0, \ldots, Y_{s-3}) + \lambda Y_{s-2}}, \frac{1}{G_s(S_0, \ldots, S_{s-2}) + p \lambda X_{s-1}} \right],
\end{equation*}
and the polynomials $S_j(Y_0,\ldots,Y_j)$ are as defined above.

\begin{proposition}
\label{propintclosure}
    The integral closure of $B_{s-1}$ in $L_s$ is
    \begin{equation*}
    B_s = R \left[Y_0, \ldots, Y_{s-1}, \frac{1}{1 + \lambda Y_0}, \ldots, \frac{1}{E_{s}(Y_0, \ldots, Y_{s-2}) + \lambda Y_{s-1}}  \right].
\end{equation*}
\end{proposition}

\begin{proof}
    We apply the following change of variables:
\begin{equation*}
    Z_{s-1} =: E_{s}(Y_0, \ldots, Y_{s-2}) + \lambda Y_{s-1},
\end{equation*}
where $E_{s}(Y_0, \ldots, Y_{s-2})$ is defined in Section \ref{secE_i}, transforming equation \eqref{eqnKummer2} when $i = s-1$ into
    \begin{equation}
    \label{eqnKummer2s-1}
        (E_{s}(Y_0, \ldots, Y_{s-2}) + \lambda Y_{s-1})^p = (E_{s-1}(Y_0, \ldots, Y_{s-3}) + \lambda Y_{s-2})\left(G_{s}(S_0, \ldots, S_{s-2}) + p\lambda X_{s-1} \right),
    \end{equation}
    which can also be rewritten as
    \begin{equation*}
        \left(\frac{1}{E_{s}(Y_0, \ldots, Y_{s-2}) + \lambda Y_{s-1}} \right)^p = \frac{1}{G_s(S_0, \ldots, S_{s-2}) + p \lambda X_{s-1}} \cdot \frac{1}{E_{s-1}(Y_0, \ldots, Y_{s-3}) + \lambda Y_{s-2}}.
    \end{equation*}
    In addition, Proposition \ref{propminimialpoly} shows that \eqref{eqnKummer2s-1} can be simplified to the minimal polynomial for $Y_{s-1}$ over $B_{s-1}$:
\begin{equation*}
\frac{1}{E_{s-1}(Y_0, \ldots, Y_{s-3}) +\lambda Y_{s-2}} \left(\frac{\lambda^{p-1}}{p} Y_{s-1}^p + Y_{s-1} + H_s(Y_0, \ldots, Y_{s-1}) \right) = X_{s-1}.
\end{equation*}
These equations imply that $Y_{s-1}$ and $\frac{1}{E_{s} + \lambda Y_{s-1}}$ are integral over $B_{s-1}$, and that $X_{s-1}$ lies in $B_s$. Hence $B_s$ is integral over $B_{s-1}$ and contained in the integral closure $\widetilde{B_{s-1}}$ of $B_{s-1}$ in $L_s$. Conversely, as in Proposition~\ref{propintclosure1}, $B_s$ is normal (being a localization of a polynomial ring over the DVR $R$), contains $B_{s-1}$, and has fraction field $L_s$; therefore $\widetilde{B_{s-1}}\subseteq B_s$. Thus $\widetilde{B_{s-1}}=B_s$.
\end{proof}

\begin{corollary}
\label{corspecialunr}
    The special fiber of $\mathcal{W}_s \rightarrow \mathcal{V}_s$ is isomorphic to the morphism $\wp$ in \eqref{eqnASWdetailed}. In particular, it is unramified. 
\end{corollary}

\begin{proof}
    By Proposition~\ref{propintclosure}, $B_s$ is also the integral closure of $\mathcal{O}_{\mathcal{V}_s}$ in $L_s$. Furthermore, the special fiber of $\mathcal{W}_s$ (resp. $\mathcal{V}_s$) is $\mathbb{F}_p[y_0, \ldots, y_{s-1}]$ (resp. $\mathbb{F}_p[x_0, \ldots, x_{s-1}]$). Together with Corollary~\ref{corspecialequation}, this shows that the special fiber is isomorphic to $\wp$ in \eqref{eqnASWdetailed}. Finally, $\wp$ is unramified (equivalently étale) on $W_s$ since $\ker(\wp)\cong \mathbb{Z}/p^s$ is finite étale by \eqref{eqnASWdetailed}. 
\end{proof}

\begin{proposition} \label{propetalereduction}
The cover $\mathcal{W}_s \to \mathcal{V}_s$ is an étale cover with good reduction.
\end{proposition}
\begin{proof}
The assertion of good reduction holds because the special fiber is an étale $\mathbb{Z}/p^s$-torsor between smooth $k$-schemes, as established in Corollary~\ref{corspecialunr}.

For flatness, fix $x \in \mathcal{W}_s$ mapping to $y \in \mathcal{V}_s$, and set $A := \mathcal{O}_{\mathcal{V}_s,y}$ and $B := \mathcal{O}_{\mathcal{W}_s,x}$. Since $\mathcal{V}_s$ and $\mathcal{W}_s$ are smooth over the DVR $R$, both $A$ and $B$ are regular local rings, hence $B$ is Cohen--Macaulay. The morphism $\mathcal{W}_s \to \mathcal{V}_s$ is finite, so $A \to B$ is integral, and in particular $\dim A = \dim B$. By Miracle Flatness (Stacks Project, Lemma 10.128.1), it follows that $B$ is flat over $A$. Hence $\mathcal{W}_s \to \mathcal{V}_s$ is finite and flat.

For étaleness, let $M := \Omega_{\mathcal{W}_s/\mathcal{V}_s}$. Since $\mathcal{W}_s \to \mathcal{V}_s$ is of finite type, $M$ is a coherent $\mathcal{O}_{\mathcal{W}_s}$-module; since the morphism is finite, $M$ is in fact finite over $\mathcal{O}_{\mathcal{W}_s}$. By Proposition~\ref{propgenericunr}, the generic fiber is unramified, so $M[1/\pi]=0$. Thus, for each $x$, the stalk $M_x$ is a finitely generated $B$-module with $(M_x)[1/\pi]=0$, hence there exists $n \ge 1$ such that $\pi^n M_x = 0$. By Corollary~\ref{corspecialunr}, the special fiber is unramified, so $M_x/\pi M_x = 0$, i.e. $M_x=\pi M_x$. Therefore $M_x = \pi^n M_x = 0$ for all $x$, so $M=0$ and the morphism is unramified. Since it is also flat, it is étale.
\end{proof}

\section{The Group Structure}
\label{secgroupstructure}
In this section, we explain how the covers constructed in the previous sections carry natural group scheme structures, and how the morphism $\psi_s\colon \mathcal{W}_s\to\mathcal{V}_s$ becomes a group homomorphism. The key point is that, after completing along the identity sections, certain explicit coordinate maps (defined below) identify the formal neighborhoods of $\mathcal{V}_s$ and $\mathcal{W}_s$ with the formal Witt vector scheme; we then transport Witt vector addition to obtain formal group laws, and finally descend these formal structures to algebraic Hopf algebras over $R$.

Let $R := \mathbb{Z}_{(p)}[\zeta_{p^s}]$ be a discrete valuation ring, and let $\pi := \zeta_{p^s}-1$ be its uniformizer.

\medskip

\noindent\emph{Coordinate rings.}
Let
\[
A := \mathcal{O}(\mathcal{V}_s)
= R[X_0,\dots,X_{s-1}]\bigl[U_0^{-1},\dots,U_{s-1}^{-1}\bigr],
\]
with defining relations
\[
U_0 = 1 + p\lambda X_0,\qquad
U_j = G_{j+1}(X_0,\dots,X_{j-1}) + p\lambda X_j
\quad (1 \le j \le s-1),
\]
and
\[
B := \mathcal{O}(\mathcal{W}_s)
= R[Y_0,\dots,Y_{s-1}]\bigl[Z_0^{-1},\dots,Z_{s-1}^{-1}\bigr],
\]
with defining relations
\[
Z_0 = 1 + \lambda Y_0,\qquad
Z_j = E_{j+1}(Y_0,\dots,Y_{j-1}) + \lambda Y_j
\quad (1 \le j \le s-1).
\]

\medskip

\noindent\emph{Completions along the identity sections.}
Let $I_{\mathcal V}:=(X_0,\ldots,X_{s-1})\subset A$ and $I_{\mathcal W}:=(Y_0,\ldots,Y_{s-1})\subset B$ be the augmentation ideals, and let $\widehat{A}$ (resp.\ $\widehat{B}$) be the $I_{\mathcal V}$-adic (resp.\ $I_{\mathcal W}$-adic) completion.
Since each $U_j$ (resp.\ $Z_j$) has constant term $1$, we have $U_j\equiv 1\pmod{I_{\mathcal V}}$ (resp.\ $Z_j\equiv 1\pmod{I_{\mathcal W}}$); hence $U_j$ (resp.\ $Z_j$) is automatically a unit in the completion. Consequently,
\[
\widehat{A}\cong R\llbracket X_0,\ldots,X_{s-1}\rrbracket,
\qquad
\widehat{B}\cong R\llbracket Y_0,\ldots,Y_{s-1}\rrbracket.
\]

Let $e_{\mathcal V}\in \mathcal{V}_s(R)$ and $e_{\mathcal W}\in \mathcal{W}_s(R)$ denote the identity sections. The formal completions of $\mathcal{V}_s$ and $\mathcal{W}_s$ along these sections are
\[
\widehat{\mathcal{V}}_s := \widehat{\mathcal{V}}_{s,e_{\mathcal V}} = \Spf(\widehat{A}),
\qquad
\widehat{\mathcal{W}}_s := \widehat{\mathcal{W}}_{s,e_{\mathcal W}} = \Spf(\widehat{B}).
\]
Similarly, letting $0\in W_s(R)$ be the identity section, we write
\[
\widehat{W}_s := \widehat{W}_{s,0} = \Spf\bigl(\widehat{\mathcal{O}}_{W_s,0}\bigr)
\]
for the formal completion of $W_s$ at $0$.

From Proposition \ref{propcomputeY} and Proposition \ref{propcomputeX}, we have the (not necessarily algebraic) coordinate maps $\gamma$ and $\tau$. On formal completions along the identity sections, $\tau$ maps $\widehat{\mathcal{V}}_s$ to $\widehat{W}_s$, and $\gamma$ maps $\widehat{\mathcal{W}}_s$ to $\widehat{W}_s$. Concretely,
\begin{equation}
\label{eqngamma}
\gamma\colon \widehat{\mathcal{W}}_s \longrightarrow \widehat{W}_s,
\qquad
(Y_0, \ldots, Y_{s-1}) \longmapsto (Y_0', \ldots, Y_{s-1}')
\end{equation}
\begin{equation}\label{eqntau}
\tau\colon \widehat{\mathcal{V}}_s \longrightarrow \widehat{W}_s,
\qquad
(X_0, \ldots, X_{s-1}) \longmapsto (X_0', \ldots, X_{s-1}').
\end{equation}
More precisely, each $X_i'$ is computed inductively as follows:
\allowdisplaybreaks
\[
\begin{aligned}
    X_0' &= \frac{1}{p \lambda} \ln(1 + p\lambda X_0), \\
    X_1' &= \frac{1}{p \lambda} \ln \left( \frac{G_2(X_0) + p\lambda X_1}{G_{2,p}((X_0', 0))} \right), \\
    &\vdots \\
    X_{s-1}' &= \frac{1}{p \lambda} \ln \left( \frac{G_s(X_0, \ldots, X_{s-2}) + p\lambda X_{s-1}}{G_{s,p}((X_0', \ldots, X_{s-2}', 0))} \right).
\end{aligned}
\]
	Conversely, given $(Y_0',\ldots, Y_{s-1}') \in \widehat{W}_s$ (resp. $(X_0',\ldots, X_{s-1}')\in \widehat{W}_s$), one can compute the corresponding point in the formal completion $\widehat{\mathcal{W}}_s$ (resp. $\widehat{\mathcal{V}}_s$) along the identity section as follows.

\allowdisplaybreaks
\begin{equation}
\label{eqnfromWittbackY}
\begin{aligned}
    Y_0 &= \frac{1}{\lambda} \left( \exp \left( \lambda Y_0' \right) -1 \right), \\
    Y_1 &= \frac{1}{\lambda} \left( \exp \left( \lambda Y_1' \right) E_{2,p}((Y_0',0)) - E_2(Y_0) \right), \\
    &\vdots \\
    Y_{s-1} &= \frac{1}{\lambda} \left( \exp \left( \lambda Y_{s-1}' \right) E_{s,p}((Y_0', \ldots, Y_{s-2}',0)) - E_s(Y_0, \ldots, Y_{s-2}) \right).
\end{aligned}
\end{equation}
(resp.
\allowdisplaybreaks
\begin{equation}
\label{eqnfromWittback}
    \begin{aligned}
    X_0 &= \frac{1}{p \lambda} \left( \exp \left( p \lambda X_0' \right) -1 \right), \\
    X_1 &= \frac{1}{p \lambda} \left( \exp \left( p \lambda X_1' \right) G_{2,p}((X_0',0)) - G_2(X_0) \right), \\
    &\vdots \\
    X_{s-1} &= \frac{1}{p \lambda} \left( \exp \left( p \lambda X_{s-1}' \right) G_{s,p}((X_0', \ldots, X_{s-2}',0)) - G_s(X_0, \ldots, X_{s-2}) \right)
\end{aligned}
\end{equation}
). We denote the induced identification $\widehat{W}_s \to \widehat{\mathcal{V}}_s$ by $\tau^{-1}$ (resp. the identification $\widehat{W}_s \to \widehat{\mathcal{W}}_s$ by $\gamma^{-1}$).

\begin{lemma}\label{lem:tau-gamma-formal-iso}
With notation as above, we have the following.
\begin{enumerate}[label=(\alph*)]
\item \label{lem:tau-gamma-formal-isoa}
The coordinate map $\tau$ defined in \eqref{eqntau} restricts to an isomorphism of formal schemes over $R$, which we denote
\[
\widehat{\tau}\colon \widehat{\mathcal{V}}_s \xrightarrow{\ \sim\ } \widehat{W}_s,
\]
whose inverse is induced by the $X$-formulas in \eqref{eqnfromWittback}.
Equivalently, on completed local rings one gets mutually inverse continuous $R$-algebra homomorphisms
\[
\widehat{\tau}^{\sharp}\colon
\widehat{\mathcal{O}}_{W_s,0}\longrightarrow \widehat{\mathcal{O}}_{\mathcal{V}_s,e_{\mathcal V}},
\qquad
(\widehat{\tau}^{-1})^{\sharp}\colon
\widehat{\mathcal{O}}_{\mathcal{V}_s,e_{\mathcal V}}\longrightarrow \widehat{\mathcal{O}}_{W_s,0}.
\]
\item \label{lem:tau-gamma-formal-isob}
The coordinate map $\gamma$ defined in \eqref{eqngamma} restricts to an isomorphism of formal schemes over $R$, which we denote
\[
\widehat{\gamma}\colon \widehat{\mathcal{W}}_s \xrightarrow{\ \sim\ } \widehat{W}_s,
\]
whose inverse is induced by the $Y$-formulas in \eqref{eqnfromWittbackY}.
Equivalently, on completed local rings one gets mutually inverse continuous $R$-algebra homomorphisms
\[
\widehat{\gamma}^{\sharp}\colon
\widehat{\mathcal{O}}_{W_s,0}\longrightarrow \widehat{\mathcal{O}}_{\mathcal{W}_s,e_{\mathcal W}},
\qquad
(\widehat{\gamma}^{-1})^{\sharp}\colon
\widehat{\mathcal{O}}_{\mathcal{W}_s,e_{\mathcal W}}\longrightarrow \widehat{\mathcal{O}}_{W_s,0}.
\]
\end{enumerate}
\end{lemma}

\begin{proof}
We prove \ref{lem:tau-gamma-formal-isoa}; the proof of \ref{lem:tau-gamma-formal-isob} is identical, replacing $p\lambda$ with $\lambda$ and $G_{j+1},G_{j+1,p}$ with $E_{j+1},E_{j+1,p}$.

Let $I_{\mathcal V}:=(X_0,\ldots,X_{s-1})\subset A$ and $I_W:=(X_0',\ldots,X_{s-1}')\subset \mathcal{O}(W_s)$ be the ideals of the identity sections.
By the discussion above, the $I_{\mathcal V}$-adic completion of $A$ identifies with the completed local ring of $\mathcal{V}_s$ at $e_{\mathcal V}$, i.e.
\[
\widehat{A}\cong \widehat{\mathcal{O}}_{\mathcal{V}_s,e_{\mathcal V}}\cong R\llbracket X_0,\ldots,X_{s-1}\rrbracket,
\qquad
\widehat{\mathcal{O}}_{W_s,0}\cong R\llbracket X_0',\ldots,X_{s-1}'\rrbracket.
\]

\smallskip
\noindent\emph{Step 1: $\tau$ defines a morphism of formal schemes.}
To define a morphism $\widehat{\tau}:\widehat{\mathcal V}_s\to \widehat{W}_s$, it suffices to give a continuous $R$-algebra map
\[
\widehat{\tau}^{\sharp}:R\llbracket X_0',\ldots,X_{s-1}'\rrbracket\to R\llbracket X_0,\ldots,X_{s-1}\rrbracket
\]
sending $I_W$ into $I_{\mathcal V}$.

For $j=0$ set
\[
X_0'=\frac{1}{p\lambda}\ln(1+p\lambda X_0)=\sum_{n\ge 1}(-1)^{n-1}\frac{(p\lambda)^{n-1}}{n}X_0^n.
\]
Since $p\lambda\in pR$, one has $\nu((p\lambda)^{n-1})\ge n-1\ge \nu(n)$, so the coefficients lie in $R$ and $X_0'\in I_{\mathcal V}$.

Assume inductively that $X_0',\ldots,X_{j-1}'\in R\llbracket X_0,\ldots,X_{j-1}\rrbracket$ and have zero constant term.
Define
\[
X_j'=\frac{1}{p\lambda}\ln\!\Bigl(\frac{G_{j+1}(X_0,\ldots,X_{j-1})+p\lambda X_j}{G_{j+1,p}(X_0',\ldots,X_{j-1}',0)}\Bigr).
\]
By Proposition~\ref{propapproxG},
\[
G_{j+1}(X_0,\ldots,X_{j-1})\equiv G_{j+1,p}(X_0',\ldots,X_{j-1}',0)\pmod{p\lambda}.
\]
Indeed, $G_{j+1,p}(X_0',\ldots,X_{j-1}',0)-G_{j+1}(X_0,\ldots,X_{j-1})=G'_{j+1}$ has
$\nu(G'_{j+1})>\frac{p}{p-1}=\nu(p\lambda)$ by Proposition~\ref{propapproxG}, hence $G'_{j+1}\in(p\lambda)$.
Moreover $G_{j+1,p}(X_0',\ldots,X_{j-1}',0)$ has constant term $1$, hence is a unit in the complete local ring $R\llbracket X_0,\ldots,X_{j-1}\rrbracket$. Therefore the ratio is of the form $1+p\lambda\Delta$ with $\Delta\in R\llbracket X_0,\ldots,X_j\rrbracket$, and since numerator and denominator both have constant term $1$, one has $\Delta\in I_{\mathcal V}\,R\llbracket X_0,\ldots,X_j\rrbracket$.

Thus $\ln(1+p\lambda\Delta)$ is a well-defined formal power series in the $I_{\mathcal V}$-adic topology, and dividing by $p\lambda$ yields $X_j'\in I_{\mathcal V}$. This defines a continuous $R$-algebra map $\widehat{\tau}^{\sharp}$.

\smallskip
\noindent\emph{Step 2: $\tau^{-1}$ defines a morphism of formal schemes.}
Similarly, to define $\widehat{\tau}^{-1}:\widehat{W}_s\to\widehat{\mathcal V}_s$ it suffices to define a continuous map on completed local rings.
For $j=0$ set
\[
X_0=\frac{1}{p\lambda}\bigl(\exp(p\lambda X_0')-1\bigr)=\sum_{n\ge1}\frac{(p\lambda)^{n-1}}{n!}(X_0')^n.
\]
As $p\lambda\in pR$, one has $\nu((p\lambda)^{n-1})\ge n-1\ge \nu(n!)$, hence $X_0\in R\llbracket X_0'\rrbracket$ and $X_0\in I_W$.

Assume inductively that $X_0,\ldots,X_{j-1}\in R\llbracket X_0',\ldots,X_{j-1}'\rrbracket$ lie in $I_W$, and that
after substituting these series into the forward formulas, one recovers $X_0',\ldots,X_{j-1}'$.
Define $X_j$ by the recursion \eqref{eqnfromWittback}:
\[
X_j=\frac{1}{p\lambda}\Bigl(\exp(p\lambda X_j')\,G_{j+1,p}(X_0',\ldots,X_{j-1}',0)-G_{j+1}(X_0,\ldots,X_{j-1})\Bigr).
\]
Write the numerator as
\[
(\exp(p\lambda X_j')-1)\,G_{j+1,p}(X_0',\ldots,X_{j-1}',0)
\;+\;
\Bigl(G_{j+1,p}(X_0',\ldots,X_{j-1}',0)-G_{j+1}(X_0,\ldots,X_{j-1})\Bigr).
\]
The first term is in $(p\lambda)$. For the second term, recall from Proposition~\ref{propapproxG} that
\[
G'_{j+1}(X):=
G_{j+1,p}(X_0'(X),\ldots,X_{j-1}'(X),0)-G_{j+1}(X_0,\ldots,X_{j-1})
\in (p\lambda)\,R\llbracket X_0,\ldots,X_{j-1}\rrbracket.
\]
Substituting $X_i=X_i(X_0',\ldots,X_i')$ and using the inductive identities
$X_i'(X_0,\ldots,X_i)=X_i'$ for $i<j$, we get
\[
G_{j+1,p}(X_0',\ldots,X_{j-1}',0)-G_{j+1}(X_0,\ldots,X_{j-1})\in(p\lambda).
\]
Hence the numerator lies in $(p\lambda)$, so $X_j\in R\llbracket X_0',\ldots,X_j'\rrbracket$. Since both terms have constant term $1$, their difference has constant term $0$, so $X_j\in I_W$.

\smallskip
\noindent\emph{Step 3: the maps are inverse.}
We check that $\widehat{\tau}^{\sharp}$ and $(\widehat{\tau}^{-1})^{\sharp}$ are mutually inverse by verifying on the topological generators.
Let $I_{\mathcal V}:=(X_0,\ldots,X_{s-1})\subset R\llbracket X_0,\ldots,X_{s-1}\rrbracket$ and $I_W:=(X_0',\ldots,X_{s-1}')\subset R\llbracket X_0',\ldots,X_{s-1}'\rrbracket$.

The identities
\[
\exp(\ln(1+T))=1+T,\qquad \ln(\exp(T))=T,
\]
hold as formal identities in $R\llbracket T\rrbracket$; hence they hold after substitution for $T\in I_{\mathcal V}$ (resp.\ $T\in I_W$) in the $I_{\mathcal V}$-adic (resp.\ $I_W$-adic) topology.

For $j=0$, we have
\[
X_0'=\frac{1}{p\lambda}\ln(1+p\lambda X_0),
\qquad
X_0=\frac{1}{p\lambda}(\exp(p\lambda X_0')-1),
\]
so composing in either order reduces immediately to the identities above with $T=p\lambda X_0$ (resp. $T=p\lambda X_0'$).

Assume inductively that the two compositions are the identity on
$X'_0,\ldots,X'_{j-1}$ (equivalently on $X_0,\ldots,X_{j-1}$).
Set
\[
\Theta_j \,:=\, G_{j+1,p}(X'_0,\ldots,X'_{j-1},0)\in R\llbracket X'_0,\ldots,X'_{j-1}\rrbracket,
\qquad
\widetilde\Theta_j \,:=\, \widehat{\tau}^{\sharp}(\Theta_j)\in R\llbracket X_0,\ldots,X_{j-1}\rrbracket .
\]
Both $\Theta_j$ and $\widetilde\Theta_j$ have constant term $1$, hence are units.

\smallskip
\noindent\emph{(1) Check $(\widehat{\tau}^{-1})^{\sharp}\circ \widehat{\tau}^{\sharp}(X'_j)=X'_j$.}
Set
\[
V_j \,:=\,
\frac{G_{j+1}(X_0,\ldots,X_{j-1})+p\lambda X_j}{\widetilde\Theta_j}
\in R\llbracket X_0,\ldots,X_j\rrbracket,
\]
so by definition
\[
\widehat{\tau}^{\sharp}(X'_j)=\frac{1}{p\lambda}\ln(V_j).
\]
By Proposition~\ref{propapproxG}, $V_j=1+p\lambda\Delta_j$, so the logarithm is defined. Applying $(\widehat{\tau}^{-1})^{\sharp}$ and using the induction hypothesis on lower coordinates gives
\[
(\widehat{\tau}^{-1})^{\sharp}\!\bigl(G_{j+1}(X_0,\ldots,X_{j-1})\bigr)
= G_{j+1}\bigl((\widehat{\tau}^{-1})^{\sharp}(X_0),\ldots,(\widehat{\tau}^{-1})^{\sharp}(X_{j-1})\bigr),
\]
and similarly $(\widehat{\tau}^{-1})^{\sharp}(\widetilde\Theta_j)=\Theta_j$.
By the recursion \eqref{eqnfromWittback},
\[
p\lambda\,(\widehat{\tau}^{-1})^{\sharp}(X_j)
=
\exp(p\lambda X'_j)\,\Theta_j
-
G_{j+1}\bigl((\widehat{\tau}^{-1})^{\sharp}(X_0),\ldots,(\widehat{\tau}^{-1})^{\sharp}(X_{j-1})\bigr).
\]
Substituting this into $(\widehat{\tau}^{-1})^{\sharp}(V_j)$ yields
\[
(\widehat{\tau}^{-1})^{\sharp}(V_j)
=
\frac{G_{j+1}(\cdots)+p\lambda\,(\widehat{\tau}^{-1})^{\sharp}(X_j)}{\Theta_j}
=
\frac{G_{j+1}(\cdots)+\bigl(\exp(p\lambda X'_j)\,\Theta_j - G_{j+1}(\cdots)\bigr)}{\Theta_j}
=
\exp(p\lambda X'_j).
\]
Therefore
\[
(\widehat{\tau}^{-1})^{\sharp}\bigl(\widehat{\tau}^{\sharp}(X'_j)\bigr)
=
\frac{1}{p\lambda}\ln\!\bigl((\widehat{\tau}^{-1})^{\sharp}(V_j)\bigr)
=
\frac{1}{p\lambda}\ln\!\bigl(\exp(p\lambda X'_j)\bigr)
=
X'_j,
\]
using $\ln(\exp(T))=T$ for $T=p\lambda X'_j$ (with the stated valuation condition).

\smallskip
\noindent\emph{(2) Check $\widehat{\tau}^{\sharp}\circ(\widehat{\tau}^{-1})^{\sharp}(X_j)=X_j$.}
Start from the recursion \eqref{eqnfromWittback}, viewed as the definition of $(\widehat{\tau}^{-1})^{\sharp}(X_j)$:
\[
p\lambda\,(\widehat{\tau}^{-1})^{\sharp}(X_j)
=
\exp(p\lambda X'_j)\,\Theta_j
-
G_{j+1}\bigl((\widehat{\tau}^{-1})^{\sharp}(X_0),\ldots,(\widehat{\tau}^{-1})^{\sharp}(X_{j-1})\bigr).
\]
Apply $\widehat{\tau}^{\sharp}$ to both sides. By the induction hypothesis,
$\widehat{\tau}^{\sharp}\circ(\widehat{\tau}^{-1})^{\sharp}$ fixes $X_0,\ldots,X_{j-1}$, hence it fixes the polynomial
$G_{j+1}(X_0,\ldots,X_{j-1})$. Also $\widehat{\tau}^{\sharp}(\Theta_j)=\widetilde\Theta_j$, and
$\widehat{\tau}^{\sharp}(\exp(p\lambda X'_j))=\exp\bigl(p\lambda\,\widehat{\tau}^{\sharp}(X'_j)\bigr)$.
But $\widehat{\tau}^{\sharp}(X'_j)$ was defined so that
\[
\exp\bigl(p\lambda\,\widehat{\tau}^{\sharp}(X'_j)\bigr)=V_j.
\]
Since $\widehat{\tau}^{\sharp}(\Theta_j)=\widetilde\Theta_j$, we get
\[
\widehat{\tau}^{\sharp}\bigl(\exp(p\lambda X'_j)\,\Theta_j\bigr)
=
\exp\bigl(p\lambda\,\widehat{\tau}^{\sharp}(X'_j)\bigr)\,\widetilde\Theta_j
=
V_j\,\widetilde\Theta_j.
\]
Therefore
\[
p\lambda\cdot \widehat{\tau}^{\sharp}\bigl((\widehat{\tau}^{-1})^{\sharp}(X_j)\bigr)
=
V_j\,\widetilde\Theta_j - G_{j+1}(X_0,\ldots,X_{j-1})
=
p\lambda X_j,
\]
so dividing by $p\lambda$ gives
$\widehat{\tau}^{\sharp}\circ(\widehat{\tau}^{-1})^{\sharp}(X_j)=X_j$.
This completes the induction for $j$.

Thus $\widehat{\tau}$ is an isomorphism of formal schemes with inverse induced by \eqref{eqnfromWittback}.
\end{proof}

We can therefore naturally declare the group operations on the formal completions $\widehat{\mathcal{W}}_s$ and $\widehat{\mathcal{V}}_s$ over $R$.

\begin{definition}
\label{defngroupstructurecompletion}
    Define a group structure on $\widehat{\mathcal{W}}_s$ (and similarly on $\widehat{\mathcal{V}}_s$) over $R$ by pulling back the Witt vector group structure on $\widehat{W}_s$ via the coordinate maps $\gamma$ (or $\tau$).
\end{definition}

\begin{proposition}\label{prop:formal-group-structure}
The map $\tau$ (resp.\ $\gamma$) induces a \emph{formal group scheme} structure
over $R$ on the formal completion $\widehat{\mathcal V}_s$ (resp.\ $\widehat{\mathcal W}_s$)
along the identity section. Equivalently, the structure maps
\[
\widehat{\mathcal V}_s\times_R \widehat{\mathcal V}_s \longrightarrow \widehat{\mathcal V}_s,
\qquad
\widehat{\mathcal V}_s \longrightarrow \widehat{\mathcal V}_s,
\qquad
\Spf(R)\longrightarrow \widehat{\mathcal V}_s
\]
(and similarly for $\widehat{\mathcal W}_s$)
are given in coordinates by $s$-tuples of formal power series with coefficients in $R$.
\end{proposition}

\begin{proof}
We treat $\mathcal V_s$; the case of $\mathcal W_s$ is identical with $\gamma$ in place of $\tau$.

By Lemma~\ref{lem:tau-gamma-formal-iso}, the map
\[
\widehat{\tau}:\widehat{\mathcal V}_s \longrightarrow \widehat{W}_s
\]
is an isomorphism of formal schemes over $R$; equivalently, on completed local rings it induces mutually inverse continuous $R$-algebra maps
\[
\widehat{\tau}^{\sharp}: \widehat{\mathcal O}_{W_s,0}\to \widehat{\mathcal O}_{\mathcal V_s,e_{\mathcal V}},
\qquad
(\widehat{\tau}^{-1})^{\sharp}: \widehat{\mathcal O}_{\mathcal V_s,e_{\mathcal V}}\to \widehat{\mathcal O}_{W_s,0}.
\]

\smallskip
\noindent\emph{Step 1: Witt operations restrict to the formal completion.}
Let $\mu_W:W_s\times_R W_s\to W_s$ and $\iota_W:W_s\to W_s$ denote Witt addition and inversion.
Since these are morphisms of $R$-schemes sending $(0,0)$ to $0$ and $0$ to $0$, they induce morphisms on formal completions
\[
\widehat{\mu}_W:\widehat{W}_s\times_R\widehat{W}_s\to\widehat{W}_s,
\qquad
\widehat{\iota}_W:\widehat{W}_s\to\widehat{W}_s.
\]
On coordinate rings, this means $\widehat{\mu}_W^{\sharp}$ and $\widehat{\iota}_W^{\sharp}$ are continuous homomorphisms
\begin{align*}
R\llbracket X_0',\ldots,X_{s-1}'\rrbracket
&\longrightarrow
R\llbracket X_{0,1}',\ldots,X_{s-1,1}',X_{0,2}',\ldots,X_{s-1,2}'\rrbracket,\\
R\llbracket X_0',\ldots,X_{s-1}'\rrbracket
&\longrightarrow
R\llbracket X_0',\ldots,X_{s-1}'\rrbracket,
\end{align*}
where the fiber product $\times_R$ corresponds (on completed coordinate rings) to the completed tensor product.
Here $X_{i,1}' = X_i'\otimes 1$ and $X_{i,2}'=1\otimes X_i'$.

\smallskip
\noindent\emph{Step 2: Transport of structure.}
Define the formal multiplication and inverse on $\widehat{\mathcal V}_s$ by
\[
\widehat{\mu}_\tau := \widehat{\tau}^{-1}\circ \widehat{\mu}_W\circ (\widehat{\tau}\times \widehat{\tau}),
\qquad
\widehat{\iota}_\tau := \widehat{\tau}^{-1}\circ \widehat{\iota}_W\circ \widehat{\tau},
\qquad
 e := \widehat{\tau}^{-1}(0).
\]
Each map is a morphism of formal schemes because it is a composition of morphisms of formal schemes.

\smallskip
\noindent\emph{Step 3: Coordinate description (power series over $R$).}
Using
$\widehat{\mathcal O}_{\mathcal V_s,e_{\mathcal V}}\cong R\llbracket X_0,\ldots,X_{s-1}\rrbracket$
and
$\widehat{\mathcal O}_{W_s,0}\cong R\llbracket X_0',\ldots,X_{s-1}'\rrbracket$,
write
\begin{align*}
\widehat{\tau}^{\sharp}(X_i')
&= f_i(X_0,\ldots,X_{s-1})
\in (X_0,\ldots,X_{s-1})\,R\llbracket X_0,\ldots,X_{s-1}\rrbracket,\\
(\widehat{\tau}^{-1})^{\sharp}(X_i)
&= g_i(X_0',\ldots,X_{s-1}')
\in (X_0',\ldots,X_{s-1}')\,R\llbracket X_0',\ldots,X_{s-1}'\rrbracket.
\end{align*}

\smallskip
\noindent\emph{(a) The inverse map.}
The formal inverse $\widehat{\iota}_\tau$ has comorphism
\[
\widehat{\iota}_\tau^{\sharp}
=\widehat{\tau}^{\sharp}\circ \widehat{\iota}_W^{\sharp}\circ (\widehat{\tau}^{-1})^{\sharp}
\colon R\llbracket X_0,\ldots,X_{s-1}\rrbracket\to R\llbracket X_0,\ldots,X_{s-1}\rrbracket.
\]
Witt inversion on $W_s$ is given by universal polynomials $I_0,\ldots,I_{s-1}\in\mathbb{Z}[\underline U]$, i.e.
\(
\widehat{\iota}_W^{\sharp}(X_i')=I_i(X_0',\ldots,X_{s-1}')
\).
Applying $(\widehat{\tau}^{-1})^{\sharp}$ sends $X_j$ to $g_j\in R\llbracket X'\rrbracket$, so
$I_i(g_0,\ldots,g_{s-1})\in R\llbracket X'\rrbracket$ since $I_i$ has integral coefficients.
Composing with $\widehat{\tau}^{\sharp}$ then yields an element of $R\llbracket X\rrbracket$.
Hence each coordinate of $\widehat{\iota}_\tau$ is a power series with coefficients in $R$.

\smallskip
\noindent\emph{(b) The multiplication map.}
Likewise, the formal multiplication $\widehat{\mu}_\tau$ has comorphism
\[
\widehat{\mu}_\tau^{\sharp}
=(\widehat{\tau}\times\widehat{\tau})^{\sharp}\circ \widehat{\mu}_W^{\sharp}\circ (\widehat{\tau}^{-1})^{\sharp}
\colon R\llbracket X_0,\ldots,X_{s-1}\rrbracket\to R\llbracket X_{0,1},\ldots,X_{s-1,1},X_{0,2},\ldots,X_{s-1,2}\rrbracket.
\]
Here $X_{i,1}=X_i\otimes 1$ and $X_{i,2}=1\otimes X_i$.
Witt addition is given by universal Witt polynomials $\Sigma_0,\ldots,\Sigma_{s-1}\in\mathbb{Z}[\underline U,\underline V]$, i.e.
\(
\widehat{\mu}_W^{\sharp}(X_i')=\Sigma_i(X_{0,1}',\ldots,X_{s-1,1}',X_{0,2}',\ldots,X_{s-1,2}')
\).
Applying $(\widehat{\tau}^{-1})^{\sharp}$ and then $\widehat{\mu}_W^{\sharp}$ gives an element of
$R\llbracket X_{0,1}',\ldots,X_{s-1,1}', \allowbreak X_{0,2}',\ldots,X_{s-1,2}'\rrbracket$
since each $\Sigma_i$ is an integral polynomial.
Finally, $(\widehat{\tau}\times\widehat{\tau})^{\sharp}$ substitutes
$X_{j,1}'\mapsto f_j(X_{0,1},\ldots,X_{s-1,1})$ and $X_{j,2}'\mapsto f_j(X_{0,2},\ldots,X_{s-1,2})$.
Since each $f_j$ has zero constant term, this substitution is well-defined and yields an element of
$R\llbracket X_{0,1},\ldots,X_{s-1,1},X_{0,2},\ldots,X_{s-1,2}\rrbracket$.
Thus each coordinate of $\widehat{\mu}_\tau$ is a power series with coefficients in $R$.

\smallskip
\noindent\emph{Step 4: Group axioms.}
Associativity, commutativity, identity, and inverses hold because they hold for $(\widehat{W}_s,\widehat{\mu}_W,\widehat{\iota}_W)$ and we transported the structure along the isomorphism $\widehat{\tau}$.
\end{proof}

To describe these group schemes algebraically, we specify the Hopf algebra structures on their coordinate rings corresponding to Witt vector addition.

\medskip

Let $K=\mathrm{Frac}(R)$, and set $A_K:=A\otimes_R K$ and $B_K:=B\otimes_R K$. We denote
\[
T_K^{\mathcal{V}} := K[U_0^{\pm1},\dots,U_{s-1}^{\pm1}] \subset A_K
\qquad\text{and}\qquad
T_K^{\mathcal{W}} := K[Z_0^{\pm1},\dots,Z_{s-1}^{\pm1}] \subset B_K.
\]

\begin{proposition}
\label{prop:Hopf-over-K}
\begin{enumerate}[label=(\alph*)]
\item\label{prop:Hopf-over-Ka}
There exists a unique Hopf $K$-algebra structure
\[
(\Delta_K^{\mathcal{V}},\varepsilon_K^{\mathcal{V}},S_K^{\mathcal{V}})\colon
A_K \longrightarrow A_K\otimes_K A_K,\; K,\; A_K
\]
such that, for all $0\le j\le s-1$,
\[
\Delta_K^{\mathcal{V}}(U_j)=U_j\otimes U_j,\qquad
\varepsilon_K^{\mathcal{V}}(U_j)=1,\qquad
S_K^{\mathcal{V}}(U_j)=U_j^{-1}.
\]
With respect to the generators $X_j$, the structure maps are given recursively by
\begin{align*}
\varepsilon_K^{\mathcal{V}}(X_j) &= 0 \qquad (0\le j\le s-1),\\[4pt]
\Delta_K^{\mathcal{V}}(X_0)
&= \frac{U_0\otimes U_0-1}{p\lambda}
= \frac{(1+p\lambda X_0)\otimes(1+p\lambda X_0)-1}{p\lambda},\\[4pt]
\Delta_K^{\mathcal{V}}(X_j)
&= \frac{U_j\otimes U_j-
G_{j+1}\bigl(\Delta_K^{\mathcal{V}}(X_0),\dots,\Delta_K^{\mathcal{V}}(X_{j-1})\bigr)}{p\lambda}
\qquad (1\le j\le s-1),\\[4pt]
S_K^{\mathcal{V}}(X_0)
&= \frac{U_0^{-1}-1}{p\lambda}
= \frac{(1+p\lambda X_0)^{-1}-1}{p\lambda},\\[4pt]
S_K^{\mathcal{V}}(X_j)
&= \frac{U_j^{-1}-
G_{j+1}\bigl(S_K^{\mathcal{V}}(X_0),\dots,S_K^{\mathcal{V}}(X_{j-1})\bigr)}{p\lambda}
\qquad (1\le j\le s-1).
\end{align*}

\item\label{prop:Hopf-over-Kb}
There exists a unique Hopf $K$-algebra structure
\[
(\Delta_K^{\mathcal W},\varepsilon_K^{\mathcal W},S_K^{\mathcal W})\colon
B_K \longrightarrow B_K\otimes_K B_K,\; K,\; B_K
\]
such that, for all $0\le j\le s-1$,
\[
\Delta_K^{\mathcal W}(Z_j)=Z_j\otimes Z_j,\qquad
\varepsilon_K^{\mathcal W}(Z_j)=1,\qquad
S_K^{\mathcal W}(Z_j)=Z_j^{-1}.
\]
With respect to the generators $Y_j$, the structure maps are given recursively by
\begin{align*}
\varepsilon_K^{\mathcal W}(Y_j) &= 0 \qquad (0\le j\le s-1),\\[4pt]
\Delta_K^{\mathcal W}(Y_0)
&= \frac{Z_0\otimes Z_0-1}{\lambda}
= \frac{(1+\lambda Y_0)\otimes(1+\lambda Y_0)-1}{\lambda},\\[4pt]
\Delta_K^{\mathcal W}(Y_j)
&= \frac{Z_j\otimes Z_j-
E_{j+1}\bigl(\Delta_K^{\mathcal W}(Y_0),\dots,
\Delta_K^{\mathcal W}(Y_{j-1})\bigr)}{\lambda}
\qquad (1\le j\le s-1),\\[4pt]
S_K^{\mathcal W}(Y_0)
&= \frac{Z_0^{-1}-1}{\lambda}
= \frac{(1+\lambda Y_0)^{-1}-1}{\lambda},\\[4pt]
S_K^{\mathcal W}(Y_j)
&= \frac{Z_j^{-1}-
E_{j+1}\bigl(S_K^{\mathcal W}(Y_0),\dots,
S_K^{\mathcal W}(Y_{j-1})\bigr)}{\lambda}
\qquad (1\le j\le s-1).
\end{align*}
\end{enumerate}
\end{proposition}

\begin{proof}
We prove \ref{prop:Hopf-over-Ka} and \ref{prop:Hopf-over-Kb} in parallel, since the constructions are formally identical.

\smallskip
\noindent\emph{Convention.} We repeatedly use that each of the polynomials $G_{j+1}(X_0,\dots,X_{j-1})$ and $E_{j+1}(Y_0,\dots,Y_{j-1})$ has constant term $1$ (equivalently $G_{j+1}(0,\dots,0)=E_{j+1}(0,\dots,0)=1$). 

\medskip\noindent
\emph{Proof of \ref{prop:Hopf-over-Ka}.}
\smallskip\noindent
\emph{Step 1: Torus subalgebra.}
Declare on \(T_K^{\mathcal{V}}\) the standard split torus Hopf structure
\[
\Delta^{\mathcal{V}}_K(U_j)=U_j\otimes U_j,\qquad
\varepsilon^{\mathcal{V}}_K(U_j)=1,\qquad
S^{\mathcal{V}}_K(U_j)=U_j^{-1}.
\]
Then \(T_K^{\mathcal{V}}\simeq \mathcal{O}((\mathbb{G}_m)^s_K)\) is a Hopf \(K\)-algebra.

\smallskip\noindent
\emph{Step 2: Counit on \(A_K\).}
Define \(\varepsilon^{\mathcal{V}}_K\colon A_K\to K\) by \(\varepsilon^{\mathcal{V}}_K(X_j)=0\) for all \(0\le j\le s-1\).
We check compatibility with the defining relations of \(A\).
For \(j=0\), the relation \(U_0=1+p\lambda X_0\) gives \(\varepsilon^{\mathcal{V}}_K(U_0)=1\).
For \(j\ge 1\), using
\[
U_j=G_{j+1}(X_0,\dots,X_{j-1})+p\lambda X_j
\quad\text{and}\quad
G_{j+1}(0,\dots,0)=1,
\]
we obtain \(\varepsilon^{\mathcal{V}}_K(U_j)=1\).
Hence \(\varepsilon^{\mathcal{V}}_K\) extends the counit on \(T_K^{\mathcal{V}}\).

\smallskip\noindent
\emph{Step 3: Comultiplication on \(A_K\).}
Since \(p\lambda\in K^\times\), the defining relations can be rewritten as
\[
p\lambda X_0=U_0-1,\qquad
p\lambda X_j=U_j-G_{j+1}(X_0,\dots,X_{j-1})\quad (1\le j\le s-1).
\]
We already fixed \(\Delta^{\mathcal{V}}_K(U_j)=U_j\otimes U_j\).
Define \(\Delta^{\mathcal{V}}_K(X_j)\) recursively by
\[
\Delta^{\mathcal{V}}_K(X_0):=\frac{U_0\otimes U_0-1}{p\lambda},
\qquad
\Delta^{\mathcal{V}}_K(X_j):=\frac{U_j\otimes U_j-
G_{j+1}\bigl(\Delta^{\mathcal{V}}_K(X_0),\dots,\Delta^{\mathcal{V}}_K(X_{j-1})\bigr)}{p\lambda}.
\]
Applying \(\Delta^{\mathcal{V}}_K\) to the relation
\(p\lambda X_j=U_j-G_{j+1}(X_0,\dots,X_{j-1})\)
shows that the image of the relation is \(0\) in \(A_K\otimes_K A_K\) by construction.
Moreover, since each \(U_j\) is group-like,
\(\Delta^{\mathcal{V}}_K(U_j^{-1})=(U_j\otimes U_j)^{-1}=U_j^{-1}\otimes U_j^{-1}\),
so \(\Delta^{\mathcal{V}}_K\) respects the localization and extends uniquely to a \(K\)-algebra map
\(A_K\to A_K\otimes_K A_K\).

\smallskip\noindent
\emph{Step 4: Antipode on \(A_K\).}
Set \(S^{\mathcal{V}}_K(U_j)=U_j^{-1}\) for all \(j\).
Define \(S^{\mathcal{V}}_K(X_j)\) recursively by
\[
S^{\mathcal{V}}_K(X_0):=\frac{U_0^{-1}-1}{p\lambda},
\qquad
S^{\mathcal{V}}_K(X_j):=\frac{U_j^{-1}-
G_{j+1}\bigl(S^{\mathcal{V}}_K(X_0),\dots,S^{\mathcal{V}}_K(X_{j-1})\bigr)}{p\lambda}.
\]
Applying \(S^{\mathcal{V}}_K\) to the relation
\(p\lambda X_j=U_j-G_{j+1}(X_0,\dots,X_{j-1})\)
shows it is preserved.
Also \(S^{\mathcal{V}}_K(U_j^{-1})=U_j\), so \(S^{\mathcal{V}}_K\) respects the localization, hence extends to a
\(K\)-algebra homomorphism \(A_K\to A_K\).

\smallskip\noindent
\emph{Step 5: Hopf axioms and uniqueness.}
The Hopf identities hold on the Hopf subalgebra \(T_K^{\mathcal{V}}\).
Assume inductively that they hold on the \(K\)-subalgebra generated by
\(T_K^{\mathcal{V}}\) and \(X_0,\dots,X_{j-1}\).
Consider the defining identity
\[
p\lambda X_j=U_j-G_{j+1}(X_0,\dots,X_{j-1}).
\]

\emph{Coassociativity.}
Apply \((\Delta^{\mathcal{V}}_K\otimes \mathrm{id})\circ\Delta^{\mathcal{V}}_K\) and \((\mathrm{id}\otimes\Delta^{\mathcal{V}}_K)\circ\Delta^{\mathcal{V}}_K\)
to the formula defining \(\Delta^{\mathcal{V}}_K(X_j)\).
The \(U_j\)-term gives \(U_j\otimes U_j\otimes U_j\) on both sides since \(U_j\) is group-like.
By the induction hypothesis, \(\Delta^{\mathcal{V}}_K\) is coassociative on \(X_0,\dots,X_{j-1}\), hence the
two ways of applying \(\Delta^{\mathcal{V}}_K\) to the polynomial
\(G_{j+1}\bigl(\Delta^{\mathcal{V}}_K(X_0),\dots,\Delta^{\mathcal{V}}_K(X_{j-1})\bigr)\)
agree, and therefore coassociativity holds on \(X_j\).

\emph{Counit.}
Apply \((\varepsilon^{\mathcal{V}}_K\otimes \mathrm{id})\circ\Delta^{\mathcal{V}}_K\) to the defining formula for \(\Delta^{\mathcal{V}}_K(X_j)\).
Since \(\varepsilon^{\mathcal{V}}_K(U_j)=1\), we have
\[
(\varepsilon^{\mathcal{V}}_K\otimes\mathrm{id})(U_j\otimes U_j)=U_j.
\]
By the induction hypothesis, \((\varepsilon^{\mathcal{V}}_K\otimes\mathrm{id})\Delta^{\mathcal{V}}_K(X_i)=X_i\) for \(i<j\), hence
\[
(\varepsilon^{\mathcal{V}}_K\otimes\mathrm{id})\,
G_{j+1}\bigl(\Delta^{\mathcal{V}}_K(X_0),\dots,\Delta^{\mathcal{V}}_K(X_{j-1})\bigr)
=
G_{j+1}(X_0,\dots,X_{j-1}).
\]
Therefore
\[
(\varepsilon^{\mathcal{V}}_K\otimes \mathrm{id})\Delta^{\mathcal{V}}_K(X_j)
=
\frac{U_j-G_{j+1}(X_0,\dots,X_{j-1})}{p\lambda}
= X_j,
\]
using the defining relation \(p\lambda X_j=U_j-G_{j+1}(X_0,\dots,X_{j-1})\).
Similarly \((\mathrm{id}\otimes \varepsilon^{\mathcal{V}}_K)\Delta^{\mathcal{V}}_K(X_j)=X_j\).

\emph{Antipode axiom.}
Apply \(m\circ(\mathrm{id}\otimes S^{\mathcal{V}}_K)\) to the defining formula for \(\Delta^{\mathcal{V}}_K(X_j)\).
The group-like property gives
\(m(\mathrm{id}\otimes S^{\mathcal{V}}_K)\Delta^{\mathcal{V}}_K(U_j)=U_j\cdot U_j^{-1}=1\).
By the induction hypothesis,
\(m(\mathrm{id}\otimes S^{\mathcal{V}}_K)\Delta^{\mathcal{V}}_K(X_i)=\varepsilon^{\mathcal{V}}_K(X_i)=0\) for \(i<j\),
so
\[
m(\mathrm{id}\otimes S^{\mathcal{V}}_K)\,
G_{j+1}\bigl(\Delta^{\mathcal{V}}_K(X_0),\dots,\Delta^{\mathcal{V}}_K(X_{j-1})\bigr)
=
G_{j+1}(0,\dots,0)=1.
\]
Hence \(m(\mathrm{id}\otimes S^{\mathcal{V}}_K)\Delta^{\mathcal{V}}_K(X_j)=0=\varepsilon^{\mathcal{V}}_K(X_j)\).
The identity \(m(S^{\mathcal{V}}_K\otimes \mathrm{id})\Delta^{\mathcal{V}}_K(X_j)=\varepsilon^{\mathcal{V}}_K(X_j)\) is analogous.

This proves that \((A_K,\Delta^{\mathcal{V}}_K,\varepsilon^{\mathcal{V}}_K,S^{\mathcal{V}}_K)\) is a Hopf \(K\)-algebra.
Uniqueness follows immediately: if \((\Delta,\varepsilon,S)\) is any Hopf structure on \(A_K\)
with all \(U_j\) group-like, then \(\varepsilon(U_j)=1\) and \(S(U_j)=U_j^{-1}\) are forced, and
applying \(\Delta\) (resp.\ \(S\)) to
\(p\lambda X_j=U_j-G_{j+1}(X_0,\dots,X_{j-1})\)
and using \(p\lambda\in K^\times\) forces exactly the displayed recursive formulas for
\(\Delta(X_j)\) and \(S(X_j)\).

\medskip\noindent
\emph{Proof of \ref{prop:Hopf-over-Kb}.}
The construction is the same. Declare the split torus Hopf structure on $T_K^{\mathcal{W}}$
\[
\Delta_K^{\mathcal{W}}(Z_j)=Z_j\otimes Z_j,\qquad
\varepsilon_K^{\mathcal{W}}(Z_j)=1,\qquad
S_K^{\mathcal{W}}(Z_j)=Z_j^{-1}.
\]
Define \(\varepsilon_K^{\mathcal{W}}(Y_j)=0\).
Using the relations
\[
\lambda Y_0=Z_0-1,\qquad
\lambda Y_j=Z_j-E_{j+1}(Y_0,\dots,Y_{j-1})\quad (1\le j\le s-1)
\]
and the fact that \(E_{j+1}(0,\dots,0)=1\), define \(\Delta_K^{\mathcal{W}}(Y_j)\) and
\(S_K^{\mathcal{W}}(Y_j)\) by the displayed recursive formulas. The verifications of the Hopf
axioms and uniqueness are identical to \ref{prop:Hopf-over-Ka}, replacing \(p\lambda,G_{j+1},U_j,X_j\) by
\(\lambda,E_{j+1},Z_j,Y_j\).
\end{proof}

\begin{proposition}\label{prop:compat-group-law}
\begin{enumerate}[label=\textup{(\alph*)},leftmargin=2.2em]
\item\label{prop:compat-group-lawa}
Let $(\Delta^{\mathcal{V}}_K,\varepsilon^{\mathcal{V}}_K,S^{\mathcal{V}}_K)$ be the Hopf $K$-algebra structure on
$A_K=\mathcal{O}(\mathcal{V}_{s,K})$ from Proposition~\ref{prop:Hopf-over-K}\ref{prop:Hopf-over-Ka},
and let
\[
\mu^{\mathcal{V}}_{\mathrm{Hopf}}\colon \mathcal{V}_{s,K}\times_K \mathcal{V}_{s,K}\to \mathcal{V}_{s,K},
\qquad
\iota^{\mathcal{V}}_{\mathrm{Hopf}}\colon \mathcal{V}_{s,K}\to \mathcal{V}_{s,K}
\]
be the induced group law and inverse.

Let $\widehat{A}$ be the $(X_0,\dots,X_{s-1})$-adic completion of $A=\mathcal{O}(\mathcal{V}_s)$, so that
$\Spf(\widehat{A})=\widehat{\mathcal{V}}_{s,R}$ is the formal completion of $\mathcal{V}_s$ along the identity section.

\smallskip
\noindent
(i) Completing $(\Delta^{\mathcal{V}}_K,\varepsilon^{\mathcal{V}}_K,S^{\mathcal{V}}_K)$ at the identity yields a completed Hopf $K$-algebra structure
\[
(\widehat{\Delta}^{\mathcal{V}}_{K},\widehat{\varepsilon}^{\mathcal{V}}_{K},\widehat{S}^{\mathcal{V}}_{K})
\quad\text{on }\widehat{A}_K:=\widehat{A} \otimes_R K,
\]
where $\widehat{\Delta}^{\mathcal{V}}_{K}\colon \widehat{A}_K\to
\widehat{A}_K\widehat{\otimes}_{K}\widehat{A}_K$.

\smallskip
\noindent
(ii) Let
\[
\widehat{\tau}\colon \widehat{\mathcal{V}}_{s,R}\stackrel{\sim}{\longrightarrow}\widehat{W}_{s,R}
\]
denote the restriction of the coordinate map $\tau$ (cf.\ Lemma~\ref{lem:tau-gamma-formal-iso}) to the formal completions.
Pulling back Witt addition (and inverse) along $\widehat{\tau}$ yields a formal group law
$(\widehat{\mu}_\tau,\widehat{\iota}_\tau)$ on $\widehat{\mathcal{V}}_{s,R}$, equivalently a completed Hopf $R$-algebra structure
\[
(\widehat{\Delta}_\tau,\widehat{\varepsilon}_\tau,\widehat{S}_\tau)
\quad\text{on }\widehat{A}.
\]
Using the canonical identification
\[
(\widehat{A}\widehat{\otimes}_R \widehat{A})\otimes_R K
\;\cong\;
\widehat{A}_K\widehat{\otimes}_K \widehat{A}_K
\]
(where the completion is with respect to
$J_K=(X_0,\dots,X_{s-1})\otimes 1 + 1\otimes (X_0,\dots,X_{s-1})$),
we may view
$(\widehat{\Delta}_\tau,\widehat{\varepsilon}_\tau,\widehat{S}_\tau)\otimes_R K$
as a completed Hopf $K$-algebra structure on $\widehat{A}_K$. Then on $\widehat{A}_K$
\[
(\widehat{\Delta}^{\mathcal{V}}_{K},\widehat{\varepsilon}^{\mathcal{V}}_{K},\widehat{S}^{\mathcal{V}}_{K})
=
(\widehat{\Delta}_\tau,\widehat{\varepsilon}_\tau,\widehat{S}_\tau)\otimes_R K.
\]
equivalently the formal completion of $(\mu^{\mathcal{V}}_{\mathrm{Hopf}},\iota^{\mathcal{V}}_{\mathrm{Hopf}})$
at the identity coincides with $(\widehat{\mu}_\tau,\widehat{\iota}_\tau)$.

\item\label{prop:compat-group-lawb}
The analogous statement holds for $\mathcal{W}_{s}$, replacing
$\tau,G_{j+1},U_j,A_K,p\lambda$ by $\gamma,E_{j+1},Z_j,B_K,\lambda$.
\end{enumerate}
\end{proposition}

\begin{proof}
We prove \ref{prop:compat-group-lawa}; the proof of \ref{prop:compat-group-lawb} is identical with the indicated replacements.

Recall that $\widehat{A}$ is the $I_{\mathcal V}$-adic completion of $A$, and
$\widehat{A}\cong R\llbracket X_0,\ldots,X_{s-1}\rrbracket$.

\smallskip\noindent
\emph{Completion of the Hopf structure over $K$.}
Let $I_{\mathcal V,K}:=\ker(\varepsilon^{\mathcal{V}}_K)=(X_0,\ldots,X_{s-1})\subset A_K$.
Since $(A_K,\Delta^{\mathcal{V}}_K,\varepsilon^{\mathcal{V}}_K,S^{\mathcal{V}}_K)$ is a Hopf $K$-algebra, we have
\[
\Delta^{\mathcal{V}}_K(I_{\mathcal V,K})\subset I_{\mathcal V,K}\otimes_K A_K \;+\; A_K\otimes_K I_{\mathcal V,K},
\qquad
S^{\mathcal{V}}_K(I_{\mathcal V,K})\subseteq I_{\mathcal V,K}.
\]
Hence $\Delta^{\mathcal{V}}_K$ and $S^{\mathcal{V}}_K$ are continuous for the $I_{\mathcal V,K}$-adic topology, and $\varepsilon^{\mathcal{V}}_K$ kills $I_{\mathcal V,K}^n$ for all $n$.
Therefore $\Delta^{\mathcal{V}}_K,\varepsilon^{\mathcal{V}}_K,S^{\mathcal{V}}_K$ extend uniquely to continuous maps
\[
\widehat{\Delta}^{\mathcal{V}}_K\colon \widehat{A}_K\to \widehat{A}_K\widehat{\otimes}_K\widehat{A}_K,
\qquad
\widehat{\varepsilon}^{\mathcal{V}}_K\colon \widehat{A}_K\to K,
\qquad
\widehat{S}^{\mathcal{V}}_K\colon \widehat{A}_K\to \widehat{A}_K,
\]
giving a completed Hopf $K$-algebra structure $(\widehat{\Delta}^{\mathcal{V}}_K,\widehat{\varepsilon}^{\mathcal{V}}_K,\widehat{S}^{\mathcal{V}}_K)$ on $\widehat{A}_K$.

\smallskip\noindent
\emph{Pullback of Witt addition.}
By Proposition~\ref{prop:formal-group-structure}, the restriction of $\tau$ to the formal completions is an isomorphism
\[
\widehat{\tau}\colon \widehat{\mathcal{V}}_{s,R}\stackrel{\sim}{\longrightarrow}\widehat{W}_{s,R}.
\]
Transporting Witt addition (and inverse) along $\widehat{\tau}$ yields the completed Hopf $R$-algebra structure $(\widehat{\Delta}_\tau,\widehat{\varepsilon}_\tau,\widehat{S}_\tau)$ on $\widehat{A}$; base-changing to $K$ gives the corresponding completed Hopf $K$-algebra structure $(\widehat{\Delta}_\tau,\widehat{\varepsilon}_\tau,\widehat{S}_\tau) \otimes_R K$ on $\widehat{A}_K$.

\smallskip
\noindent\emph{Claim.} For each $0\le j\le s-1$,
\[
\widehat{\Delta}_\tau(U_j)=U_j\widehat{\otimes}_R U_j,\qquad
\widehat{\varepsilon}_\tau(U_j)=1,\qquad
\widehat{S}_\tau(U_j)=U_j^{-1}.
\]

\smallskip
\noindent\emph{Proof of the claim.}
Let $S$ be a complete local $R$-algebra and take $x_1,x_2\in \widehat{\mathcal{V}}_{s,R}(S)$.
Set $x_3:=\widehat{\mu}_\tau(x_1,x_2)$. By definition,
\[
\widehat{\tau}(x_3)=\widehat{\tau}(x_1)+\widehat{\tau}(x_2)\qquad\text{in }W_s(S),
\]
where $+$ denotes Witt addition. Writing $\widehat{\tau}(x_i)=(X'_{0,i},\dots,X'_{s-1,i})$,
the defining property of $\widehat{\tau}$ gives
\[
U_j(x_i)=G_{j+1,p}\bigl(X'_{0,i},\dots,X'_{j,i}\bigr)\in S^\times.
\]
Indeed, $U_j=\exp(p\lambda X_j')\,G_{j+1,p}(X_0',\ldots,X_{j-1}',0)=G_{j+1,p}(X_0',\ldots,X_j')$ by \eqref{eqnGscommutative}.
Using the multiplicativity of $G_{j+1,p}$ with respect to Witt addition (cf.\ \eqref{eqnGscommutative}) yields
$U_j(x_3)=U_j(x_1)U_j(x_2)$, which is exactly the group-like identity
$\widehat{\Delta}_\tau(U_j)=U_j\widehat{\otimes}U_j$.

For the counit, let $e$ be the identity section. Since $x=\widehat{\mu}_\tau(x,e)$ for all $x$,
the multiplicativity with $x_2=e$ gives $U_j(x)=U_j(x)\,U_j(e)$ for all $x$, hence $U_j(e)=1$.
By definition of the counit, $\widehat{\varepsilon}_\tau(U_j)=U_j(e)=1$.

For the antipode, let $\widehat{\iota}_\tau$ be the inverse. Since $\widehat{\mu}_\tau(x,\widehat{\iota}_\tau(x))=e$,
multiplicativity gives $1=U_j(e)=U_j(x)\,U_j(\widehat{\iota}_\tau(x))$, hence
$U_j(\widehat{\iota}_\tau(x))=U_j(x)^{-1}$ for all $x$. This is exactly $\widehat{S}_\tau(U_j)=U_j^{-1}$.
\qed

\smallskip
Now compare the two completed Hopf structures on $\widehat{A}_K$.
The completed Hopf $K$-algebra $(\widehat{\Delta}^{\mathcal{V}}_K,\widehat{\varepsilon}^{\mathcal{V}}_K,\widehat{S}^{\mathcal{V}}_K)$ on $\widehat{A}_K$
has each $U_j$ group-like by construction, and the base-changed structure
$(\widehat{\Delta}_\tau,\widehat{\varepsilon}_\tau,\widehat{S}_\tau)\otimes_R K$ also has each $U_j$ group-like by the claim.
Recall that for $j=0,\ldots,s-1$ we have in $\widehat{A}_K$
\[
X_j=\frac{U_j-G_{j+1}(X_0,\ldots,X_{j-1})}{p\lambda},
\]
and note that $p\lambda\in K^\times$.
We show by induction on $j$ that the two Hopf $K$-algebra structures coincide on $X_j$ (for $\widehat{\Delta}$, $\widehat{\varepsilon}$, and $\widehat{S}$).
For $j=0$ this follows from the displayed identity and the fact that both structures agree on $U_0$ (and send $1$ to $1$).
Assume the structures coincide on $X_0,\ldots,X_{j-1}$.
Since $G_{j+1}$ is a polynomial in $X_0,\ldots,X_{j-1}$ with coefficients in $K$, functoriality of $\widehat{\Delta}$, $\widehat{\varepsilon}$, and $\widehat{S}$ with respect to polynomials implies they also coincide on $G_{j+1}(X_0,\ldots,X_{j-1})$.
Together with the fact that both structures send $U_j$ to a group-like element, the displayed identity then forces the two structures to coincide on $X_j$.
Hence they coincide on all generators $X_0,\ldots,X_{s-1}$, and therefore on $\widehat{A}_K=K\llbracket X_0,\ldots,X_{s-1}\rrbracket$.
\end{proof}

\begin{lemma}[Intersection lemma]\label{lem:intersection}
\begin{enumerate}[label=\textup{(\alph*)},leftmargin=2.2em]
\item\label{lem:intersectiona}
Let $\widehat A$ be the $(X_0,\dots,X_{s-1})$-adic completion of $A$, and set
$S_A:=\widehat A_K\widehat\otimes_K \widehat A_K$.
Then:
\begin{enumerate}[label=\textup{(\roman*)}]
\item inside $\widehat A_K=\widehat A\otimes_R K$, one has $A_K\cap \widehat A = A$;
\item inside $S_A$, one has $(A_K\otimes_K A_K)\ \cap\ (\widehat A\widehat\otimes_R \widehat A) = A\otimes_R A$.
\end{enumerate}

\item\label{lem:intersectionb}
Let $\widehat B$ be the $(Y_0,\dots,Y_{s-1})$-adic completion of $B$, and set
$S_B:=\widehat B_K\widehat\otimes_K \widehat B_K$.
Then:
\begin{enumerate}[label=\textup{(\roman*)}]
\item inside $\widehat B_K=\widehat B\otimes_R K$, one has $B_K\cap \widehat B = B$;
\item inside $S_B$, one has $(B_K\otimes_K B_K)\ \cap\ (\widehat B\widehat\otimes_R \widehat B) = B\otimes_R B$.
\end{enumerate}
\end{enumerate}
\end{lemma}

\begin{proof}
We prove \ref{lem:intersectiona}; the proof of \ref{lem:intersectionb} is identical (see the end).

\smallskip

\noindent\emph{Step 0: identify the completions.}
By the discussion of coordinate rings and completions above, the $I_{\mathcal V}$-adic completion $\widehat A$ identifies with
\[
\widehat A \cong R\llbracket X_0,\ldots,X_{s-1}\rrbracket,
\qquad
\widehat A_K=\widehat A\otimes_R K \cong K\llbracket X_0,\ldots,X_{s-1}\rrbracket.
\]

\smallskip

\noindent\emph{Proof of \ref{lem:intersectiona}(i).}
Work inside $\widehat A_K\cong K\llbracket X_0,\ldots,X_{s-1}\rrbracket$.
Let $f\in A_K\cap \widehat A$.
Since $A_K$ is the localization of $K[X_0,\ldots,X_{s-1}]$ at the multiplicative set generated by $U_0,\ldots,U_{s-1}$, we may (after clearing denominators) write $f$ in $A_K$ as
\[
f=\frac{g}{u},
\qquad
g\in K[X_0,\ldots,X_{s-1}],
\quad u=\prod_{j=0}^{s-1}U_j^{n_j}\ \text{for some }n_j\ge 0.
\]
Then $u\in R[X_0,\ldots,X_{s-1}]$ and $u\equiv 1\ (\mathrm{mod}\ I_{\mathcal V})$.
Hence $u$ is a unit in $\widehat A$.
Since $f\in \widehat A$, we get
\[
g=f\cdot u\in \widehat A \cong R\llbracket X_0,\ldots,X_{s-1}\rrbracket.
\]
But $g$ is a polynomial with coefficients in $K$, so the condition
$g\in R\llbracket X_0,\ldots,X_{s-1}\rrbracket$ forces all coefficients of $g$
to lie in $R$, i.e. $g\in R[X_0,\ldots,X_{s-1}]$.
Therefore $f=g/u\in A$.
The reverse inclusion $A\subseteq A_K\cap \widehat A$ is obvious, so
$A_K\cap \widehat A=A$.

\smallskip

\noindent\emph{Proof of \ref{lem:intersectiona}(ii).}
Let $X_i^{(1)}:=X_i\otimes 1$ and $X_i^{(2)}:=1\otimes X_i$ in $A\otimes_R A$, and set
\[
J_A:=(X_0^{(1)},\ldots,X_{s-1}^{(1)},X_0^{(2)},\ldots,X_{s-1}^{(2)}).
\]
As above, $U_j^{(1)}:=U_j\otimes 1$ and $U_j^{(2)}:=1\otimes U_j$ satisfy
$U_j^{(1)}\equiv 1\equiv U_j^{(2)}\ (\mathrm{mod}\ J_A)$, hence are units in the
$J_A$-adic completion. Consequently,
\[
\widehat A\widehat\otimes_R \widehat A
\cong R\llbracket X_0^{(1)},\ldots,X_{s-1}^{(1)},X_0^{(2)},\ldots,X_{s-1}^{(2)}\rrbracket,
\]
and after base change,
\[
S_A=\widehat A_K\widehat\otimes_K \widehat A_K
\cong K\llbracket X_0^{(1)},\ldots,X_{s-1}^{(1)},X_0^{(2)},\ldots,X_{s-1}^{(2)}\rrbracket.
\]
Inside $S_A$ we have
\[
A_K\otimes_K A_K
=
K[X_0^{(1)},\ldots,X_{s-1}^{(1)},X_0^{(2)},\ldots,X_{s-1}^{(2)}]
\bigl[(U_0^{(1)})^{-1},\ldots,(U_{s-1}^{(1)})^{-1},(U_0^{(2)})^{-1},\ldots,(U_{s-1}^{(2)})^{-1}\bigr].
\]

Now let
\[
f\in (A_K\otimes_K A_K)\cap(\widehat A\widehat\otimes_R \widehat A)
\subseteq S_A.
\]
As above, since $A_K\otimes_K A_K$ is a localization of the polynomial ring
$K[\underline{X}^{(1)},\underline{X}^{(2)}]$ at the multiplicative set generated by the units
$U_j^{(1)}$ and $U_j^{(2)}$, we may (after clearing denominators) write
\[
f=\frac{g}{u},
\qquad
g\in K[X_0^{(1)},\ldots,X_{s-1}^{(1)},X_0^{(2)},\ldots,X_{s-1}^{(2)}],
\]
with
\[u=\prod_{j=0}^{s-1}(U_j^{(1)})^{n_j}\prod_{j=0}^{s-1}(U_j^{(2)})^{m_j}
\quad\text{for some }n_j,m_j\ge 0.
\]
Then $u\in R[\underline{X}^{(1)},\underline{X}^{(2)}]$ and $u\equiv 1\ (\mathrm{mod}\ J_A)$,
so $u$ is a unit in $\widehat A\widehat\otimes_R \widehat A$.
Since $f\in \widehat A\widehat\otimes_R \widehat A$, we get
\[
g=f\cdot u \in \widehat A\widehat\otimes_R \widehat A
\cong R\llbracket X_0^{(1)},\ldots,X_{s-1}^{(1)},X_0^{(2)},\ldots,X_{s-1}^{(2)}\rrbracket.
\]
But $g$ is a polynomial over $K$, hence the above containment forces all its
coefficients to lie in $R$, i.e.
\[
g\in R[X_0^{(1)},\ldots,X_{s-1}^{(1)},X_0^{(2)},\ldots,X_{s-1}^{(2)}].
\]
Therefore $f=g/u$ lies in
\[
R[\underline{X}^{(1)},\underline{X}^{(2)}]
\bigl[(U_0^{(1)})^{-1},\ldots,(U_{s-1}^{(1)})^{-1},(U_0^{(2)})^{-1},\ldots,(U_{s-1}^{(2)})^{-1}\bigr]
=A\otimes_R A.
\]
The reverse inclusion $A\otimes_R A\subseteq (A_K\otimes_K A_K)\cap(\widehat A\widehat\otimes_R \widehat A)$
is clear, hence equality holds.

\smallskip

\noindent\emph{Proof of \ref{lem:intersectionb}.}
Replace $(X_j,U_j,A,\widehat A)$ by $(Y_j,Z_j,B,\widehat B)$ everywhere.
The only input needed is that each $Z_j$ has constant term $1$, so
$Z_j\equiv 1\ (\mathrm{mod}\ (Y_0,\ldots,Y_{s-1}))$ and thus becomes a unit in the
$(Y_0,\ldots,Y_{s-1})$-adic completion; then the same argument applies verbatim.
\end{proof}

\begin{proposition}[Descent from $K$ to $R$]\label{prop:descent-alt}
\begin{enumerate}[label=\textup{(\alph*)},leftmargin=2.2em]
\item\label{prop:descent-alta}
The Hopf $K$-algebra structure $(A_K,\Delta^{\mathcal{V}}_K,\varepsilon^{\mathcal{V}}_K,S^{\mathcal{V}}_K)$ on $\mathcal{O}(\mathcal{V}_{s,K})$ preserves the $R$-lattice $A$:
\[
\Delta^{\mathcal{V}}_K(A)\subseteq A\otimes_R A,\qquad
\varepsilon^{\mathcal{V}}_K(A)\subseteq R,\qquad
S^{\mathcal{V}}_K(A)\subseteq A.
\]

\item\label{prop:descent-altb}
The Hopf $K$-algebra structure $(B_K,\Delta_K^{\mathcal{W}},\varepsilon_K^{\mathcal{W}},S_K^{\mathcal{W}})$ on $\mathcal{O}(\mathcal{W}_{s,K})$ preserves the $R$-lattice $B$:
\[
\Delta_K^{\mathcal{W}}(B)\subseteq B\otimes_R B,\qquad
\varepsilon_K^{\mathcal{W}}(B)\subseteq R,\qquad
S_K^{\mathcal{W}}(B)\subseteq B.
\]
\end{enumerate}
\end{proposition}

\begin{proof}
We prove \ref{prop:descent-alta}; the proof of \ref{prop:descent-altb} is formally identical (see below).

Recall $A=R[X_0,\ldots,X_{s-1}][U_0^{-1},\ldots,U_{s-1}^{-1}]$ and $A_K=A\otimes_R K$.

\smallskip
\noindent\emph{Step 1: The generators $X_j$.}
For each $0\le j\le s-1$, we claim $\Delta^{\mathcal{V}}_K(X_j)\in A\otimes_R A$.
Indeed, $\Delta^{\mathcal{V}}_K(X_j)\in A_K\otimes_K A_K$ since $\Delta^{\mathcal{V}}_K$ is a $K$-algebra map.

On the other hand, by Proposition~\ref{prop:compat-group-law}\ref{prop:compat-group-lawa}, the completed comultiplication
$\widehat{\Delta}^{\mathcal{V}}_K$ on $\widehat{A}_K$ coincides with the pullback of the formal group law $\widehat{\Delta}_\tau\otimes_R K$.
In particular, the image of $\Delta^{\mathcal{V}}_K(X_j)$ under the natural completion map
\[
A_K\otimes_K A_K \longrightarrow \widehat{A}_K\widehat{\otimes}_K\widehat{A}_K
\]
is precisely $\widehat{\Delta}^{\mathcal{V}}_K(X_j)$ (this is the defining compatibility of a continuous map with its completion).
Thus
\[
\widehat{\Delta}^{\mathcal{V}}_K(X_j)=\widehat{\Delta}_\tau(X_j)\in \widehat{A}\widehat{\otimes}_R\widehat{A}.
\]
Viewing both inside $S=\widehat{A}_K\widehat{\otimes}_K\widehat{A}_K$, we get
\[
\Delta^{\mathcal{V}}_K(X_j) \in (A_K\otimes_K A_K)\cap(\widehat{A}\widehat{\otimes}_R\widehat{A}),
\]
so Lemma~\ref{lem:intersection} yields $\Delta^{\mathcal{V}}_K(X_j)\in A\otimes_R A$.

\smallskip
\noindent\emph{Step 2: The units $U_j$ and $U_j^{-1}$.}
By construction of the Hopf structure over $K$ (Proposition~\ref{prop:Hopf-over-K}\ref{prop:Hopf-over-Ka}),
each $U_j$ is group-like, so
\[
\Delta^{\mathcal{V}}_K(U_j)=U_j\otimes U_j\in A\otimes_R A,\qquad
\varepsilon^{\mathcal{V}}_K(U_j)=1\in R,\qquad
S^{\mathcal{V}}_K(U_j)=U_j^{-1}\in A.
\]
Hence also $\Delta^{\mathcal{V}}_K(U_j^{-1})=(U_j\otimes U_j)^{-1}=U_j^{-1}\otimes U_j^{-1}\in A\otimes_R A$.

\smallskip
\noindent\emph{Step 3: Conclusion for all of $A$.}
Since $A$ is generated as an $R$-algebra by the elements $X_j$ and $U_j^{\pm1}$,
Steps 1--2 imply that $\Delta^{\mathcal{V}}_K$ sends each generator to $A\otimes_R A$; hence
$\Delta^{\mathcal{V}}_K(A)\subseteq A\otimes_R A$.
Similarly, $\varepsilon^{\mathcal{V}}_K(X_j)=0\in R$ and $\varepsilon^{\mathcal{V}}_K(U_j^{\pm1})=1\in R$, so
$\varepsilon^{\mathcal{V}}_K(A)\subseteq R$.

For the antipode, we already know $S^{\mathcal{V}}_K(U_j^{\pm1})\in A$.
For the generators $X_j$, Proposition~\ref{prop:compat-group-law}\ref{prop:compat-group-lawa} implies that
the completed antipode $\widehat S^{\mathcal{V}}_K$ on $\widehat A_K$ coincides with
$\widehat S_\tau\otimes_R K$. Since $\widehat S_\tau\colon \widehat A\to \widehat A$ is defined over $R$, we have
$\widehat S^{\mathcal{V}}_K(X_j)=\widehat S_\tau(X_j)\in \widehat A$.

Since $S^{\mathcal{V}}_K(X_j)\in A_K$ (as $S^{\mathcal{V}}_K\colon A_K\to A_K$), we compare it with its image in $\widehat A_K$.
Under $\widehat A_K\cong K\llbracket X_0,\ldots,X_{s-1}\rrbracket$, each $U_j$ has constant term $1$, hence is a unit; therefore the localization map
\[
A_K=K[X_0,\ldots,X_{s-1}][U_0^{-1},\ldots,U_{s-1}^{-1}] \longrightarrow \widehat A_K
\]
is injective (because $K[X_0,\ldots,X_{s-1}]\hookrightarrow K\llbracket X_0,\ldots,X_{s-1}\rrbracket$ is injective). Hence $S^{\mathcal{V}}_K(X_j)$ lies in $A_K\cap \widehat A$.
Lemma~\ref{lem:intersection} then forces $S^{\mathcal{V}}_K(X_j) \in A$.
Thus $S^{\mathcal{V}}_K(A)\subseteq A$.

\medskip
\noindent\emph{Proof of \ref{prop:descent-altb}.}
The proof is identical to \ref{prop:descent-alta} upon replacing $A, X_j, U_j$ with $B, Y_j, Z_j$ and invoking Proposition~\ref{prop:compat-group-law}\ref{prop:compat-group-lawb}.
In particular, for each $0\le j\le s-1$ we have $\Delta_K^{\mathcal{W}}(Y_j)\in B_K\otimes_K B_K$ and its image in
$S_B=\widehat{B}_K\widehat{\otimes}_K\widehat{B}_K$ equals the formally defined comultiplication on $\widehat{B}$ (pulled back via $\gamma$),
which lies in $\widehat{B}\widehat{\otimes}_R\widehat{B}$. Hence
\[
\Delta_K^{\mathcal{W}}(Y_j)\in (B_K\otimes_K B_K)\cap(\widehat{B}\widehat{\otimes}_R\widehat{B})
\subseteq S_B,
\]
and Lemma~\ref{lem:intersection}\,\ref{lem:intersectionb}(ii) gives $\Delta_K^{\mathcal{W}}(Y_j)\in B\otimes_R B$.
The antipode case is analogous using Lemma~\ref{lem:intersection}\,\ref{lem:intersectionb}(i).
\end{proof}

By Proposition~\ref{prop:descent-alt}, the Hopf $K$-algebra structures on $A_K$ and $B_K$ descend to Hopf $R$-algebra structures on $A$ and $B$.
From now on, let $(\Delta_A,\varepsilon_A,S_A)$ and $(\Delta_B,\varepsilon_B,S_B)$ denote these descended Hopf $R$-algebra structure maps (whose base change to $K$ recovers the structures of Proposition~\ref{prop:Hopf-over-K}).

Let us now consider the {\'e}tale $R$-morphism $\psi_s:\mathcal{W}_s\to\mathcal{V}_s$. 

\begin{lemma}\label{lemmapsigroupmorphism}
The morphism $\psi_s:\mathcal{W}_s\to\mathcal{V}_s$ is a homomorphism of group schemes over $R$.
\end{lemma}

\begin{proof}
Let
\[
\varphi:=\psi_s^\#\colon A\longrightarrow B
\]
be the induced $R$-algebra homomorphism.

We show that $\varphi$ is a morphism of Hopf $R$-algebras. Let $I_{\mathcal V}:=\ker(\varepsilon_A)\subset A$ and
$I_{\mathcal W}:=\ker(\varepsilon_B)\subset B$ be the augmentation ideals.
By Lemma~\ref{lem:intersection}\,\ref{lem:intersectiona}(i) and Lemma~\ref{lem:intersection}\,\ref{lem:intersectionb}(i),
the natural maps
\[
A\hookrightarrow \widehat{A},\qquad B\hookrightarrow \widehat{B}
\]
into the $I_{\mathcal V}$-adic and $I_{\mathcal W}$-adic completions are injective.

\smallskip
\noindent\emph{Claim.} One has $\varphi(I_{\mathcal V})\subset I_{\mathcal W}$.

\smallskip
\noindent\textit{Proof of the claim.}
We prove by induction on $j=0,1,\ldots,s-1$ that $\varphi(X_j)\in I_{\mathcal W}$.
Since $I_{\mathcal V}=(X_0,\ldots,X_{s-1})$, this implies $\varphi(I_{\mathcal V})\subset I_{\mathcal W}$.

\smallskip
\noindent\emph{Base case $j=0$.}
By construction of $\psi_s$ (cf.\ the explicit computation in \eqref{eqnASWtype1simplified}), we have
\[
\varphi(X_0)=S_0(Y_0)\in R\Bigl[Y_0,\frac{1}{1+\lambda Y_0}\Bigr]\subset B,
\]
and $S_0(Y_0)$ has zero constant term. Hence $\varphi(X_0)\in (Y_0)\subset I_{\mathcal W}$.

\smallskip
\noindent\emph{Inductive step.}
Fix $1\le j\le s-1$. Consider the analogue of \eqref{eqnminimalpoly} with $s-1$ replaced by $j$, i.e.
the defining relation for $X_j$, viewed in $B$ via $\varphi$:
\begin{equation}\label{eqnminimalpoly-j}
\frac{\lambda^{p-1}}{p}Y_j^p+Y_j+H_{j+1}(Y_0,\ldots,Y_j)
-\bigl(E_j(Y_0,\ldots,Y_{j-2})+\lambda Y_{j-1}\bigr)\,\varphi(X_j)=0.
\end{equation}
(Here $H_{j+1}$ and $E_j$ are the polynomials obtained by the same construction as in
Proposition~\ref{propminimialpoly}, with the usual convention that
$E_1(\,)=1$ when $j=1$.)

Reduce \eqref{eqnminimalpoly-j} modulo $I_{\mathcal W}=(Y_0,\ldots,Y_{s-1})$.
Then $Y_j\equiv 0$ and $Y_j^p\equiv 0$, and $H_{j+1}(Y_0,\ldots,Y_j)\equiv 0$
since $H_{j+1}\in (Y_0,\ldots,Y_j)$ (in particular $H_{j+1}(0,\ldots,0)=0$). Moreover,
\[
E_j(Y_0,\ldots,Y_{j-2})+\lambda Y_{j-1}\equiv E_j(0,\ldots,0)=1\pmod{I_{\mathcal W}}.
\]
Therefore \eqref{eqnminimalpoly-j} reduces to
\[
-\overline{\varphi(X_j)}=0\qquad\text{in }B/I_{\mathcal W},
\]
so $\varphi(X_j)\in I_{\mathcal W}$.
This completes the induction and proves the claim.
\qed

\smallskip
Since $\varphi(I_{\mathcal V})\subset I_{\mathcal W}$, the map $\varphi$ is continuous for the $I_{\mathcal V}$-adic
topology on $A$ and the $I_{\mathcal W}$-adic topology on $B$, hence induces a unique continuous map on completions
\[
\widehat{\varphi}\colon \widehat{A}\longrightarrow \widehat{B}.
\]
By the construction of $\psi_s$, its restriction to the formal completions along the identity sections satisfies
\[
\widehat{\psi}_s=\widehat{\tau}^{-1}\circ \widehat{\gamma},
\]
so on completed coordinate rings
\[
\widehat{\varphi}
=
\widehat{\gamma}^{\#}\circ (\widehat{\tau}^{\#})^{-1}
\colon \widehat{A}\longrightarrow \widehat{B}.
\]
By Proposition~\ref{prop:formal-group-structure} (applied to $\tau$) and its analogue for $\gamma$,
the maps $\widehat{\tau}$ and $\widehat{\gamma}$ are isomorphisms of formal group schemes over $R$; moreover,
by Proposition~\ref{prop:compat-group-law}, the Hopf structure on completed local rings agrees with the completion
of the algebraic Hopf structure. Consequently, $\widehat{\varphi}$ is a morphism of completed Hopf $R$-algebras:
\[
\widehat{\Delta}_B\circ \widehat{\varphi}
=
(\widehat{\varphi}\,\widehat{\otimes}\,\widehat{\varphi})\circ \widehat{\Delta}_A,
\qquad
\widehat{\varepsilon}_B\circ \widehat{\varphi}
=
\widehat{\varepsilon}_A.
\]

We now descend these identities to the algebraic level. Consider the two $R$-algebra maps
\[
\Delta_B\circ \varphi,
\quad (\varphi\otimes \varphi)\circ \Delta_A
\;\colon\;
A \longrightarrow B\otimes_R B.
\]
Let $I_{B\otimes B}:=I_{\mathcal W}\otimes_R B + B\otimes_R I_{\mathcal W}$, the augmentation ideal of $B\otimes_R B$.
After $I_{B\otimes B}$-adic completion, these two maps become equal because $\widehat{\varphi}$ respects comultiplication.

By Lemma~\ref{lem:intersection}\,\ref{lem:intersectionb}(ii), the natural map
\[
B\otimes_R B \longrightarrow \widehat{B}\,\widehat{\otimes}_R\,\widehat{B}
\]
is injective (viewing both rings inside $\widehat{B}_K\widehat{\otimes}_K\widehat{B}_K$).
Hence the two maps were already equal before completion:
\[
\Delta_B\circ \varphi = (\varphi\otimes \varphi)\circ \Delta_A.
\]
Similarly, the maps $\varepsilon_B\circ \varphi$ and $\varepsilon_A$ agree after completion, hence agree on $A$:
\[
\varepsilon_B\circ \varphi=\varepsilon_A.
\]
Therefore $\varphi$ is a morphism of Hopf $R$-algebras, and thus $\psi_s$ is a homomorphism of group schemes over $R$.
\end{proof}

\begin{lemma}\label{lem:kernel_finite_flat}
The kernel $\ker(\psi_s)$ is finite flat over $R$ of rank $p^s$.
\end{lemma}

\begin{proof}
Since $\psi_s$ is a homomorphism of group schemes, its kernel is the fiber over the identity section $e: \operatorname{Spec} R \to \mathcal{V}_s$:
\[
\ker(\psi_s) = \mathcal{W}_s \times_{\mathcal{V}_s, e} \operatorname{Spec} R.
\]
This is the base change of the finite flat morphism $\psi_s$ (Proposition~\ref{propetalereduction}) along $e$, hence $\ker(\psi_s)$ is finite flat over $R$.

For the rank, base change to the special fiber gives
\[
\ker(\psi_s)\otimes_R \mathbb{F}_p \cong \ker(\wp:W_s\to W_s),
\]
by Corollary~\ref{corspecialunr}. The latter kernel is the constant group scheme
$\underline{\mathbb{Z}/p^s} \cong W_s(\mathbb{F}_p)$, hence has rank $p^s$.
Since $R$ is a DVR, $\spec R$ is connected. For a finite flat $R$-scheme, the rank (equivalently, the degree) is locally constant on the base, hence constant on $\spec R$. Therefore, because the special fiber has rank $p^s$, we also have
$\operatorname{rank}_R \ker(\psi_s) = p^s$.
\end{proof}

\begin{proposition}\label{propisogenyoverR}
The homomorphism $\psi_s:\mathcal{W}_s\to\mathcal{V}_s$ is an isogeny of smooth
commutative group schemes over $R$; its kernel is finite flat of rank $p^s$.
\end{proposition}

\begin{proof}
By Lemma~\ref{lemmapsigroupmorphism}, the morphism $\psi_s$ is a homomorphism of group schemes over $R$.

Both $\mathcal{V}_s$ and $\mathcal{W}_s$ are smooth over $R$, since each is an open subscheme of an affine space over $R$ by definition.
Moreover, by Proposition~\ref{propetalereduction} the morphism $\psi_s$ is finite étale; in particular, it is finite. On affine coordinate rings this gives an integral extension
\[
A=\mathcal O(\mathcal V_s)\hookrightarrow B=\mathcal O(\mathcal W_s),
\]
and the map is injective since both embed in the same function field extension (with $B$ the integral closure of $A$ there). Hence, by lying-over, $\spec(B)\to\spec(A)$ is surjective.

Finally, Lemma~\ref{lem:kernel_finite_flat} shows that $\ker(\psi_s)$ is finite flat of rank $p^s$.
Therefore $\psi_s$ is an isogeny of smooth commutative group schemes over $R$.
\end{proof}

The proof of the following result is straightforward.

\begin{corollary}
    With the notations of Corollary \ref{corspecialequation}, the special-fiber morphism $(\psi_s)_k:\mathcal{W}_{s,k}\to\mathcal{V}_{s,k}$ is the Artin--Schreier--Witt isogeny
    \begin{equation*}
        W_s(\mathbb{F}_p) \longrightarrow W_s(\mathbb{F}_p), \quad (y_0, y_1+b_1, \ldots, y_{s-1}+b_{s-1}) \mapsto \wp(y_0, y_1+b_1, \ldots, y_{s-1}+b_{s-1}).
    \end{equation*}
\end{corollary}

\begin{proposition}
\label{propdiagramgroup}
On formal completions along the identity sections, we have a commutative diagram of morphisms of formal group schemes over $R$:
\begin{center}
\begin{tikzcd}
\widehat{\mathcal{W}}_s \arrow[r, "\widehat{\gamma}"] \ar[d, "\widehat{\psi}_s"] & \widehat{W}_s \ar[d, equal] \arrow[r, "\tilde{E}_{s,p}"] & \widehat{(\mathbb{G}_m)^s} \arrow[d, "\widehat{\phi}_s"] \\
\widehat{\mathcal{V}}_s \arrow[r, "\widehat{\tau}"] & \widehat{W}_s \arrow[r, "\tilde{\Psi}_s"] & \widehat{(\mathbb{G}_m)^s}
\end{tikzcd}
\end{center}
Here $\widehat{\phi}_s$ is the morphism on formal completions induced by $\phi_s$ from \eqref{eqphis}.
\end{proposition}

\begin{proof}
The commutativity of the first square follows from the description of the restriction of $\psi_s$ to formal completions,
\[
\widehat{\psi}_s=\widehat{\tau}^{-1}\circ \widehat{\gamma},
\]
(cf.\ Lemma~\ref{lem:tau-gamma-formal-iso}), equivalently $\widehat{\tau}\circ\widehat{\psi}_s=\widehat{\gamma}$.

For the second square, we define the maps $\tilde{E}_{s,p}$ and $\tilde{\Psi}_{s}$ on $\widehat{W}_s$ by
\[
\begin{split}
    (Y_0', \ldots, Y_{s-1}') &\mapsto (E_{1,p}((Y_0')), \ldots, E_{s,p}((Y_0', \ldots, Y_{s-1}'))), \\
    (X_0', \ldots, X_{s-1}') &\mapsto (G_{1,p}((X_0')), \ldots, G_{s,p}((X_0', \ldots, X_{s-1}'))).
\end{split}
\]
The fact that these are homomorphisms of (formal) groups follows from \eqref{eqnEscommutative} and \eqref{eqnGscommutative}.

Moreover, by construction of the formal coordinates (cf.\ \eqref{eqngamma} and \eqref{eqntau}), we have
\[
Z_i = E_{i+1,p}((Y_0',\ldots,Y_i'))
\qquad\text{and}\qquad
U_i
= G_{i+1,p}((X_0',\ldots,X_i'))
\]
for $i=0,\ldots,s-1$.

Recall that $\phi_s$ is defined in \eqref{eqphis}. The commutativity of the second square follows from
\[
    \left(E_{i+1,p}((X_0', \ldots, X_i'))\right)^p
    \left(E_{i,p}((X_0', \ldots, X_{i-1}'))\right)^{-1}
    = G_{i+1,p}((X_0', \ldots, X_i')).
\]
This is exactly the defining relation of $G_{i+1,p}$ in terms of $E_{i+1,p}$
(see the definition preceding \eqref{eqnGscommutative}), together with
$E_{i+1,p}(p\,\underline{a})=E_{i+1,p}(\underline{a})^p$ for Witt multiplication by $p$.
This completes the proof.
\end{proof}

\begin{proposition}
\label{propdiagramgroup-algebraic}
There is a commutative diagram of homomorphisms of algebraic group schemes over $R$:
\begin{center}
\begin{tikzcd}[column sep=large]
\mathcal{W}_s \arrow[r, "{(Z_0,\ldots,Z_{s-1})}"] \ar[d, "\psi_s"] & (\mathbb{G}_m)^s \arrow[d, "\phi_s"] \\
\mathcal{V}_s \arrow[r, "{(U_0,\ldots,U_{s-1})}"] & (\mathbb{G}_m)^s
\end{tikzcd}
\end{center}
where $U_0,\ldots,U_{s-1}\in \mathcal{O}(\mathcal{V}_s)^\times$ and $Z_0,\ldots,Z_{s-1}\in \mathcal{O}(\mathcal{W}_s)^\times$ are the standard units from the definitions of $\mathcal{V}_s$ and $\mathcal{W}_s$.
Here $\phi_s$ is as in \eqref{eqphis}.
This diagram is the algebraic version of Proposition~\ref{propdiagramgroup}: one replaces the formal torus coordinates by the corresponding global units $(Z_0,\ldots,Z_{s-1})$ and $(U_0,\ldots,U_{s-1})$.
\end{proposition}

\begin{proof}
The horizontal arrows are induced by the tuples of units $(Z_0,\ldots,Z_{s-1})$ on $\mathcal{W}_s$
and $(U_0,\ldots,U_{s-1})$ on $\mathcal{V}_s$.
Moreover, by construction of the descended Hopf algebra structures on $\mathcal{O}(\mathcal{W}_s)$ and
$\mathcal{O}(\mathcal{V}_s)$, each $Z_i$ and $U_i$ is group-like (i.e.\ $\Delta_B(Z_i)=Z_i\otimes Z_i$
and $\Delta_A(U_i)=U_i\otimes U_i$), hence these maps to $(\mathbb{G}_m)^s$ are homomorphisms of group schemes.

For the left square, by construction of $\psi_s$ (cf.\ \eqref{eqnKummer2}), in $\mathcal{O}(\mathcal{W}_s)^\times$ we have
\[
U_0=Z_0^p,\qquad U_j=Z_j^p Z_{j-1}^{-1}\ \ (1\le j\le s-1).
\]
Equivalently, $(U_0,\ldots,U_{s-1})=\phi_s(Z_0,\ldots,Z_{s-1})$, so the square commutes.
\end{proof}

This concludes the proof of Theorem \ref{theoremuniversalKASW}. \qed

\section{The Universality}
\label{secuniversality}
In this section, we establish the proof of Theorem \ref{theoremuniversality}. Our methodology follows standard procedures as outlined in \cite[Section 3]{SS:kasw2} and \cite[Section 6.4]{MR3051249}. To begin, we fix some notation. Let $R$ be a finite extension of $\mathbb{Z}_{(p)}[\zeta_{p^s}]$, and let $\pi$ be its uniformizer. Define \(R_{\pi} := R/\pi\), and denote by \(\mathbb{G}_{m,R_\pi}\) the special fiber of \(\mathbb{G}_m := \operatorname{Spec} R\left[X, \frac{1}{X} \right]\). We have a natural closed immersion $\mathbb{G}_{m,R_\pi} \hookrightarrow \mathbb{G}_m$ denoted by $\iota$.

For the computation of \(\mathcal{W}_1 = \operatorname{Spec} R\left[ Y_0, \frac{1}{1+\lambda Y_0} \right]\), the following exact sequence of sheaves on the {\'e}tale site of \(\operatorname{Spec} R\) is particularly useful (see \cite[Lemma 1.1]{MR1011987}):
\begin{align}\label{W1sequence}
    0 \longrightarrow \mathcal{W}_1 \xlongrightarrow{\alpha} \mathbb{G}_{m,R} \xlongrightarrow{\beta} \iota_{*} \mathbb{G}_{m, R_{\pi}} \longrightarrow 0 
\end{align}
where \(\alpha: x \mapsto 1 + \lambda x\) and \(\beta: t \mapsto t \mod \pi\). The following result is derived from \cite[Theorem 2.2]{SS:kasw2}.

\begin{lemma}
     Suppose $R$ contains the $p^s$-th root of unity. If $B$ is a flat local $R$-algebra, then 
     \begin{align*}
         \cohom^1(\spec B, \mathcal{W}_1) = 0.
     \end{align*}
\end{lemma}

\begin{proof}
By considering the long exact sequence associated with \eqref{W1sequence} relative to \(\spec B\), we obtain the following:
\begin{align*}
    0 \longrightarrow \mathcal{W}_1(B) \longrightarrow \mathbb{G}_m(B) &= B^* \longrightarrow \iota_*\mathbb{G}_{m,R_{\pi}}(B) = (B/(\pi B))^* \longrightarrow \\
    &\longrightarrow \cohom^1(\spec B, \mathcal{W}_1) \longrightarrow \cohom^1(\spec B, \mathbb{G}_m) \longrightarrow \cdots
\end{align*}

By Hilbert's Theorem 90, we know that \(\cohom^1(\spec B, \mathbb{G}_m) = 0\). Given that the map \(f: B^* \to (B/(\pi B))^*\) is surjective, the exactness of the long exact sequence implies that \(\cohom^1(\spec B, \mathcal{W}_1) \cong \operatorname{coker}(f) = 0\).
\end{proof}

We can employ induction by leveraging the following result.

\begin{lemma}
    For any $s > 1$, the following sequence of group schemes
    \begin{align}\label{W_s_exact}
        0 \longrightarrow \mathcal{W}_1 \longrightarrow \mathcal{W}_s \xlongrightarrow{} \mathcal{W}_{s-1} \longrightarrow 0
    \end{align}
    is exact and it fits into the following commutative diagram:
    \begin{center}
         \begin{tikzcd}
        0 \ar[r] & \mathbb{Z}/p \ar[r] \ar[d] & \mathbb{Z}/p^s \ar[r] \ar[d] & \mathbb{Z}/p^{s-1} \ar[r] \ar[d] & 0 \\
        0 \ar[r]  & \mathcal{W}_1 \ar[r] & \mathcal{W}_s \ar[r]  & \mathcal{W}_{s-1} \ar[r]  & 0
    \end{tikzcd}
    \end{center}
\end{lemma}

\begin{proof}
    The morphism \(\mathcal{W}_s \to \mathcal{W}_{s-1}\) arises from the inclusion of the \(R\)-algebras that define the schemes, with the kernel given by
    \[
    \spec R \left[ Y_{s-1}, \frac{1}{1+\lambda Y_{s-1}} \right] \cong \mathcal{W}_1.
    \]
    Moreover, the Artin--Schreier--Witt theory establishes the desired commutative diagram.
\end{proof}

By induction, we obtain the following result.

\begin{proposition}
    Suppose \(R\) contains the \(p^s\)-th root of unity. If \(B\) is a flat local \(R\)-algebra, then for every $s \geq 1$, we obtain
    \begin{align*}
        \cohom^1(\spec B,\mathcal{W}_s) = 0.
    \end{align*}
\end{proposition}

\begin{proof}
    Assume the induction hypothesis that $\cohom^1(\spec B,\mathcal{W}_{s-1}) =0$. We claim that $\cohom^1(\spec B, \mathcal{W}_s)$ vanishes. Taking the long exact sequence associated with \eqref{W_s_exact} with respect to \(\spec B\) yields:
\begin{multline*}
    0 \longrightarrow \mathcal{W}_1(B) \longrightarrow \mathcal{W}_s(B) \longrightarrow \mathcal{W}_{s-1}(B) \longrightarrow \cohom^1(\spec B, \mathcal{W}_1) \longrightarrow \\
    \longrightarrow \cohom^1(\spec B, \mathcal{W}_s) \longrightarrow \cohom^1(\spec B, \mathcal{W}_{s-1}) \longrightarrow \cdots
\end{multline*}
    Since $\cohom^1(\spec B, \mathcal{W}_1) = \cohom^1(\spec B, \mathcal{W}_{s-1}) = 0$, we get the desired result.
\end{proof}

We will now proceed to prove our main theorem, which states that the group scheme \(\mathcal{W}_{s}\) over \(R\) provides a unified framework for the Kummer--Artin--Schreier--Witt theory.

\begin{theorem}
    Suppose $R$ contains the $p^s$-th root of unity. If $B$ is a flat local $R$-algebra, then any unramified $\mathbb{Z}/p^s$-cover $\spec C \rightarrow \spec B$ is given by a Cartesian diagram of the form:
    \begin{center}
    \begin{tikzcd}
        \spec C \arrow{r} \ar[d] &   \mathcal{W}_s \ar[d, "\psi_s"] \\ 
        \spec B \arrow{r}  &  \mathcal{V}_s
    \end{tikzcd}
    \end{center}
\end{theorem}

\begin{proof}
    Taking the long exact sequence from the Kummer sequence:
    \begin{align*}
        0 \longrightarrow \mathbb{Z}/p^s \longrightarrow \mathcal{W}_s \longrightarrow \mathcal{V}_s \longrightarrow 0
    \end{align*}
    yields
    \begin{align*}
        0 \longrightarrow \mathcal{V}_s(B)/\mathcal{W}_s(B) \longrightarrow \mathrm{H}^1(\spec B, \mathbb{Z}/p^s) \longrightarrow \mathrm{H}^1(\spec B, \mathcal{W}_s) = 0.
    \end{align*}
    It implies that \(\mathbb{Z}/p^s\)-torsors over \(\spec B\) are in one-to-one correspondence with elements of
    \(\mathcal{V}_s(B)/\psi_s(\mathcal{W}_s(B))\).
\end{proof}

\section{Matsuda’s Theory and the Refined Swan Conductor}
\label{secrefinedswan}

Let $\mathbb{K}$ be a complete discretely valued field with residue field $\kappa$ of characteristic $p>0$ and uniformizer $\pi$. Let $e$ be the absolute ramification index of $\mathbb{K}$. When $\kappa$ is perfect, classical ramification theory studies the behavior of finite Galois extensions in great detail. One important invariant arising from this theory is the Swan conductor \cite[Chapter IV]{MR554237}, which measures the degree of wild ramification and has many applications in algebraic geometry.

For example, let $X$ be a curve over an algebraically closed field $\kappa$, let $U \subset X$ be a nonempty open subset, and let $\mathscr{F}$ be a $\mathbb{Q}_{\ell}$-sheaf on $U$. Then, for each closed point $x \in X \setminus U$, the Swan conductor $\mathrm{sw}_x(\mathscr{F})$ is a nonnegative integer measuring the wild ramification of the associated extension of $\widehat{\mathcal{O}}_{X,x} \cong \kappa \llbracket t \rrbracket$. The Grothendieck--Ogg--Shafarevich formula can be expressed in terms of $\mathrm{sw}_x(\mathscr{F})$ as follows:
\begin{equation}
\label{eqnGOS}
    \chi_{c}(U_{\overline{\kappa}}, \mathscr{F})
= \operatorname{rank}(\mathscr{F}) \cdot \chi_{c}(U_{\overline{\kappa}}, \mathbb{Q}_{\ell})
- \sum_{x \in X \setminus U} \mathrm{sw}_x(\mathscr{F}).
\end{equation}

When the residue field is not perfect, there is an analogous notion of the
\emph{refined Swan conductor}, which may be viewed as the Swan conductor together with additional differential data. Note that the Swan conductor vanishes precisely when there is no wild ramification; in particular, in our setting this implies that the associated residue field extension is separable. When the Swan conductor is positive, one can write
\begin{equation*}
    \mathrm{rsw}(\mathscr{F}) = \pi^{-e\,\mathrm{sw}(\mathscr{F})} \otimes \omega
    \in \mathfrak{p}_{\mathbb{K}}^{-\mathrm{sw}(\mathscr{F})} \otimes \Omega^1_{\kappa}.
\end{equation*}
Recall that
\[
\mathfrak{p}_{\mathbb{K}}^{t} := \{x\in \mathbb{K} \mid \nu(x)\ge t\}.
\]
The differential $\omega$ is usually called the \emph{differential Swan conductor} of $\mathscr{F}$. Introduced by Kato in 1989 for characters of degree one, this invariant has since been thoroughly studied by various authors, including Deligne, Laumon, Saito, and Abbes. They have approached the subject from different perspectives and obtained numerous applications in a wide range of contexts (see \cite{MR1465067}, \cite{MR629128}, \cite{MR1925338}, \cite{MR3484149}, \cite{MR3935906}, \cite{MR3714509}, \cite{MR4651011}). For instance, formula~\eqref{eqnGOS} extends to higher dimensions \cite{MR2415398} and gives rise to a vanishing cycle formula \cite{MR904945}, a valuable tool for detecting good reduction.

The primary case of interest for the lifting problem is when $U$ is an open dense subset of the closed unit disc $\mathcal{D}=\spec R\{X\}$, where $R$ is a complete discrete valuation ring with residue field $\kappa$. Let $\mathbb{K}$ be the function field of $\mathcal{D}$; its residue field $\kappa(x)$ is imperfect.

\subsection{The equal-characteristic case}
\label{secrswequal}

When $\mathbb{K}$ has equal characteristic $p>0$, the study of
$\mathscr{F} \in \mathrm{H}^1(\mathbb{K}, \mathbb{Q}/\mathbb{Z})$
often suffices to understand many properties of the Swan conductor in this setting. Recall that
\[
\mathrm{H}^1(\mathbb{K}, \mathbb{Z}/p^s)
\cong
W_s(\mathbb{K}) \big/ \wp\!\left(W_s(\mathbb{K})\right)
\]
(see \S\ref{secasw}). In this context, we have the \emph{refined Swan homomorphism}
\cite{MR1465067}:
\begin{equation}
\label{eqnswanASW}
\begin{split}
\mathrm{Rsw}: \mathrm{Fil}_n W_s(\mathbb{K})
&\longrightarrow
\mathrm{Fil}_n \Omega^1_{\mathbb{K}}
\big/
\mathrm{Fil}_{\lfloor n/p \rfloor} \Omega^1_{\mathbb{K}}, \\
(a_0, \ldots, a_{s-1})
&\longmapsto
\omega := \sum_{i=0}^{s-1} a_i^{p^{\,s-i-1}-1}\, d a_i,
\end{split}
\end{equation}
where $n \in \mathbb{Q}_{\ge 0}$, and $\mathrm{Fil}$ on the left-hand
side denotes Brylinski's filtration of Witt vectors
\cite{MR723946}. In particular, when $(a_0, \ldots, a_{s-1})$ is
\emph{best} (see \cite[Definition~2.2]{MR3726102}), the refined Swan
conductor is determined by $\omega$.

\begin{theorem}[{\cite[Lemma 3.7]{MR991978}}]
\label{theoremrswequal}
In the above notation, suppose $(a_0, \ldots, a_{s-1})$ is best. Then
\begin{equation*}
\mathrm{rsw}(\mathscr{F})
= \omega \otimes 1=
\pi^{e \nu(\omega)}
\otimes
\overline{
\pi^{-e \nu(\omega)}
\sum_i a_i^{p^{\,s-i-1}-1} \, d a_i
}
\in
\mathfrak{p}_{\mathbb{K}}^{-\mathrm{sw}(\mathscr{F})}
\otimes
\Omega^1_{\kappa},
\end{equation*}
where $\mathrm{sw}(\mathscr{F}) = -\nu(\omega)$ is the Swan conductor.
\end{theorem}

This description makes it possible to compute and deduce numerous
properties of $\mathrm{sw}(\mathscr{F})$; see
\cite{MR3726102}, \cite{MR3665400}, \cite{MR2567506}, \cite{MR3741881}.

A length-one Witt vector $(a_0) \in W_1(\mathbb{K})$ is best if $\nu(a_0) < 0$
and the reduction of $x_0 := \pi^{-e \nu(a_0)} a_0$ is not a $p$-power.

\begin{corollary}
\label{corrswequals1}
In the above notation, suppose
$\mathscr{F} \in \mathrm{H}^1(\mathbb{K}, \mathbb{Z}/p)$
is represented by a best length-one Witt vector
$(a_0) \in \mathbb{K}$. Then
\begin{equation}
\mathrm{rsw}(\mathscr{F})
= da_0 \otimes 1=
\pi^{e \nu(a_0)} \otimes d\overline{x}_0
\in
\mathfrak{p}_{\mathbb{K}}^{-\mathrm{sw}(\mathscr{F})}
\otimes
\Omega^1_{\kappa},
\end{equation}
where $\mathrm{sw}(\mathscr{F}) = -\nu(a_0)$.
\end{corollary}

\begin{proof}
    Apply Theorem~\ref{theoremrswequal} with $s=1$ and note that for a best length-one Witt vector we have $\omega=da_0$, $\nu(\omega)=\nu(a_0)$, and $d\overline{x}_0=\overline{\pi^{-e\nu(a_0)} d a_0}$.
\end{proof}

\subsection{The mixed-characteristic case}
\label{secrswmixed}

There is no known algorithm to compute the refined Swan conductor of a cyclic character in mixed characteristic $(0,p)$ using its Kummer class like Theorem \ref{theoremrswequal}. The only known case is for
\[
\mathscr{F} \in \mathrm{H}^1(\mathbb{K}, \mathbb{Z}/p)
\cong \mathbb{K}^{\times}/(\mathbb{K}^{\times})^p.
\]
Under certain mild assumptions, we may assume that the extension
associated with $\mathscr{F}$ is \emph{fierce} (see \cite[\S 2.3]{MR3167623}).
We may then further assume that the corresponding Kummer element
$u \in \mathbb{K}^{\times}$ is of one of the following forms
\cite[Proposition~5.2]{MR3167623}:
\begin{enumerate}
\item \label{case1red}
$u \in \mathcal{O}_{\mathbb{K}}^{\times}$ and
$\overline{u} \notin \kappa^p$.
\item \label{case2red}
$u = 1 + \alpha^p w$ with $\alpha \in \mathbb{K}$ satisfying
$0 < \nu(\alpha) < \tfrac{1}{p-1}$, and
$w \in \mathcal{O}_{\mathbb{K}}$ such that
$\overline{w} \notin \kappa^p$.
\end{enumerate}
If $u$ is of one of the above forms, we say that it is \emph{reduced}.

\begin{proposition}[{\cite[Proposition~5.3]{MR3167623}}]
\label{proprsworderpwewers}
Suppose
$\mathscr{F} = \mathfrak{K}_p(u) \in \mathrm{H}^1(\mathbb{K}, \mathbb{Z}/p)$
where $u$ is reduced. Then the refined Swan conductor
$\mathrm{rsw}(\mathscr{F})$ is given as follows.
\begin{enumerate}[label=\arabic*)]
\item \label{rsworderpcase1}
If $\overline{u} \notin \kappa^p$
\emph{(Case~\ref{case1red})}, then
\begin{equation*}
\mathrm{rsw}(\mathscr{F})
=
\lambda^{-p}
\otimes
\frac{d\overline{u}}{\overline{u}}.
\end{equation*}
\item \label{rsworderpcase2}
If $u = 1 + \alpha^p w$
\emph{(Case~\ref{case2red})}, then
\begin{equation*}
\mathrm{rsw}(\mathscr{F})
=
\alpha^p \lambda^{-p}
\otimes
d\overline{w}.
\end{equation*}
\end{enumerate}
\end{proposition}

The following result recovers Corollary~\ref{corrswequals1}
in Case~\ref{case2red}.

\begin{proposition}
\label{proprswmixeds1}
Suppose $\mathscr{F} = \mathfrak{K}_p(u) \in \mathrm{H}^1(\mathbb{K}, \mathbb{Z}/p)$, where $u$ is reduced as in Case~\ref{case2red}. Suppose moreover that $X_0 \in \mathbb{K}$ is the Witt vector coordinate associated with $u$. Set $x_0 := X_0\,\pi^{-e\nu(X_0)}$. Then
\begin{equation*}
\mathrm{rsw}(\mathscr{F}) = -dX_0 \otimes 1 = \pi^{e\nu(X_0)} \otimes \bigl(-d\overline{x}_0\bigr).
\end{equation*}
\end{proposition}

\begin{proof}
We are in the situation of Case~\ref{case2red} of reduced form, so 
$u = 1 + \alpha^p w$. The first Witt vector coordinate is given by
\[
  X_0 = \frac{1}{p\lambda}\ln(1+\alpha^p w).
\]

We expand the terms on the right-hand side to isolate the leading term.
First, the expansion of the logarithm gives
\[
  \ln(1+\alpha^p w) = \alpha^p w + h_1,
\qquad
  \nu(h_1) \ge 2p\nu(\alpha).
\]
Second, consider the expansion $(1+\lambda)^p = 1$, which gives $p\lambda + \lambda^p = -\sum_{k=2}^{p-1} \binom{p}{k} \lambda^k$. 
If $p=2$, then this identity simply reads $2\lambda+\lambda^2=0$, i.e.\ $p\lambda=-\lambda^p$.
Assume now that $p>2$. The term with the lowest valuation on the right is $\binom{p}{2}\lambda^2$, which has valuation $1 + \frac{2}{p-1} = \frac{p+1}{p-1}$. 
Thus, in all cases we have the relation
\[
p\lambda \equiv -\lambda^p \pmod{\mathfrak{p}_{\mathbb{K}}^{\frac{p+1}{p-1}}}.
\]
This implies
\begin{equation}
    \label{eqnlambdap}
    \frac{1}{p\lambda} = -\lambda^{-p} + \delta,
\end{equation}
where the error $\delta$ has valuation strictly larger than $\nu(\lambda^{-p})$.

Multiplying these expansions, we find
\[
\begin{aligned}
  X_0
  &= (-\lambda^{-p} + \delta)(\alpha^p w + h_1)  \\
  &= -\lambda^{-p}\alpha^p w 
     + \left( -\lambda^{-p} h_1 \,+\, \delta\alpha^p w \,+\, \delta h_1 \right).
\end{aligned}
\]

Let $C_0 := -\alpha^p \lambda^{-p}$. The main term $C_0 w$ has valuation
$p\nu(\alpha) - \frac{p}{p-1}$. Since $p\nu(\alpha)>0$, one checks easily that every
term in the parenthesis has valuation strictly larger than $\nu(C_0 w)$.
Thus, we have the key valuation inequality
\[
  \nu(X_0 - C_0 w) > \nu(C_0 w) = \nu(X_0).
\]

Now define the normalized unit
\[
  x_0 := X_0 \pi^{-e\nu(X_0)}.
\]
The valuation inequality above implies that the difference $X_0 - C_0 w$ has strictly larger valuation than $X_0$ itself. After multiplying by $\pi^{-e\nu(X_0)}$, we obtain an element of positive valuation, hence it vanishes in the residue field. Therefore, in $\kappa$ we obtain the clean reduction
\[
\overline{x}_0 = \overline{C}\,\overline{w},
\qquad
C := C_0\,\pi^{-e\nu(X_0)} \in \mathcal{O}_{\mathbb{K}}^{\times},
\]
where $C$ is a unit since $\overline{w}\neq 0$.
Since $\overline{C}$ is a constant scalar in $\kappa^\times$, taking differentials gives the relation
\[
d\overline{x}_0 = \overline{C}\, d\overline{w}.
\]
Substituting this into the expression for $\mathrm{rsw}(\mathscr{F})$ in the statement:
\[
\begin{aligned}
  \mathrm{rsw}(\mathscr{F})
  &= \pi^{e\nu(X_0)} \otimes- d\overline{x}_0 \\
  &= -\pi^{e\nu(X_0)} \otimes \bigl( \overline{C}\, d\overline{w} \bigr) \\
  &= -C_0 \otimes d\overline{w}
   \,=\, \alpha^p \lambda^{-p} \otimes d\overline{w}.
\end{aligned}
\]
This matches the formula in Proposition~\ref{proprsworderpwewers}~\ref{rsworderpcase2}.
\end{proof}

\begin{remark}
\label{remarkorderpcase1}
In Case~\ref{case1red}, we cannot compute the Witt vector $(X_0)$ associated with the Kummer class $u$, since
\[
X_0=\frac{1}{p\lambda}\ln(u),
\]
and $\ln(u)$ is not defined when $u\in\mathcal{O}_{\mathbb{K}}^{\times}$ with $\overline{u}\notin\kappa^{p}$. Nevertheless, we may differentiate formally, writing $d\!\ln(u)=\tfrac{du}{u}$. In the filtered module $\mathfrak{p}_{\mathbb{K}}^{-\frac{p}{p-1}}\otimes\Omega^1_{\kappa}$, we obtain
\[
-dX_0\otimes1=-\frac{1}{p\lambda}\frac{du}{u}\otimes1.
\]

The relation \eqref{eqnlambdap} implies
\[
\frac{1}{p\lambda}\equiv-\lambda^{-p}\pmod{\mathfrak{p}_{\mathbb{K}}^{\frac{p+1}{p-1}}}.
\]
Substituting yields the reduction
\[
-dX_0\otimes1=\lambda^{-p}\otimes\frac{d\overline{u}}{\overline{u}},
\]
which agrees with Proposition~\ref{proprsworderpwewers} \ref{rsworderpcase1}.
\end{remark}

Taken together, these computations indicate that the description provided by Theorem~\ref{theoremrswequal} should admit a mixed-characteristic analogue, with the Witt vector $-(X_{0},\ldots,X_{s-1})$ playing the role of $(x_{0},\ldots,x_{s-1})$. Developing such an extension would offer a unified perspective on refined Swan conductors across characteristics. This direction is the subject of our ongoing work.

\bibliographystyle{plain}
\bibliography{sekiguchisuwa}
\end{document}